\documentclass[reqno]{amsart}
\usepackage{setspace,tikz,xcolor,mathrsfs,listings,multicol}
\usepackage{rotating}
\usepackage[vcentermath]{youngtab}
\usepackage{fullpage}
\usepackage{enumerate}
\usepackage{booktabs}
\usepackage{cite}
\usepackage[all,cmtip]{xy}
\usetikzlibrary{arrows,matrix}
\tikzset{tab/.style={matrix of math nodes,column sep=-.35, row sep=-.35,text height=7pt,text width=7pt,align=center,inner sep=2,font=\footnotesize}}
\usepackage[colorlinks=true, pdfstartview=FitV, linkcolor=blue, citecolor=blue, urlcolor=blue]{hyperref}

\newcommand{\g}{\mathfrak{g}}
\newcommand{\HH}{\mathcal{H}}
\newcommand{\clfw}{\overline{\Lambda}} % classical fundamental weights
\newcommand{\clsr}{\overline{\alpha}} % classical simple roots
\newcommand{\clW}{\overline{W}} % classical Weyl group
\newcommand{\inner}[2]{\left\langle #1, #2 \right\rangle}
\newcommand{\iso}{\cong}

\newcommand{\column}[1]{[#1]}

\newcommand{\qbinom}[3]{\genfrac{[}{]}{0pt}{}{#1}{#2}_{#3}}
\newcommand{\absval}[1]{\left\lvert #1 \right\rvert}
\newcommand{\virtual}[1]{\widehat{#1}}
\newcommand{\folding}{\searrow}
\newcommand{\KSS}[1]{\widetilde{#1}} % Denote the KSS-type bijection/map
\newcommand{\lusztig}{\star}

\DeclareMathOperator{\RC}{RC} % rigged configurations
\DeclareMathOperator{\wt}{wt} % weight
\DeclareMathOperator{\ls}{ls} % left split
\DeclareMathOperator{\rs}{rs} % right split
\DeclareMathOperator{\lb}{lb} % left box
\DeclareMathOperator{\rb}{rb} % right box
\DeclareMathOperator{\cc}{cc} % cocharge
\DeclareMathOperator{\fillmap}{fill} % the filling map
\DeclareMathOperator{\id}{id} % identity
\newcommand{\hwRC}{\RC^{HW}} % highest weight rigged configurations

\newcommand{\ZZ}{\mathbb{Z}}

\newcommand{\bon}{\overline{1}}
\newcommand{\btw}{\overline{2}}
\newcommand{\bth}{\overline{3}}
\newcommand{\bfo}{\overline{4}}
\newcommand{\bfive}{\overline{5}}
\newcommand{\bsix}{\overline{6}}
\newcommand{\bseven}{\overline{7}}
\newcommand{\beight}{\overline{8}}

\newcommand{\bn}{\overline{n}}
\newcommand{\ellbar}{\overline{\ell}}

\newcommand{\tnu}{\widetilde{\nu}}
\newcommand{\tJ}{\widetilde{J}}

\definecolor{darkred}{rgb}{0.7,0,0} % darkred color
\newcommand{\defn}[1]{{\color{darkred}\emph{#1}}} % emphasis of a definition

\definecolor{UQgold}{RGB}{196, 158, 54} % UQ gold
\definecolor{UQpurple}{RGB}{73, 7, 94} % UQ purple
\definecolor{UMNgold}{RGB}{255,200,46} % UMN gold
\definecolor{UMNmaroon}{RGB}{106,0,50} % UMN maroon

\usepackage{listings}
\lstdefinelanguage{Sage}[]{Python}
{morekeywords={False,sage,True},sensitive=true}
\definecolor{dblackcolor}{rgb}{0.0,0.0,0.0}
\definecolor{dbluecolor}{rgb}{0.01,0.02,0.7}
\definecolor{dgreencolor}{rgb}{0.2,0.4,0.0}
\definecolor{dgraycolor}{rgb}{0.30,0.3,0.30}

\usepackage{xparse}

\makeatletter
\protected\def\specialmergetwolists{%
  \begingroup
  \@ifstar{\def\cnta{1}\@specialmergetwolists}
    {\def\cnta{0}\@specialmergetwolists}%
}
\def\@specialmergetwolists#1#2#3#4{%
  \def\tempa##1##2{%
    \edef##2{%
      \ifnum\cnta=\@ne\else\expandafter\@firstoftwo\fi
      \unexpanded\expandafter{##1}%
    }%
  }%
  \tempa{#2}\tempb\tempa{#3}\tempa
  \def\cnta{0}\def#4{}%
  \foreach \x in \tempb{%
    \xdef\cnta{\the\numexpr\cnta+1}%
    \gdef\cntb{0}%
    \foreach \y in \tempa{%
      \xdef\cntb{\the\numexpr\cntb+1}%
      \ifnum\cntb=\cnta\relax
        \xdef#4{#4\ifx#4\empty\else,\fi\x#1\y}%
        \breakforeach
      \fi
    }%
  }%
  \endgroup
}
\makeatother

\usepackage{xparse}
\DeclareDocumentCommand\rpp{ m m g }{
	\foreach \x [count=\s from 1] in {#1}{
	        {\ifnum\s=1
	                \draw (0,-\s)--(\x,-\s);
	                \fi}
	   \draw (0,-\s-1) to (\x,-\s-1);
	   \foreach \y in {0, ..., \x} {\draw (\y,-\s)--(\y,-\s-1);}
	}
	\specialmergetwolists{/}{#1}{#2}\ziplist
	\foreach \x/\y [count=\yi from 1] in \ziplist{
	    \node[anchor=west,font=\small] at (\x,-\yi - .5) {$\y$};
	}
	\IfValueT {#3}
	{\foreach \z [count=\zi from 1] in {#3} {\node[anchor=east,font=\small] at (0,-\zi - .5) {$\z$};}}
	{}
}

\theoremstyle{plain}
\newtheorem{theorem}{Theorem}[section]
\newtheorem{lemma}[theorem]{Lemma}
\newtheorem{conjecture}[theorem]{Conjecture}
\newtheorem{proposition}[theorem]{Proposition}
\newtheorem{corollary}[theorem]{Corollary}
\theoremstyle{definition}
\newtheorem{definition}[theorem]{Definition}
\newtheorem{example}[theorem]{Example}
\newtheorem{remark}[theorem]{Remark}
\newtheorem{problem}[theorem]{Problem}
\numberwithin{equation}{section}
\usepackage[colorinlistoftodos]{todonotes}

\begin{document}
\title{Uniform description of the rigged configuration bijection}

\author[T.~Scrimshaw]{Travis Scrimshaw}
\address[T. Scrimshaw]{School of Mathematics and Physics, The University of Queensland, St.\ Lucia, QLD 4072, Australia}
\email{tcscrims@gmail.com}
\urladdr{https://people.smp.uq.edu.au/TravisScrimshaw/}

\keywords{crystal, rigged configuration, Kirillov--Reshetikhin crystal, fermionic formula}
\subjclass[2010]{05E10, 17B37, 05A19, 81R50, 82B23}

\thanks{The author was partially supported by the National Science Foundation RTG grant NSF/DMS-1148634.}

\begin{abstract}
We give a uniform description of the bijection $\Phi$ from rigged configurations to tensor products of Kirillov--Reshetikhin crystals of the form $\bigotimes_{i=1}^N B^{r_i,1}$ in dual untwisted types: simply-laced types and types $A_{2n-1}^{(2)}$, $D_{n+1}^{(2)}$, $E_6^{(2)}$, and $D_4^{(3)}$. We give a uniform proof that $\Phi$ is a bijection and preserves statistics. We describe $\Phi$ uniformly using virtual crystals for all remaining types, but our proofs are type-specific. We also give a uniform proof that $\Phi$ is a bijection for $\bigotimes_{i=1}^N B^{r_i,s_i}$ when $r_i$, for all $i$, map to $0$ under an automorphism of the Dynkin diagram. Furthermore, we give a description of the Kirillov--Reshetikhin crystals $B^{r,1}$ using tableaux of a fixed height $k_r$ depending on $r$ in all affine types. Additionally, we are able to describe crystals $B^{r,s}$ using $k_r \times s$ shaped tableaux that are conjecturally the crystal basis for Kirillov--Reshetikhin modules for various nodes $r$.
\end{abstract}

\maketitle

\setcounter{tocdepth}{1} % Show sections
\tableofcontents
\pagebreak

%=====================================================================
\section{Introduction}
\label{sec:introduction}

Kashiwara began the study of crystals in the early 1990's as a method to explore the representation theory of quantum groups~\cite{K90,K91}. One particular application is the highest weight elements of a tensor product of Kirillov--Reshetikhin (KR) crystals naturally index solutions on two-dimensional solvable lattice models from using Baxter's corner transfer matrix~\cite{B89}. Kerov, Kirillov, and Reshetikhin introduced combinatorial objects called rigged configurations that naturally index solutions to the Bethe Ansatz for the isotropic Heisenberg spin model~\cite{KKR86,KR86}. Moreover, the row-to-row transfer matrices can be described by tensor product of KR crystals. This suggests a link between rigged configurations and highest weight elements of a tensor product of KR crystals.

This was formalized by Kerov, Kirillov, and Reshetikhin by constructing a bijection $\Phi$ for the tensor product $(B^{1,1})^{\otimes N}$ in type $A_n^{(1)}$ and the corresponding rigged configurations. This was extended to the general case $\bigotimes_{i=1}^N B^{r_i, s_i}$ in type $A_n^{(1)}$ in~\cite{KSS02}, and it was soon conjectured that there exists an analogous bijection in all affine types. For the remaining non-exceptional types, such a bijection for $(B^{1,1})^{\otimes N}$ was given in~\cite{OSS03} and type $E_6^{(1)}$ in~\cite{OS12}. Many other special cases are also known: the $\bigotimes_{i=1}^N B^{1, s_i}$ case for non-exceptional types~\cite{OSS03III,SS2006}; the $\bigotimes_{i=1}^N B^{r_i, 1}$ case for types $D_{n+1}^{(2)}$, $A_{2n}^{(2)}$, and $C_n^{(1)}$~\cite{OSS03II}, and type $D_n^{(1)}$~\cite{S05}; the case $B^{r,s}$ for type $D_n^{(1)}$~\cite{OSS13} and other non-exceptional types~\cite{SchillingS15}; and both types of tensor products are known for type $D_4^{(3)}$~\cite{Scrimshaw15}. Recently, the general case for type $D_n^{(1)}$ was proven~\cite{OSSS16}, followed soon thereafter for all non-exceptional types~\cite{OSS17}. Additionally, the bijection $\Phi$ was extended to a crystal isomorphism for the full crystal in type $A_n^{(1)}$ in~\cite{DS06,SW10} and a classical crystal isomorphism for type $D_n^{(1)}$ in~\cite{Sakamoto14} and $A_{2n-1}^{(2)}$ in~\cite{SchillingS15}.

Despite being defined recursively, obfuscating many of its properties, the bijection $\Phi$ has many remarkable (conjectural) attributes. There is a natural statistic defined on tensor products of KR crystals called energy that arose from the related statistical mechanics, but energy is an algebraic statistic whose computation requires using the very intricate combinatorial $R$-matrix. On the rigged configuration side, there is a combinatorial statistic called cocharge, which also comes from the related physics, and $\Phi$ sends energy to cocharge (with interchanging riggings and coriggings). This gives a combinatorial proof the $X = M$ conjecture of~\cite{HKOTY99,HKOTT02}. We recall that the $X$ side comes from the sum over the classically highest weight elements of tensor products of KR crystals and is related to the one-point function of 2D lattice models. Additionally, the $M$ side is summed over highest weight rigged configurations and is related to solutions to the Bethe equation of the Heisenberg spin chain. Moreover, the combinatorial $R$-matrix gets sent to the identity map under $\Phi$.

Because of these properties, rigged configurations in type $A_n^{(1)}$ describe the action-angle variables of box-ball systems~\cite{KOSTY06}, which is an ultradiscrete version of the Korteweg-de Vries (KdV) equation. More specifically, the partition $\nu^{(1)}$ describes the sizes of the solitons when there is no interaction~\cite{KOSTY06,Takagi05}. A tropicalization of a ratio of (cylindric) loop Schur functions is conjectured to describe $\Phi$ for box-ball systems~\cite{LPS15,Scrimshaw17}, and $\Phi^{-1}$ can be described using the $\tau$ function from the Kadomtsev--Petviashvili (KP) heirarchy~\cite{KSY07}. Generalizations of box-ball systems, soliton cellular automata~\cite{Mahathir12,HKOTY02,HKT00,MOW12,Yamada04,Yamada07}, are also believed to have deep connections with rigged configurations. In type $A_n^{(1)}$, the state energy was related to rigged configurations~\cite{Sakamoto09}.

There are many properties of rigged configurations that are known to be uniform. A crystal structure on rigged configurations was first given for simply-laced types in~\cite{S06}, which was then extended to a classical crystal structure for $U_q'(\g)$-crystals for affine types~\cite{SchillingS15} and highest weight crystals and $B(\infty)$ for general Kac--Moody algebras in~\cite{SalisburyS15,SalisburyS15II}. Furthermore, the $\ast$-involution on $B(\infty)$~\cite{K93,Lusztig90} is the map that replaces all riggings with their respective coriggings~\cite{SalisburyS16II}.
In~\cite{SalisburyS16}, the bijection $\Phi$ was also extended (uniformly) to describe a bijection between the rigged configurations and marginally large tableaux~\cite{Cliff98,HL08} for $B(\infty)$.

Similarly, there are also uniform descriptions of KR crystals of the form $B = \bigotimes_{i=1}^N B^{r_i,1}$ using the alcove path model (up to non-dual Demazure arrows)~\cite{LL15} and quantum and projected level-zero LS paths~\cite{LNSSS14,LNSSS16,LNSSS14II,NS06II,NS08II,NS08}. This is based upon the work of Kashiwara~\cite{Kashiwara02}, where $B$ is a crystal basis of the tensor product of the corresponding KR modules and is constructed by projecting the crystal basis of a level-zero extremal weight module. A uniform model of extremal level-zero crystals using Nakajima monomials was given in~\cite{HN06}, but the projection onto $B$ was done type-by-type.
The connection of (resp.\ Demazure) characters of $B$ with (resp.\ non-symmetric) Macdonald polynomials was given in~\cite{LNSSS14,LNSSS14II} (resp.~\cite{LNSSS15}).

KR crystals also have a number of other additional properties. Their characters (resp.\ $q$-characters in the sense of~\cite{FR99}) give solutions to Q-systems (resp.\ T-systems)~\cite{Hernandez06,Hernandez10,Nakajima01,Nakajima03,Nakajima10}. The existence and combinatorial structure of $B^{r,s}$ was given for non-exceptional types in~\cite{FOS09,Okado07,OS08} and a few other special cases~\cite{JS10,KMOY07,Yamane98}. Existence for types $G_2^{(1)}$ and $D_4^{(3)}$ was recently proven in~\cite{Naoi17}. KR crystals are conjectured to generally be perfect, which is known for non-exceptional types~\cite{FOS10} and some other special cases~\cite{KMOY07,Yamane98}.

While many special cases of the conjectured bijection $\Phi$ are known (as mentioned above), the description of $\Phi$ is given in a type-by-type fashion, meaning that there is no natural extension to the other exceptional types. The original goal of this paper was to extend $\Phi$ to $\bigotimes_{i=1}^N B^{r_i, 1}$ for type $E_{6,7,8}^{(1)}$ by using the crystal graph, which was first explicitly used by Okado--Sano~\cite{OS12} for $\bigotimes_{i=1}^N B^{1,1}$ in type $E_6^{(1)}$. However, it soon became apparent that our description of $\Phi$ could be given uniformly for dual untwisted types, and moreover, the proofs given here are uniform. Using this, we are able to prove a number of special cases of the $X = M$ conjecture in all exceptional types, where there has otherwise been very little progress~\cite{JS10,KMOY07,OS12,Scrimshaw15,Yamane98}.

Explicitly, the core of our main result is a description of $\Phi$ when the basic map $\delta$ removes the left-most factor $B^{r,1}$, where $\clfw_r$ is either a minuscule weight (Section~\ref{sec:minuscule}, Lemma~\ref{lemma:minuscule_bijections}) or the highest (short) root\footnote{Our results also include type $A_n^{(1)}$, where we instead have $B^{1,1} \otimes B^{n,1}$ as the atomic object.} (\textit{i.e.}, it is the perfect crystal of~\cite{BFKL06} or $B(\clfw_r)$ is the (``little'') adjoint representation) (Section~\ref{sec:adjoint}, Lemma~\ref{lemma:adjoint_bijections}). We then extend the bijection to $B = \bigotimes_{i=1}^N B^{r_i,1}$ (Section~\ref{sec:box_map}, Proposition~\ref{prop:lb_bijection_reduction}). As stated above, the description and proof of this is uniform for all dual untwisted types. For the remaining types, we give a uniform description using virtual crystals (Section~\ref{sec:untwisted}), and while our proof is essentially uniform, it does contain some type-specific arguments. However, the last part of our main results are that we give a uniform proof that the virtualization map commutes with the bijection $\Phi$ (Theorem~\ref{thm:virtual_bijection}).

We show that these descriptions of $\delta$ are equivalent to those described in~\cite{Mahathir,KSS02,OSS03,S05,Scrimshaw15} (in particular, proving the conjectural description of $\Phi$ in~\cite{Mahathir}) (Section~\ref{sec:equivalence}). As a secondary result, we provide further evidence of the conjecture that KR crystals $B^{r,s}$ in the exceptional types correspond to crystal bases of KR modules and of the $X = M$ conjecture by showing the fermionic formula agrees with the conjectured decompositions of~\cite{HKOTY99,HKOTT02}. We also describe the so-called Kirillov--Reshetikhin (KR) tableaux for $B^{r,1}$ in types $E_{6,7,8}^{(1)}$, $E_6^{(2)}$, $F_4^{(1)}$, and $G_2^{(1)}$ (Section~\ref{sec:filling_map}). For certain $r$, we describe the KR tableaux for $B^{r,s}$ and show that $\Phi$ gives a bijection for the single tensor factor.

We are further able to extend our bijection for $\bigotimes_{i=1}^N B^{r_i,s_i}$ when $\clfw_{r_i}$ is a minuscule weight by using the tableaux description given in~\cite{JS10} (Theorem~\ref{thm:general_minuscule_bijection}). Specifically, the tableaux can be thought of as single rows that are weakly increasing with entries in $B(\clfw_{r_i})$, which is naturally considered as a poset. Moreover, the proofs that we give are also uniform. This is the generalization of the results of~\cite{SS2006}.

Our results are evidence that there should be a natural bijection between rigged configurations and the aforementioned models for KR crystals. Additionally, it also suggests that there should be a uniform description of the $U_q'(\g)$-crystal structure on rigged configurations by considering the Demazure subcrystal of $B(\ell \Lambda_0)$ following~\cite{FSS07,ST12}. Furthermore, our results and our proof techniques are further evidence that the map $\eta$ that replaces riggings with coriggings in our setting of rigged configurations, which is key in the proof that $\Phi$ preserves statistics, is connected with the $\ast$-involution on $B(\infty)$. Our results also give a uniform description of the combinatorial $R$-matrix in the cases we consider (extend Remark~\ref{rem:universal_R_construction} to all of our results), for which a uniform description was given on the alcove path model in~\cite{LL16}, but our proof is type-specific.

Our results also give further evidence that rigged configurations are intimately connected with the Weyl chamber geometry. Indeed, as rigged configurations are well-behaved under virtualization, the results of~\cite{PS15} gives the first evidence. Yet, it is the fact that our results are given for the types where the fundamental alcove is translated by precisely $\alpha_a$ (there are some slight modifications needed for type $A_{2n}^{(2)}$) to another alcove is further evidence. This is additionally emphasized with our result showing $\Phi$ intertwines with the virtualization map, extending results of~\cite{OSS03III,OSS03II,SchillingS15}. Additionally, the related results of~\cite{OSS03} for the untwisted non-simply-laced types appears to be related to our descriptions when the rigged partition $(\nu, J)^{(a)}$ is scaled by the coefficient of $\alpha_a$ required to translate the fundamental alcove. Making this explicit would lead to a completely uniform description of $\Phi$ and more strongly link rigged configurations to the underlying geometry.

\subsection*{Summary of new results}

Recall that we consider the case $\bigotimes_{i=1}^N B^{r_i,1}$ and the case $\bigotimes_{i=1}^N B^{r_i,s_i}$, where $r_i$ is a minuscule node for all $i$.
Our results for the rigged configuration bijection and a combinatorial proof of the $X = M$ conjecture are new for all exceptional types with the exception of $(B^{r,1})^{\otimes N}$ for $r =1,6$ (the $r = 6$ is implicit in~\cite{OS12} by the diagram symmetry) and $\bigotimes_{i=1}^N B^{1,s_i}$ and $\bigotimes_{i=1}^N B^{r_i,1}$ in type $D_4^{(3)}$. Furthermore, our description of the bijection for single columns in type $A_n^{(1)}$ and spin columns in type $D_n^{(1)}$ is new, which significantly reduces the number of steps needed to compute the bijection $\Phi$. In addition, we note that our proofs are now done almost uniformly.

\subsection*{Organization}

This paper is organized as follows.
In Section~\ref{sec:background}, we describe the necessary background on crystals, KR crystals, rigged configurations, and the bijection $\Phi$.
In Section~\ref{sec:minuscule}, we describe the map $\delta$ for minuscule nodes for dual untwisted types.
In Section~\ref{sec:adjoint}, we describe the map $\delta$ for the adjoint node for dual untwisted types.
In Section~\ref{sec:box_map}, we extend the left-box map for dual untwisted types.
In Section~\ref{sec:untwisted}, we show that the map $\delta$ for untwisted types is well-defined by using the virtualization map to the corresponding dual type.
In Section~\ref{sec:results}, we give our results and proofs.
In Section~\ref{sec:equivalence}, we show that our description of $\Phi$ is the same as the KSS bijections.
In Section~\ref{sec:filling_map}, we describe the highest weight rigged configurations and KR tableaux for $B^{r,s}$ in a number of different cases.
In Section~\ref{sec:outlook}, we give our concluding remarks.

\subsection*{Acknowledgements}

The author would like to thank Masato Okado for the reference~\cite{Mahathir} and useful discussions. The author would like to thank Anne Schilling for useful discussions. The author would like to thank Ben Salisbury for comments on an early draft of this paper. The author would like to thank the referee for their useful comments. This work benefited from computations using {\sc SageMath}~\cite{sage,combinat}.

The majority of this work was done while the author was at the University of Minnesota.

%=====================================================================
\section{Background}
\label{sec:background}

In this section, we provide the necessary background.

%%%%%%%%%%
\subsection{Crystals}

\begin{figure}[t]
\begin{center}
% E6~
\begin{tikzpicture}[scale=0.5, baseline=-20]
\draw (-1,0) node[anchor=east] {$E_6^{(1)}$};
\draw (0 cm,0) -- (8 cm,0);
\draw (4 cm, 0 cm) -- +(0,2 cm);
\draw (4 cm, 2 cm) -- +(0,2 cm);
\draw[fill=white] (4 cm, 4 cm) circle (.25cm) node[right=3pt]{$0$};
\draw[fill=white] (0 cm, 0 cm) circle (.25cm) node[below=4pt]{$1$};
\draw[fill=white] (2 cm, 0 cm) circle (.25cm) node[below=4pt]{$3$};
\draw[fill=white] (4 cm, 0 cm) circle (.25cm) node[below=4pt]{$4$};
\draw[fill=white] (6 cm, 0 cm) circle (.25cm) node[below=4pt]{$5$};
\draw[fill=white] (8 cm, 0 cm) circle (.25cm) node[below=4pt]{$6$};
\draw[fill=white] (4 cm, 2 cm) circle (.25cm) node[right=3pt]{$2$};
\end{tikzpicture}
% E7~
%\begin{scope}[xshift=15cm]
\begin{tikzpicture}[scale=0.5, baseline=-20]
\draw (-1,0) node[anchor=east] {$E_7^{(1)}$};
\draw (0 cm,0) -- (12 cm,0);
\draw (6 cm, 0 cm) -- +(0,2 cm);
\draw[fill=white] (0 cm, 0 cm) circle (.25cm) node[below=4pt]{$0$};
\draw[fill=white] (2 cm, 0 cm) circle (.25cm) node[below=4pt]{$1$};
\draw[fill=white] (4 cm, 0 cm) circle (.25cm) node[below=4pt]{$3$};
\draw[fill=white] (6 cm, 0 cm) circle (.25cm) node[below=4pt]{$4$};
\draw[fill=white] (8 cm, 0 cm) circle (.25cm) node[below=4pt]{$5$};
\draw[fill=white] (10 cm, 0 cm) circle (.25cm) node[below=4pt]{$6$};
\draw[fill=white] (12 cm, 0 cm) circle (.25cm) node[below=4pt]{$7$};
\draw[fill=white] (6 cm, 2 cm) circle (.25cm) node[right=3pt]{$2$};
%\end{scope}
\end{tikzpicture}

% E8~
%\begin{scope}[xshift=5cm,yshift=-5cm]
\begin{tikzpicture}[scale=0.5]
\draw (-1,0) node[anchor=east] {$E_8^{(1)}$};
\draw (0 cm,0) -- (12 cm,0);
\draw (4 cm, 0 cm) -- +(0,2 cm);
\draw (12 cm,0) -- +(2 cm,0);
\draw[fill=white] (14 cm, 0 cm) circle (.25cm) node[below=4pt]{$0$};
\draw[fill=white] (0 cm, 0 cm) circle (.25cm) node[below=4pt]{$1$};
\draw[fill=white] (2 cm, 0 cm) circle (.25cm) node[below=4pt]{$3$};
\draw[fill=white] (4 cm, 0 cm) circle (.25cm) node[below=4pt]{$4$};
\draw[fill=white] (6 cm, 0 cm) circle (.25cm) node[below=4pt]{$5$};
\draw[fill=white] (8 cm, 0 cm) circle (.25cm) node[below=4pt]{$6$};
\draw[fill=white] (10 cm, 0 cm) circle (.25cm) node[below=4pt]{$7$};
\draw[fill=white] (12 cm, 0 cm) circle (.25cm) node[below=4pt]{$8$};
\draw[fill=white] (4 cm, 2 cm) circle (.25cm) node[right=3pt]{$2$};
%\end{scope}
\end{tikzpicture}

%F4~
\begin{tikzpicture}[scale=0.5,baseline=10]
\draw (-1,0) node[anchor=east] {$F_4^{(1)}$};
\draw (0 cm,0) -- (2 cm,0);
{
\pgftransformxshift{2 cm}
\draw (0 cm,0) -- (2 cm,0);
\draw (2 cm, 0.1 cm) -- +(2 cm,0);
\draw (2 cm, -0.1 cm) -- +(2 cm,0);
\draw (4.0 cm,0) -- +(2 cm,0);
\draw[shift={(3.2, 0)}, rotate=0] (135 : 0.45cm) -- (0,0) -- (-135 : 0.45cm);
\draw[fill=white] (0 cm, 0 cm) circle (.25cm) node[below=4pt]{$1$};
\draw[fill=white] (2 cm, 0 cm) circle (.25cm) node[below=4pt]{$2$};
\draw[fill=white] (4 cm, 0 cm) circle (.25cm) node[below=4pt]{$3$};
\draw[fill=white] (6 cm, 0 cm) circle (.25cm) node[below=4pt]{$4$};
}
\draw[fill=white] (0 cm, 0 cm) circle (.25cm) node[below=4pt]{$0$};

% E_6^(2)
\begin{scope}[xshift=13cm]
\draw (-1,0) node[anchor=east] {$E_6^{(2)}$};
\draw (0 cm,0) -- (2 cm,0);
{
\pgftransformxshift{2 cm}
\draw (0 cm,0) -- (2 cm,0);
\draw (2 cm, 0.1 cm) -- +(2 cm,0);
\draw (2 cm, -0.1 cm) -- +(2 cm,0);
\draw (4.0 cm,0) -- +(2 cm,0);
\draw[shift={(2.8, 0)}, rotate=180] (135 : 0.45cm) -- (0,0) -- (-135 : 0.45cm);
\draw[fill=white] (0 cm, 0 cm) circle (.25cm) node[below=4pt]{$1$};
\draw[fill=white] (2 cm, 0 cm) circle (.25cm) node[below=4pt]{$2$};
\draw[fill=white] (4 cm, 0 cm) circle (.25cm) node[below=4pt]{$3$};
\draw[fill=white] (6 cm, 0 cm) circle (.25cm) node[below=4pt]{$4$};
}
\draw[fill=white] (0 cm, 0 cm) circle (.25cm) node[below=4pt]{$0$};
\end{scope}
\end{tikzpicture}

% G_2^{(1)}
\begin{tikzpicture}[scale=0.5,baseline=20]
\begin{scope}
\draw (-1,0) node[anchor=east] {$G_2^{(1)}$};
\draw (2 cm,0) -- (4.0 cm,0);
\draw (0, 0.15 cm) -- +(2 cm,0);
\draw (0, -0.15 cm) -- +(2 cm,0);
\draw (0,0) -- (2 cm,0);
\draw (0, 0.15 cm) -- +(2 cm,0);
\draw (0, -0.15 cm) -- +(2 cm,0);
\draw[shift={(0.8, 0)}, rotate=180] (135 : 0.45cm) -- (0,0) -- (-135 : 0.45cm);
\draw[fill=white] (0 cm, 0 cm) circle (.25cm) node[below=4pt]{$1$};
\draw[fill=white] (2 cm, 0 cm) circle (.25cm) node[below=4pt]{$2$};
\draw[fill=white] (4 cm, 0 cm) circle (.25cm) node[below=4pt]{$0$};
\end{scope}

% D_4^{(3)}
\begin{scope}[xshift=15cm]
\draw (-1,0) node[anchor=east] {$D_4^{(3)}$};
\draw (2 cm,0) -- (4.0 cm,0);
\draw (0, 0.15 cm) -- +(2 cm,0);
\draw (0, -0.15 cm) -- +(2 cm,0);
\draw (0,0) -- (2 cm,0);
\draw (0, 0.15 cm) -- +(2 cm,0);
\draw (0, -0.15 cm) -- +(2 cm,0);
\draw[shift={(1.2, 0)}, rotate=0] (135 : 0.45cm) -- (0,0) -- (-135 : 0.45cm);
\draw[fill=white] (0 cm, 0 cm) circle (.25cm) node[below=4pt]{$2$};
\draw[fill=white] (2 cm, 0 cm) circle (.25cm) node[below=4pt]{$1$};
\draw[fill=white] (4 cm, 0 cm) circle (.25cm) node[below=4pt]{$0$};
\end{scope}
\end{tikzpicture}
\end{center}
\caption{Dynkin diagrams for the exceptional affine types.}
\label{fig:exceptional_types}
\end{figure}

Let $\g$ be an affine Kac--Moody Lie algebra with index set $I$, Cartan matrix $(A_{ab})_{a,b \in I}$, simple roots $(\alpha_a)_{a \in I}$, simple coroots $(\alpha_a^{\vee})_{a \in I}$, fundamental weights $(\Lambda_a)_{a \in I}$, weight lattice $P$, coweight lattice $P^{\vee}$, and canonical pairing $\langle\ ,\ \rangle \colon P^{\vee} \times P \to \ZZ$ given by $\inner{\alpha_a^{\vee}}{\alpha_b} = A_{ab}$.
We note that we follow the labeling given in~\cite{Bourbaki02} (see Figure~\ref{fig:exceptional_types} for the exceptional types and their labellings).
Let $\g_0$ denote the canonical simple Lie algebra given by the index set $I_0 = I \setminus \{0\}$.
Let $\clfw_a$ and $\clsr_a$ denote the natural projection of $\Lambda_a$ and $\alpha_a$, respectively, onto the weight lattice $\overline{P}$ of $\g_0$. Note $(\clsr_a)_{a \in I_0}$ are the simple roots in $\g_0$.

Let $c_a$ and $c_a^{\vee}$ denote the Kac and dual Kac labels~\cite[Table~Aff1-3]{kac90}. We define
\[
t_a := \max\left(\frac{c_a}{c_a^{\vee}}, c_0^{\vee}\right),
\hspace{40pt}
t_a^{\vee} := \max\left(\frac{c_a^{\vee}}{c_a}, c_0\right).
\]
Note that $t_a^{\vee}$ for type $\g$ equals $t_a$ for the dual type of $\g$.
Let $T_a$ denote the \defn{translation factors}, the smallest factors such that $T_a \alpha_a$ maps the fundamental polygon, the fundamental domain of the action of the root lattice or the image of the fundamental alcove under the corresponding finite Weyl group, to another polygon.
Note that $T_a = t_a$ except for type $A_{2n}^{(2)}$ (resp.~$A_{2n}^{(2)\dagger}$), where we have $T_n = \frac{1}{2}$ (resp.~$T_0 = \frac{1}{2}$) and $T_a = 1$ otherwise.
We have $T_a = 1$ for all $a \in I$ except in the cases mentioned above, $T_n = 2$ in type $B_n^{(1)}$, $T_a = 2$ for $a \neq 0,n$ in type $C_n^{(1)}$, $T_3 = T_4 = 2$ in type $F_4^{(1)}$, and $T_1 = 3$ in type $G_2^{(1)}$.
The null root is $\delta = \sum_{a \in I} c_a \alpha_a$, and the canonical central element is $c = \sum_{a \in I} c_a^{\vee} \alpha_a^{\vee}$.
The normalized (symmetric) invariant form $( \cdot | \cdot ) \colon P \times P \to \ZZ$ is defined by  $(\alpha_a | \alpha_b) = \frac{c_a^{\vee}}{c_a} A_{ab}$.
We write $a \sim b$ if $A_{ab} \neq 0$ and $a \neq b$; in other words, the nodes $a$ and $b$ are adjacent in the Dynkin diagram of $\g$.

For $\g$ not of type $A_n^{(1)}$, let $N_{\g}$ denote the unique node such that $N_{\g} \sim 0$, which we call the \defn{adjoint node}.
We say a node $r \in I_0$ is \defn{special} if there exists a diagram automorphism $\phi \colon I \to I$ such that $\phi(0) = r$. We say a node $r \in I_0$ is \defn{minuscule} if it is special and $\g$ is of dual untwisted affine type.

An \defn{abstract $U_q(\g)$-crystal} is a set $B$ with the \defn{crystal operators} $e_a, f_a \colon B \to B \sqcup \{0\}$, for $a \in I$, and \defn{weight function} $\wt \colon B \to P$ that satisfy the following conditions.
Let $\varepsilon_a, \varphi_a \colon  B \to \ZZ_{\geq 0}$ be statistics given by
\[
\varepsilon_a(b) := \max \{ k \mid e_a^k b \neq 0 \},
\qquad \qquad
\varphi_a(b) := \max \{ k \mid f_a^k b \neq 0 \}.
\]
\begin{itemize}
\item[(1)] $\varphi_a(b) = \varepsilon_a(b) + \inner{\alpha^{\vee}_a}{\wt(b)}$ for all $b \in B$ and $a \in I$.
\item[(2)] $f_a b = b'$ if and only if $b = e_a b'$ for $b, b' \in B$ and $a \in I$.
\end{itemize}
We say an element $b \in B$ is \defn{highest weight} if $e_a b = 0$ for all $a \in I$. Define
\[
\varepsilon(b) = \sum_{a \in I} \varepsilon_a(b) \Lambda_a,
\qquad\qquad
\varphi(b) = \sum_{a \in I} \varphi_a(b) \Lambda_a.
\]

\begin{remark}
The abstract crystals we consider in this paper sometimes called \defn{regular} or \defn{seminormal} in the literature.
\end{remark}

We call an abstract $U_q(\g)$-crystal $B$ a \defn{$U_q(\g)$-crystal} if $B$ is the crystal basis of some $U_q(\g)$-module. Kashiwara has shown that the irreducible highest weight module $V(\lambda)$ admits a crystal basis~\cite{K91}. We denote this crystal basis by $B(\lambda)$, and let $u_{\lambda} \in B(\lambda)$ denote the unique highest weight element and is the unique element of weight $\lambda$. Recall that $B(\lambda)$ is connected. A $U_q(\g_0)$-crystal is a \defn{minuscule representation} if the corresponding finite Weyl group $\clW$ acts transitively on $B(\clfw_r)$. In particular, the $U_q(\g_0)$-crystal $B(\clfw_r)$ is a minuscule representation if and only if $r$ is a minuscule node.

We define the \defn{tensor product} of abstract $U_q(\g)$-crystals $B_1$ and $B_2$ as the crystal $B_2 \otimes B_1$ that is the Cartesian product $B_2 \times B_1$ with the crystal structure
\begin{align*}
e_a(b_2 \otimes b_1) & = \begin{cases}
e_a b_2 \otimes b_1 & \text{if } \varepsilon_a(b_2) > \varphi_a(b_1), \\
b_2 \otimes e_a b_1 & \text{if } \varepsilon_a(b_2) \leq \varphi_a(b_1),
\end{cases}
\\ f_a(b_2 \otimes b_1) & = \begin{cases}
f_a b_2 \otimes b_1 & \text{if } \varepsilon_a(b_2) \geq \varphi_a(b_1), \\
b_2 \otimes f_a b_1 & \text{if } \varepsilon_a(b_2) < \varphi_a(b_1),
\end{cases}
\\ \varepsilon_a(b_2 \otimes b_1) & = \max(\varepsilon_a(b_1), \varepsilon_a(b_2) - \inner{\alpha^{\vee}_a}{\wt(b_1)}),
\\ \varphi_a(b_2 \otimes b_1) & = \max(\varphi_a(b_2), \varphi_a(b_1) + \inner{\alpha^{\vee}_a}{\wt(b_2)}),
\\ \wt(b_2 \otimes b_1) & = \wt(b_2) + \wt(b_1).
\end{align*}

\begin{remark}
Our tensor product convention is opposite of Kashiwara~\cite{K91}.
\end{remark}

Let $B_1$ and $B_2$ be two abstract $U_q(\g)$-crystals. A \defn{crystal morphism} $\psi \colon B_1 \to B_2$ is a map $B_1 \sqcup \{0\} \to B_2 \sqcup \{0\}$ with $\psi(0) = 0$ such that the following properties hold for all $b \in B_1$:
\begin{itemize}
\item[(1)] If $\psi(b) \in B_2$, then $\wt\bigl(\psi(b)\bigr) = \wt(b)$, $\varepsilon_a\bigl(\psi(b)\bigr) = \varepsilon_a(b)$, and $\varphi_a\bigl(\psi(b)\bigr) = \varphi_a(b)$.
\item[(2)] We have $\psi(e_a b) = e_a \psi(b)$ if $\psi(e_a b) \neq 0$ and $e_a \psi(b) \neq 0$.
\item[(3)] We have $\psi(f_a b) = f_a \psi(b)$ if $\psi(f_a b) \neq 0$ and $f_a \psi(b) \neq 0$.
\end{itemize}
An \defn{embedding} and \defn{isomorphism} is a crystal morphism such that the induced map $B_1 \sqcup \{0\} \to B_2 \sqcup \{0\}$ is an embedding or bijection respectively. A crystal morphism is \defn{strict} if it commutes with all crystal operators.

\begin{figure}
\[
\begin{tikzpicture}[>=latex,xscale=1.4,yscale=1.4]
\node (1) at (0,0) {$1$};
\node (b12) at (1,0) {$\bon2$};
\node (b23) at (2,0) {$\btw3$};
\node (b345) at (3,0) {$\bth45$};
\node (b45) at (4,1) {$\bfo5$};
\node (4b5) at (4,-1) {$4\bfive$};
\node (3b4b5) at (5,0) {$3\bfo\bfive$};
\node (2b3) at (6,0) {$2\bth$};
\node (1b2) at (7,0) {$1\btw$};
\node (b1) at (8,0) {$\bon$};
\draw[->,red] (1) -- (b12) node[midway,above] {\scriptsize $1$};
\draw[->,blue] (b12) -- (b23) node[midway,above] {\scriptsize $2$};
\draw[->,dgreencolor] (b23) -- (b345) node[midway,above] {\scriptsize $3$};
\draw[->,black] (b345) -- (b45) node[midway,above left] {\scriptsize $4$};
\draw[->,UQpurple] (b345) -- (4b5) node[midway,below left] {\scriptsize $5$};
\draw[->,black] (b45) -- (3b4b5) node[midway,above right] {\scriptsize $4$};
\draw[->,UQpurple] (4b5) -- (3b4b5) node[midway,below right] {\scriptsize $5$};
\draw[->,dgreencolor] (3b4b5) -- (2b3) node[midway,above] {\scriptsize $3$};
\draw[->,blue] (2b3) -- (1b2) node[midway,above] {\scriptsize $2$};
\draw[->,red] (1b2) -- (b1) node[midway,above] {\scriptsize $1$};
\end{tikzpicture}
\]
\caption{The crystal $B(\clfw_1)$ in type $D_5$.}
\label{fig:B1_type_D5}
\end{figure}

In type $E_n$ for $n = 6,7$, we follow~\cite{JS10} and label elements $b \in B(\clfw_n)$ (and in $B(\clfw_1)$ for type $E_6$) by $X \subset \{1, \bon, 2, \btw, \dotsc, n, \bn\}$, where we have an $a \in X$ (resp.~$\overline{a} \in X$) if and only if $\varphi_a(b) = 1$ (resp.~$\varepsilon_a(b) = 1$). To ease notation, we write $X$ as a word in the alphabet $\{1, \bon, \dotsc, n, \bn\}$. See Figure~\ref{fig:B1_type_D5} and Figure~\ref{fig:B1_type_E6} for examples. We follow the same notation for an element of any minuscule representation.

\begin{figure}
\[
\begin{tikzpicture}[>=latex,xscale=1.4,yscale=1.1]
\node (1) at (0,0) {$1$};
\node (1b3) at (1,0) {$\bon3$};
\node (3b4) at (2,0) {$\bth4$};
\node (4b25) at (3,0) {$2\bfo5$};
\node (5b26) at (4,0) {$2\bfive6$};
\node (6b2) at (5,0) {$2\bsix$};
\node (2b5) at (3,-1) {$\btw5$};
\node (25b46) at (4,-1) {$\btw4\bfive6$};
\node (26b4) at (5,-1) {$\btw4\bsix$};
\node (4b36) at (4,-2) {3$\bfo6$};
\node (46b35) at (5,-2) {$3\bfo5\bsix$};
\node (5b3) at (6,-2) {$3\bfive$};
\node (3b16) at (4,-3) {$1\bth6$};
\node (36b15) at (5,-3) {$1\bth5\bsix$};
\node (35b14) at (6,-3) {$1\bth4\bfive$};
\node (4b12) at (7,-3) {$12\bfo$};
\node (2b1) at (8,-3) {$1\btw$};
\node (1b6) at (4,-4) {$\bon6$};
\node (16b5) at (5,-4) {$\bon5\bsix$};
\node (15b4) at (6,-4) {$\bon4\bfive$};
\node (14b23) at (7,-4) {$\bon23\bfo$};
\node (12b3) at (8,-4) {$\bon\btw3$};
\node (3b2) at (7,-5) {$2\bth$};
\node (23b4) at (8,-5) {$\btw\bth4$};
\node (4b5) at (8,-6) {$\bfo5$};
\node (5b6) at (8,-7) {$\bfive6$};
\node (6b) at (8,-8) {$\bsix$};
\draw[->,red] (1) -- (1b3) node[midway,above] {\scriptsize $1$};
\draw[->,dgreencolor] (1b3) -- (3b4) node[midway,above] {\scriptsize $3$};
\draw[->,black] (3b4) -- (4b25) node[midway,above] {\scriptsize $4$};
\draw[->,UQpurple] (4b25) -- (5b26) node[midway,above] {\scriptsize $5$};
\draw[->,UQgold] (5b26) -- (6b2) node[midway,above] {\scriptsize $6$};
\draw[->,blue] (4b25) -- (2b5) node[midway,right] {\scriptsize $2$};
\draw[->,blue] (5b26) -- (25b46) node[midway,right] {\scriptsize $2$};
\draw[->,blue] (6b2) -- (26b4) node[midway,right] {\scriptsize $2$};
\draw[->,UQpurple] (2b5) -- (25b46) node[midway,above] {\scriptsize $5$};
\draw[->,UQgold] (25b46) -- (26b4) node[midway,above] {\scriptsize $6$};
\draw[->,black] (25b46) -- (4b36) node[midway,right] {\scriptsize $4$};
\draw[->,black] (26b4) -- (46b35) node[midway,right] {\scriptsize $4$};
\draw[->,UQgold] (4b36) -- (46b35) node[midway,above] {\scriptsize $6$};
\draw[->,UQpurple] (46b35) -- (5b3) node[midway,above] {\scriptsize $5$};
\draw[->,dgreencolor] (4b36) -- (3b16) node[midway,right] {\scriptsize $3$};
\draw[->,dgreencolor] (46b35) -- (36b15) node[midway,right] {\scriptsize $3$};
\draw[->,dgreencolor] (5b3) -- (35b14) node[midway,right] {\scriptsize $3$};
\draw[->,UQgold] (3b16) -- (36b15) node[midway,above] {\scriptsize $6$};
\draw[->,UQpurple] (36b15) -- (35b14) node[midway,above] {\scriptsize $5$};
\draw[->,black] (35b14) -- (4b12) node[midway,above] {\scriptsize $4$};
\draw[->,blue] (4b12) -- (2b1) node[midway,above] {\scriptsize $2$};
\draw[->,red] (3b16) -- (1b6) node[midway,right] {\scriptsize $1$};
\draw[->,red] (36b15) -- (16b5) node[midway,right] {\scriptsize $1$};
\draw[->,red] (35b14) -- (15b4) node[midway,right] {\scriptsize $1$};
\draw[->,red] (4b12) -- (14b23) node[midway,right] {\scriptsize $1$};
\draw[->,red] (2b1) -- (12b3) node[midway,right] {\scriptsize $1$};
\draw[->,UQgold] (1b6) -- (16b5) node[midway,above] {\scriptsize $6$};
\draw[->,UQpurple] (16b5) -- (15b4) node[midway,above] {\scriptsize $5$};
\draw[->,black] (15b4) -- (14b23) node[midway,above] {\scriptsize $4$};
\draw[->,blue] (14b23) -- (12b3) node[midway,above] {\scriptsize $2$};
\draw[->,black] (14b23) -- (3b2) node[midway,right] {\scriptsize $3$};
\draw[->,black] (12b3) -- (23b4) node[midway,right] {\scriptsize $3$};
\draw[->,blue] (3b2) -- (23b4) node[midway,above] {\scriptsize $2$};
\draw[->,black] (23b4) -- (4b5) node[midway,right] {\scriptsize $4$};
\draw[->,UQpurple] (4b5) -- (5b6) node[midway,right] {\scriptsize $5$};
\draw[->] (5b6) -- (6b) node[midway,right] {\scriptsize $6$};
\end{tikzpicture}
\]
\caption{The crystal $B(\Lambda_1)$ in type $E_6$.}
\label{fig:B1_type_E6}
\end{figure}

%%%%%%%%%%
\subsection{Kirillov--Reshetikhin crystals}

Let $U_q'(\g) = U_q([\g, \g])$ denote the quantum group of the derived subalgebra of $\g$. Denote the corresponding weight lattice by $P' = P / \ZZ \delta$, and therefore, there is a linear dependence relation on the simple roots in $P'$. As we will not be considering $U_q(\g)$-crystals in this paper, we will abuse notation and also denote the $U_q'(\g)$-weight lattice by $P$. If $B$ is a $U_q'(\g)$-crystal, then we say $b \in B$ is \defn{classically highest weight} if $e_a b = 0$ for all $a \in I_0$.

\defn{Kirillov--Reshetikhin (KR) crystals} are the conjectural crystal bases corresponding to an important class of finite-dimensional irreducible $U_q'(\g)$-modules known as \defn{Kirillov--Reshetikhin (KR) modules}~\cite{HKOTY99,HKOTT02}. We denote a KR module and crystal by $W^{r,s}$ and $B^{r, s}$, respectively, where $r \in I_0$ and $s \in \ZZ_{>0}$. KR modules are classified by their Drinfel'd polynomials, and $W^{r,s}$ corresponds to the minimal affinization of $B(s\clfw_r)$~\cite{Chari95,Chari01,CP95II,CP95,CP96,CP96II,CP98}. In~\cite{OS08}, it was shown that KR modules in all non-exceptional types admit crystal bases whose combinatorial structure is given in~\cite{FOS09}. For the exceptional types, this has been done in a few special cases~\cite{JS10,KKMMNN92,KMOY07,Yamane98}. Recently, existence was established in general for types $G_2^{(1)}$ and $D_4^{(3)}$ in~\cite{Naoi17}. Moreover, a uniform model for $B^{r,1}$ was given using quantum and projected level-zero LS paths~\cite{LNSSS14,LNSSS16,LNSSS14II,NS06II,NS08II,NS08}.

We note that there is a unique element $u(B^{r,s}) \in B^{r,s}$, called the \defn{maximal element}, which is characterized by being the unique element of classical weight $s \clfw_r$ (it is also classically highest weight, so $\varepsilon_a\bigl( u(B^{r,s}) \bigr) = 0$ and $\varphi_a\bigl( u(B^{r,s}) \bigr) = \delta_{ar}s$ for all $a \in I_0$). Furthermore, it is known that tensor products of KR crystals are connected~\cite{FSS07,Okado13}, and it is known that the maximal vector of $B \otimes B'$ is given by the tensor product of maximal elements $u \otimes u'$. We define the unique $U_q'(\g)$-crystal morphism $R \colon B \otimes B' \to B' \otimes B$, called the \defn{combinatorial $R$-matrix}, by $R(u \otimes u') = u' \otimes u$, where $u$ and $u'$ are the maximal weight elements of $B$ and $B'$ respectively.

We say a KR crystal $B^{r,1}$ is \defn{minuscule} if $r$ is a minuscule node. We note that $B^{r,s} \iso B(s\clfw_r)$ as $U_q(\g_0)$-crystals if $r$ is a special node.

%%%%%%%%%%
\subsection{Virtual crystals}

We recall the definition of a virtual crystal from~\cite{OSS03III,OSS03II}. Let $\phi \colon \virtual{\Gamma} \folding \Gamma$ denote a folding of the Dynkin diagram $\virtual{\Gamma}$ of $\virtual{\g}$ onto the Dynkin diagram $\Gamma$ of  $\g$ that arises from the natural embeddings~\cite{JM85}
\begin{equation}
\begin{aligned}
C_n^{(1)}, A_{2n}^{(2)}, A_{2n}^{(2)\dagger}, D_{n+1}^{(2)} & \lhook\joinrel\longrightarrow A_{2n-1}^{(1)},
\\ B_n^{(1)}, A_{2n-1}^{(2)} & \lhook\joinrel\longrightarrow D_{n+1}^{(1)},
\\ F_4^{(1)}, E_6^{(2)} & \lhook\joinrel\longrightarrow E_6^{(1)},
\\ G_2^{(1)}, D_4^{(3)} &\lhook\joinrel\longrightarrow D_4^{(4)}.
\end{aligned}
\end{equation}
For ease of notation, if $X$ is an object for type $\g$, we denote the corresponding object for type $\virtual{\g}$ by $\virtual{X}$. By abuse of notation, let $\phi \colon \virtual{I} \folding I$ also denote the corresponding map on the index sets. We define the \defn{scaling factors} $\gamma = (\gamma_a)_{a \in I}$ by
\[
\gamma_a = \frac{\max_a T_a}{T_a}.
%\begin{cases}
%t_a^{\vee} & \text{if } \g = A_{2n}^{(2)}, A_{2n}^{(2)\dagger},
%\\ \frac{\max_a t_a}{t_a} & \text{otherwise}.
%\end{cases}
\]
Note that if $\lvert\phi^{-1}(a)\rvert = 1$, then $\gamma_a = 1$. See Table~\ref{table:scaling_factors}.

\begin{table}
\[
\begin{array}{cc}
\toprule
\g & (\gamma_a)_{a \in I}
\\ \midrule
\text{dual untwisted type} &
  \gamma_a = 1 \quad (a \in I)
\\ \midrule
B_n^{(1)} &
  \begin{array}{cl} \gamma_a = 2 & (a \neq n) \\ \gamma_n = 1 \end{array}
\\ \midrule
C_n^{(1)} &
  \begin{array}{cl} \gamma_a = 2 & (a = 0,n) \\ \gamma_a = 1 & (a \neq 0,n) \end{array}
\\ \midrule
A_{2n}^{(2)} &
  \begin{array}{cl} \gamma_a = 2 & (a \neq n) \\ \gamma_n = 1 \end{array}
\\ \midrule
A_{2n}^{(2)\dagger} &
  \begin{array}{cl} \gamma_a = 2 & (a \neq 0) \\ \gamma_0 = 1 \end{array}
\\ \midrule
F_4^{(1)} &
  (2,2,2,1,1)
\\ \midrule
G_2^{(1)} &
  (3,1,3)
\\ \bottomrule
\end{array}
\]
\caption{The values $\gamma_a$ for $a \in I$.}
\label{table:scaling_factors}
\end{table}

Furthermore, we have a natural embedding $\Psi \colon P \to \virtual{P}$ given by
\[
\Lambda_a \mapsto \gamma_a \sum_{b \in \phi^{-1}(a)} \virtual{\Lambda}_b,
\]
and note this induces a similar embedding on the root lattice
\[
\alpha_a \mapsto \gamma_a \sum_{b \in \phi^{-1}(a)} \virtual{\alpha}_b
\]
and $\Psi(\delta) = c_0 \gamma_0 \virtual{\delta}$.

\begin{definition}
Let $\virtual{B}$ be a $U_q'(\virtual{\g})$-crystal and $V \subseteq \virtual{B}$. Let $\phi$ and 
$(\gamma_a)_{a \in I}$ be the folding and the scaling factors given above. The \defn{virtual crystal operators} (of type $\g$) are defined as
\[
e^v_a := \prod_{b \in \phi^{-1}(a)} \virtual{e}_b^{\;\gamma_a},
\hspace{100pt}
f^v_a := \prod_{b \in \phi^{-1}(a)} \virtual{f}_b^{\;\gamma_a}.
\]
A \defn{virtual crystal} is the quadruple $(V, \virtual{B}, \phi, (\gamma_a)_{a \in I})$ such that $V$ has an abstract $U_q'(\g)$-crystal structure defined by
\begin{equation}
\label{eq:virtual_crystal}
\def\arraystretch{1.3}
\begin{array}{c}
e_a := e^v_a, \hspace{100pt} f_a := f^v_a, \\
\varepsilon_a := \gamma_a^{-1} \virtual{\varepsilon}_b, \hspace{100pt} \varphi_a := \gamma_a^{-1} \virtual{\varphi}_b, \\
\wt := \Psi^{-1} \circ \virtual{\wt},
\end{array}
\end{equation}
for any $b \in \phi^{-1}(a)$.
\end{definition}

When there is no danger of confusion, we simply denote the virtual crystal by $V$. We say $B$ \defn{virtualizes} in $\virtual{B}$ if there exists a $U_q'(\g)$-crystal isomorphism $v \colon B \to V$ for some virtual crystal, and we call the resulting isomorphism a \defn{virtualization map}.

It is straightforward to see that virtual crystals are closed under direct sums. Moreover, they are closed under tensor products.

\begin{proposition}[{\cite[Prop.~6.4]{OSS03III}}]
Virtual crystals form a tensor category.
\end{proposition}

Furthermore, KR crystals are conjecturally well-behaved under virtualization.

\begin{conjecture}[{\cite[Conj.~3.7]{OSS03III}}]
\label{conj:KR_virtualization}
There exists a virtualization map from the KR crystal $B^{r,s}$ into
\[
\virtual{B}^{r,s} = \begin{cases}
B^{n,s} \otimes B^{n,s} & \text{if } \g = A_{2n}^{(2)}, A_{2n}^{(2)\dagger} \text{ and } r = n,
\\ \bigotimes_{r' \in \phi^{-1}(r)} B^{r', \gamma_r s} & \text{otherwise}.
\end{cases}
\]
\end{conjecture}
Conjecture~\ref{conj:KR_virtualization} is known for all non-exceptional types~\cite[Thm.~5.1]{Okado13} and $B^{r,1}$ for all other types unless $r=2$ and $\g$ is of type $F_4^{(1)}$~\cite{SchillingS15}. A more uniform proof of Conjecture~\ref{conj:KR_virtualization} for some special cases using projected level-zero LS paths was given in~\cite{PS15}; in particular, for $B^{r,1}$ when $\gamma_r = 1$.

%%%%%%%%%%
\subsection{Adjoint crystals}

We recall the construction of certain level 1 perfect crystals from~\cite{BFKL06}. Define $\theta := \delta / c_0 - \alpha_0$, and so
%\begin{equation}
%\label{eq:highest_root}
\[
\theta = (c_1 \alpha_1 + c_2 \alpha_2 + \cdots + c_n \alpha_n) / c_0.
\]
%\end{equation}
Let $\Delta = \{\overline{\wt}(b) \mid b \in B(\theta) \} \setminus \{0\}$. We denote the vertices of the $U_q(\g_0)$-crystal $B(\theta)$ by
\[
\{x_{\alpha} \mid \alpha \in \Delta\} \sqcup \{y_a \mid a \in I_0, \alpha_a \in \Delta\},
\]
and $a$-arrows of $B(\theta)$ are given by
\begin{itemize}
\item $x_{\alpha} \xrightarrow[\hspace{20pt}]{a} x_{\beta}$ if and only if $\alpha - \alpha_a = \beta$, or
\item $x_{\alpha_a} \xrightarrow[\hspace{20pt}]{a} y_a \xrightarrow[\hspace{20pt}]{a} x_{-\alpha_a}$.
\end{itemize}
The (classical) weight function $\overline{\wt}$ is given by $\overline{\wt}(x_{\alpha}) = \alpha$ and $\overline{\wt}(y_a) = 0$. Let $\Delta^{\pm} := \Delta \cap P^{\pm}$, and we say $\overline{\wt}(b) > 0$ if $\overline{\wt}(b) \in \Delta^+$ and $\overline{\wt}(b) < 0$ if $\overline{\wt}(b) \in \Delta^-$.

\begin{remark}
\label{remark:notes_adjoint_crystal}
If $\g$ is of untwisted type, then $B(\theta)$ is the adjoint representation and $\Delta$ is the set of all roots of $\g_0$. For $\g$ of twisted type, $B(\theta)$ is the ``little'' adjoint representation of $\g_0$ with highest weight being the highest short root of $\g_0$ and $\Delta$ the set of all short roots of $\g_0$. For type $A_{2n}^{(2)}$, there does not exist an $a \in I_0$ such that $\alpha_a \in \Delta$. For type $A_{2n}^{(2)\dagger}$, we have $B(\theta) = B(\clfw_1)$ of type $B_n$.
For more information on finite (crystallographic) root systems, we refer the reader to standard texts such as~\cite{Bourbaki02,Humphreys90}.
\end{remark}

Let $\emptyset$ be the unique element of $B(0)$. We then define a $U_q'(\g)$-crystal $B^{\theta,1}$ by the classical decomposition
\[
B^{\theta,1} \iso \begin{cases}
B(\theta) & \text{if $\g$ is of type $A_{2n}^{(2)\dagger}$,} \\
B(\theta) \oplus B(0) & \text{otherwise,}
\end{cases}
\]
and $0$-arrows
\begin{itemize}
\item $x_{\alpha} \xrightarrow[\hspace{20pt}]{0} x_{\beta}$ if and only if $\alpha + \theta = \beta$ and $\alpha,\beta \neq \pm\theta$, or
\item $x_{-\theta} \xrightarrow[\hspace{20pt}]{0} \emptyset \xrightarrow[\hspace{20pt}]{0} x_{\theta}$.
\end{itemize}
The weight is given by requiring it to be level zero. That is to say, we let
\begin{equation}
\label{eq:level_zero_weights}
\wt(b) = \overline{\wt}(b) + k_0 \Lambda_0,
\end{equation}
under the natural lift $\clfw_a \to \Lambda_a$ for all $a \in I_0$ and $k_0$ is such that $\inner{c}{\wt(b)} = 0$.
Thus we have
\[
B^{\theta,1} \iso \begin{cases}
B^{n,1} \otimes B^{1,1} & \text{if $\g$ is of type $A_n^{(1)}$,} \\
B^{N_{\g},2} & \text{if $\g$ is of type $C_n^{(1)}$,} \\
B^{N_{\g},1} & \text{otherwise.}
\end{cases}
\]

\begin{remark}
\label{remark:notes_A2dual}
In the construction of $B^{\theta,1}$, we can consider $\emptyset = y_0$. Thus, for type $A_{2n}^{(2)\dagger}$, as there are not elements $x_{\alpha_0}$ nor $x_{-\alpha_0}$, we do not include $y_0$. This reflects the fact that $A_{2n}^{(2)\dagger}$ is isomorphic to $A_{2n}^{(2)}$ after forgetting about the choice of the affine node.
\end{remark}

There exists higher level analogs $B^{\theta,s}$, where as classical crystals, we have $B^{\theta,s} \iso \bigoplus_{k=0}^s B(k\theta)$. The $U_q'(\g)$-crystal structure is currently known for all non-exceptional types~\cite{FOS09,KKM94,KKMMNN92,Kodera09,SS06} and type $D_4^{(3)}$~\cite{KMOY07}. However, the $U_q'(\g)$-crystal structure is not known in general, much less uniformly. One potential approach might be to generalize the approach of~\cite{Kodera09} by examining various embeddings of $B^{\theta,s-1}$ into $B^{\theta,s}$, where the difficulty is overcoming that the multiplicity of the weights of $B(k\theta)$ that do not appear in $B\bigl((k-1)\theta\bigr)$ are not all $1$ in general.

%%%%%%%%%%
\subsection{Rigged configurations}

For this section, we assume that $\g$ is not of type $A_{2n}^{(2)}$ or $A_{2n}^{(2)\dagger}$ for simplicity of the exposition. However, the analogous statements with the necessary modifications for these types may be found in~\cite{SchillingS15}.

Denote $\HH_0 := I_0 \times \ZZ_{>0}$. Consider a tensor product of KR crystals $B = \bigotimes_{i=1}^N B^{r_i, s_i}$. A \defn{configuration} $\nu = \bigl(\nu^{(a)}\bigr)_{a \in I_0}$ is a sequence of partitions. Let $m_i^{(a)}$ denote the multiplicity of $i$ in $\nu^{(a)}$. Define the \defn{vacancy numbers} as
\begin{equation}
\label{eq:vacancy}
\begin{split}
p_i^{(a)}(\nu; B) & = \sum_{j \in \ZZ_{>0}} L_j^{(a)} \min(i, j) - \frac{1}{t_a^{\vee}} \sum_{(b,j) \in \HH_0} (\clsr_a | \clsr_b) \min(t_b \upsilon_a i, t_a \upsilon_b j) m_j^{(b)}
\\ & = \sum_{j \in \ZZ_{>0}} L_j^{(a)} \min(i,j) - \sum_{b \in I_0} \frac{A_{ab}}{\gamma_b} \sum_{j \in \ZZ_{>0}} \min(\gamma_a i, \gamma_b j) m_j^{(b)},
\end{split}
\end{equation}
where $L_s^{(r)}$ equals the number of factors $B^{r,s}$ that occur in $B$ and
\[
\upsilon_a = \begin{cases}
2 & a = n \text{ and } \g = C_n^{(1)}, \\
\frac{1}{2} & a = n \text{ and } \g = B_n^{(1)}, \\
\frac{1}{2} & a = 3,4 \text{ and } \g = F_4^{(1)}, \\
\frac{1}{3} & a = 1 \text{ and } \g = G_2^{(1)}, \\
1 & \text{otherwise.}
\end{cases}
\]
When there is no danger of confusion, we will simply write $p_i^{(a)}$. Note that when $t_a = 1$ for all $a \in I_0$, we have
\begin{equation}
\label{eq:nonzero_convexity}
-p_{i-1}^{(a)} + 2 p_i^{(a)} - p_{i+1}^{(a)} = L_i^{(a)} - \sum_{b \sim a} A_{ab} m_i^{(b)}.
\end{equation}
Moreover, when $m_i^{(a)} = 0$, we have the \defn{convexity inequality}
\begin{equation}
\label{eq:convexity}
p_{i-1}^{(a)} + p_{i+1}^{(a)} \leq 2 p_i^{(a)}
\end{equation}
or equivalently $-p_{i-1}^{(a)} + 2p_i^{(a)} - p_{i+1}^{(a)} \geq 0$.

\begin{remark}
\label{remark:conventions}
The values $(\upsilon_a)_{a \in I_0}$ arise from a different convention for rigged configurations than those used in, \textit{e.g.},~\cite{OSS03}.
\end{remark}

A \defn{$B$-rigged configuration} is the pair $(\nu, J)$, where $\nu$ is a configuration and $J = (J_i^{(a)})_{(a,i) \in \HH_0}$ is such that $J_i^{(a)}$ is a multiset $\{x \in \ZZ \mid x \leq p_i^{(a)}(\nu; B) \}$ with $\lvert J_i^{(a)} \rvert = m_i^{(a)}$ for all $(a, i) \in \HH_0$. When $B$ is clear, we call a $B$-rigged configuration simply a rigged configuration. A \defn{highest weight} rigged configuration is a rigged configuration $(\nu, J)$ such that $\min J_i^{(a)} \geq 0$ for all $(a, i) \in \HH_0$ such that $m_i^{(a)} > 0$. Let $\hwRC(B)$ denote the set of all highest weight $B$-rigged configurations.

The integers in $J_i^{(a)}$ are called \defn{riggings} or \defn{labels}. The \defn{corigging} or \defn{colabel} of a rigging $x \in J_i^{(a)}$ is defined as $p_i^{(a)} - x$. We note that we can associate a row of length $i$ in $\nu^{(a)}$ with a rigging $x$, and we call such a pair $(i, x)$ a \defn{string}. We identify each row of the partition $\nu^{(a)}$ with its corresponding string. %We let $(\nu, J)^{(a)}$ denote the partition $\nu^{(a)}$ with the corresponding riggings or equivalently the set of strings corresponding to $\nu^{(a)}$.
We say a row (or string) is \defn{singular} if $p_i^{(a)} = x$.
We say a row (or string) is \defn{quasisingular} if $p_i^{(a)} = x + 1$ and there does not exist a singular row of length $i$.

Next, let $\RC(B)$ denote the closure of $\hwRC(B)$ under the following crystal operators. Fix a rigged configuration $(\nu, J)$ and $a \in I_0$. For simplicity, we assume there exists a row of length $0$ in $\nu^{(a)}$ with a rigging of $0$. Let $x = \min \{ \min J_i^{(a)} \mid i \in \ZZ_{>0} \}$; \textit{i.e.}, the smallest rigging in $(\nu, J)^{(a)}$.
\begin{description}
\item[\defn{$e_a$}] If $x = 0$, then define $e_a(\nu, J) = 0$. Otherwise, remove a box from the smallest row with rigging $x$, replace that rigging with $x + 1$, and change all other riggings so that the coriggings remain fixed. The result is $e_a(\nu, J)$.

\item[\defn{$f_a$}] Add a box from the largest row with rigging $x$, replace that rigging with $x - 1$, and change all other riggings so that the coriggings remain fixed. The result is $f_a(\nu, J)$ unless the result is not a valid rigged configuration, in which case $f_a(\nu, J) = 0$.
\end{description}
We can extend this to a full $U_q(\g_0)$-crystal structure on $\RC(B)$ by
\[
\overline{\wt}(\nu, J) = \sum_{(a, i) \in \HH_0} i \left( L_i^{(a)} \clfw_a - m_i^{(a)} \clsr_a \right).
\]
We note that
\[
\inner{\alpha_a^{\vee}}{\overline{\wt}(\nu, J)} = p_{\infty}^{(a)} = \sum_{j \in \ZZ_{>0}} j L_j^{(a)} - \sum_{b \in I_0} A_{ab} \lvert \nu^{(b)} \rvert,
\]
and we can extend the classical weight $\overline{\wt} \colon \RC(B) \to \overline{P}$ to the affine weight $\wt \colon \RC(B) \to P$ as in Equation~\eqref{eq:level_zero_weights}.

\begin{theorem}[{\cite{S06,SchillingS15}}]
\label{thm:rc_crystal}
Let $B$ be a tensor product of KR crystals. Fix some $(\nu, J) \in \hwRC(B)$. Let $X_{(\nu, J)}$ denote the closure of $(\nu, J)$ under $e_a$ and $f_a$ for all $a \in I_0$. Then $X_{(\nu, J)} \iso B(\lambda)$, where $\lambda = \overline{\wt}(\nu, J)$.
\end{theorem}

Furthermore, we have the following way to compute the statistics $\varepsilon_a$ and $\varphi_a$ on a rigged configuration.

\begin{proposition}[{\cite{Sakamoto14,S06,SchillingS15}}]
\label{prop:ep_phi}
Let $x = \min \{ \min J_i^{(a)} \mid i \in \ZZ_{>0} \}$; \textit{i.e.}, the smallest rigging in $(\nu, J)^{(a)}$.
We have
\[
\varepsilon_a(\nu, J) = -x, \hspace{100pt} \varphi_a(\nu, J) = p_{\infty}^{(a)} - x.
\]
\end{proposition}

\begin{remark}
Proposition~\ref{prop:ep_phi} states that we could define $f_a(\nu, J) = 0$ if and only if $\varphi_a(\nu, J) = p_{\infty}^{(a)} - x = 0$.
\end{remark}

We will need the \defn{complement rigging} involution $\eta \colon \RC(B) \to \RC(B^{\mathrm{rev}})$, where $B^{\mathrm{rev}}$ is $B$ in the reverse order. The map $\eta$ is given on highest weight rigged configurations by replacing each rigging with its corresponding corigging and then extended as a $U_q(\g_0)$-crystal isomorphism.

Additionally, rigged configurations are known to be well-behaved under virtualization~\cite{OSS03III,OSS03II,SchillingS15}. We construct a virtualization map $v \colon \RC(B) \to \RC(\virtual{B})$, where the virtual rigged configuration $(\virtual{\nu}, \virtual{J}) := v(\nu, J)$ is given by
\begin{subequations}
\label{eq:virtual_RC}
\begin{align}
\label{eq:virtual_m} \virtual{m}_{\gamma_a i}^{(b)} & = m_i^{(a)},
\\ \label{eq:virtual_J} \virtual{J}_{\gamma_a i}^{(b)} & = \gamma_a J_i^{(a)},
\end{align}
\end{subequations}
for all $b \in \phi^{-1}(a)$. Moreover, we have
\begin{equation}
\label{eq:virtual_vacancy_numbers}
\virtual{p}_{\gamma_a i}^{(b)} = \gamma_a p_i^{(a)}
\end{equation}
for all $b \in \phi^{-1}(a)$.

%%%%%%%%%%
\subsection{Statistics}

We now describe two important statistics that arise from mathematical physics. The first is defined on tensor products of KR crystals and the second is defined on rigged configurations.

Let $B^{r,s}$ and $B^{r',s'}$ be KR crystals of type $\g$. The \defn{local energy function} $H \colon B^{r,s} \otimes B^{r',s'} \to \ZZ$ is defined as follows. Let $\widetilde{b}' \otimes \widetilde{b} = R(b \otimes b')$, and define the following conditions:
\begin{itemize}
\item[(LL)] $e_0(b \otimes b') = e_0 b \otimes b' \text{ and } e_0(\widetilde{b}' \otimes \widetilde{b}) = e_0 \widetilde{b}' \otimes \widetilde{b}$;
\item[(RR)] $e_0(b \otimes b') = b \otimes e_0 b' \text{ and } e_0(\widetilde{b}' \otimes \widetilde{b}) = \widetilde{b}' \otimes e_0 \widetilde{b}$.
\end{itemize}
The local energy function is given by
\begin{equation}
\label{eq:local_energy}
H\bigl( e_i(b \otimes b') \bigr) = H(b \otimes b') + \begin{cases}
-1 & \text{if } i = 0 \text{ and } (LL), \\
1 & \text{if } i = 0 \text{ and } (RR), \\
0 & \text{otherwise,}
\end{cases}
\end{equation}
and it is known $H$ is uniquely defined up to an additive constant~\cite{KKMMNN91}. We normalize $H$ by the condition $H\bigl( u(B^{r ,s}) \otimes u(B^{r',s'}) \bigr) = 0$.

Next consider $B^{r,s}$, and let $b^{\sharp}$ be the unique element such that $\varphi(b^{\sharp}) = \ell \Lambda_0$, where $\ell = \min \{ \inner{c}{\varphi(b)} \mid b \in B^{r,s} \}$. We then define $D_{B^{r,s}} \colon B^{r,s} \to \ZZ$, following~\cite{HKOTT02}, by
\[
D_{B^{r,s}}(b) = H(b \otimes b^{\sharp}) - H(u(B^{r,s}) \otimes b^{\sharp}).
\]
%\travis{This and the uniqueness of $b^{\sharp}$ might be a conjecture of~\cite{HKOTT02} still}
Let $B = \bigotimes_{i=1}^N B^{r_i,s_i}$. We define \defn{energy}~\cite{HKOTY99} $D \colon B \to \ZZ$ by
\begin{equation}
\label{eq:energy_function}
D = \sum_{1 \leq i < j \leq N} H_i R_{i+1} R_{i+2} \cdots R_{j-1} + \sum_{j=1}^N D_{B^{r_j,s_j}} R_1 R_2 \cdots R_{j-1},
\end{equation}
where $R_i$ and $H_i$ are the combinatorial $R$-matrix and local energy function, respectively, acting on the $i$-th and $(i+1)$-th factors and $D_{B^{r_j,s_j}}$ acts on the rightmost factor. Note that $D$ is constant on classical components since $H$ is and $R$ is a $U_q'(\g)$-crystal isomorphism.

For rigged configurations, we define a statistic called \defn{cocharge} as follows. First consider  a configuration $\nu$, and define the cocharge of $\nu$ by
\begin{equation}
\label{eq:cocharge_configurations}
\cc(\nu) = \frac{1}{2} \sum_{\substack{(a,i) \in \HH_0 \\ (b, j) \in \HH_0}} (\clsr_a | \clsr_b) \min(t_b \upsilon_a i, t_a \upsilon_b j) m_i^{(a)} m_j^{(b)}.
\end{equation}
To obtain the cocharge of a rigged configuration $(\nu, J)$, we add all of the riggings to $\cc(\nu)$:
\begin{equation}
\label{eq:cocharge}
\cc(\nu, J) = \cc(\nu) + \sum_{(a,i) \in \HH_0} t_a^{\vee} \sum_{x \in J_i^{(a)}} x.
\end{equation}
Moreover, it is known that cocharge is invariant under the classical crystal operators.
\begin{proposition}[{\cite{S06,SchillingS15}}]
\label{prop:cocharge_classical_invar}
Fix a classical component $X_{(\nu, J)}$ as given in Theorem~\ref{thm:rc_crystal}. The cocharge $\cc$ is constant on $X_{(\nu, J)}$.
\end{proposition}

Let $q$ be an indeterminate. The \defn{one-dimensional sum} is defined as
\begin{equation}
\label{eq:X}
X(B, \lambda; q) = \sum_{b \in \mathcal{P}(B; \lambda)} q^{D(b)},
\end{equation}
where $\mathcal{P}(B; \lambda)$ denotes the classically highest weight elements of $B$ of classical weight $\lambda$. The \defn{fermionic formula} is defined as
\begin{equation}
\label{eq:M}
M(B, \lambda; q) = \sum_{\nu \in \mathcal{C}(B; \lambda)} q^{\cc(\nu)} \prod_{(a, i) \in \HH_0} \begin{bmatrix} m_i^{(a)} + p_i^{(a)} \\ m_i^{(a)} \end{bmatrix}_q,
\end{equation}
where $\mathcal{C}(B; \lambda)$ are all $B$-configurations of classical weight $\lambda$ and $\qbinom{a}{b}{q}$ is the usual $q$-binomial.
Note that $J_i^{(a)}$ of a highest weight rigged configuration can be considered as a partition in a $p_i^{(a)} \times m_i^{(a)}$ box for all $(a, i) \in \HH_0$. Thus we can write
\[
M(B, \lambda; q) = \sum_{(\nu, J) \in \hwRC(B; \lambda)} q^{\cc(\nu, J)}.
\]
Now we recall the $X = M$ conjecture of~\cite{HKOTY99,HKOTT02}.\footnote{To obtain the formulas of~\cite{HKOTY99,HKOTT02}, we need to substitute $q = q^{-1}$.}

\begin{conjecture}[$X = M$ conjecture]
\label{conj:X=M}
Let $B$ be a tensor product of KR crystals of type $\g$. Then we have
\[
X(B, \lambda; q) = M(B, \lambda; q).
\]
\end{conjecture}

Consider a virtualization map $v \colon B \to B^v$. We first define the virtual combinatorial $R$-matrix $R^v \colon B^v \otimes B'^v \to B'^{v} \otimes B^v$ as the restriction of the type $\virtual{\g}$ combinatorial $R$-matrix $\virtual{R}$ to the image of $v$. We note that it is not clear that $R^v$ is well-defined, but this follows for $B^{r,1} \otimes B^{r',1}$ for dual untwisted types from the results of~\cite{LL16,LNSSS14,LNSSS14II,PS15}.
Thus, we may define virtual analogs of (local) energy and cocharge by
\begin{align*}
H^v(b \otimes b') & := \gamma_0^{-1} \virtual{H}\bigl( v(b) \otimes v(b') \bigr),
\\ D^v(b) & := \gamma_0^{-1} \virtual{D}\bigl( v(b) \bigr),
\\ \cc^v(\nu, J) & := \gamma_0^{-1} \virtual{\cc}(\virtual{\nu}, \virtual{J}).
\end{align*}
Hence, we define
\begin{align*}
X^v(B, \lambda; q) & = \sum_{b \in \mathcal{P}(B; \lambda)} q^{D^v(b)},
\\ M^v(B, \lambda; q) & = \sum_{(\nu, J) \in \hwRC(B; \lambda)} q^{\cc^v(\nu, J)}.
\end{align*}

\begin{proposition}[{\cite{OSS03II}}]
\label{prop:virtual_statistics}
Let $B^v$ be a virtual crystal of $B$. Then we have
\begin{align*}
D^v(b) & = D(b),
\\ \cc^v(\nu, J) & = \cc(\nu, J).
\end{align*}
Moreover, we have
\begin{align*}
X^v(B, \lambda; q) & = X(B, \lambda; q),
\\ M^v(B, \lambda; q) & = M(B, \lambda; q).
\end{align*}
\end{proposition}

%For convenience we define
%\[
%\overline{H} = -H, \hspace{20pt} \overline{D} = -D, \hspace{20pt} \overline{X}(\lambda, \mu; q) = X(\lambda, \mu; q^{-1}).
%\]

%%%%%%%%%%
\subsection{Kleber algorithm}

These results will be used in Section~\ref{sec:filling_map}.

We first recall the \defn{Kleber algorithm}~\cite{Kleber98} for when $\g$ is an affine type such that $\g_0$ is simply-laced. For $x,y \in \overline{P}^+$, let $d_{xy} := x - y$.
\begin{definition}[Kleber algorithm]
\label{def:kleber_algorithm}
Let $B$ be a tensor product of KR crystals of type $\g$. The \defn{Kleber tree} $T(B)$ is a tree whose nodes will be given by weights in $\overline{P}^+$ and edges are labeled by $d_{xy} \in \overline{Q}^+ \setminus \{0\}$ and constructed recursively as follows. Begin with $T_0$ being the tree consisting of a single node of weight $0$. We then do the following steps starting with $\ell = 1$.
\begin{itemize}
\item[(K1)] Let $T^{\prime}_{\ell}$ be obtained from $T_{\ell-1}$ by adding $\sum_{a=1}^n \clfw_a \sum_{i \geq \ell} L_i^{(a)}$ to the weight of each node.
\item[(K2)] Construct $T_{\ell}$ from $T^{\prime}_{\ell}$ as follows. Let $x$ be a node at depth $\ell - 1$. Suppose there is a weight $y \in \overline{P}^+$ such that $d_{xy} \in \overline{Q}^+ \setminus \{0\}$. If $x$ is not the root, then let $w$ be the parent of $x$. Then we have $d_{wx} - d_{xy} \in \overline{Q}^+ \setminus \{0\}$. For all such $y$, attach $y$ as a child of $x$.
\item[(K3)] If $T_{\ell} \neq T_{\ell-1}$, then repeat from~(K1); otherwise terminate and return $T(B) = T_{\ell}$.
\end{itemize}
Now we convert the tree to highest weight rigged configurations as follows. Let $x$ be a node at depth $p$ in $T(B)$, and 
$x_0, x_1, \dotsc, x_p = x$ be the weights of nodes on the path from the root of $T(B)$ to $x$. The resulting configuration $\nu$ is given by
\[
m_i^{(a)} = (x_{i-1} - 2 x_i + x_{i+1} \mid \clfw_a)
\]
where we make the convention that $x = x_j$ for all $j > p$. We then take the riggings over all possible values between $0$ and $p_i^{(a)}$.
\end{definition}

\begin{remark}
We can reformulate the construction of the configuration $\nu$ in the following ways. Suppose $d_{x_{i-1}x_i} = \sum_{b \in I} k_i^{(b)} \alpha_b$ for all $1 \leq i \leq p$. There are $k_i^{(a)}$ rows of length $i$ in $\nu^{(a)}$. We also have $\nu^{(a)}$ equal to the transpose of the partition $k_1^{(a)} k_2^{(a)} \cdots k_p^{(a)}$, or we stack a column of height $k_a^{(i)}$ over all $i$.
\end{remark}

When $\g_0$ is of non-simply-laced type, we use the \defn{virtual Kleber algorithm}~\cite{OSS03II} by using virtual rigged configurations.

\begin{definition}[Virtual Kleber algorithm]
\label{def:virtual_kleber}
The \defn{virtual Kleber tree} is defined from the Kleber tree of $\virtual{B}$ in the ambient type, but we only add a child $y$ to $x$ in step~(K2) if the following conditions are satisfied:
\begin{itemize}
\item[(V1)] $(y \mid \virtual{\alpha}_b) = (y \mid \virtual{\alpha}_{b'})$ for all $b, b' \in \phi^{-1}(a)$.
\item[(V2)] For all $a \in I_0$, if $\ell - 1 \notin \gamma_a \ZZ$, then the coefficient of $\clsr_a$ in $d_{wx}$ and $d_{xy}$,  where $w$ is the parent of $x$, must be equal.
\end{itemize}
Let $\virtual{T}(B)$ be the resulting tree, which we will call the \defn{ambient tree}. Let $\gamma = \max_{a \in I} \gamma_a$. We now select nodes which satisfy either:
\begin{itemize}
\item[(A1)] $y$ is at depth $\ell \in \gamma \ZZ$, or
\item[(A2)] $(d_{xy} \mid \virtual{\clfw}_a) = 0$ for every $a$ such that $1 < \gamma = \gamma_a$, where $x$ is the parent of $y$.
\end{itemize}
We construct the final rigged configurations from the selected nodes by devirtualizing the (virtual) rigged configurations obtained from the usual Kleber algorithm satisfying Equation~\eqref{eq:virtual_J} (note that Equation~\eqref{eq:virtual_m} is satisfied by~(V1) and~(V2)).
\end{definition}

%%%%%%%%%%
\subsection{KSS-type bijection}
\label{sec:bijection_KSS}

In this section, we describe the (conjectural) KSS-type bijection $\KSS{\Phi} \colon \RC(B) \to B$.

Let $B$ be a tensor product of KR crystals. We consider $B$ expressed in terms of the so-called \defn{Kirillov--Reshetikhin (KR) tableaux} of~\cite{OSS13,SchillingS15,Scrimshaw15}. KR tableaux, generally speaking, are $r \times s$ rectangular tableaux filled with entries of $B^{1,1}$ and determined by their classically highest weight elements. Following~\cite{KSS02}, we define a map $\KSS{\Phi} \colon \RC(B) \to B$ recursively by the composition of
\begin{align*}
\KSS{\delta} \colon B^{1,1} \otimes B^{\bullet} & \to B^{\bullet},
\\ \young(b) \otimes b^{\bullet} & \mapsto b^{\bullet},
\allowdisplaybreaks \\ \lb \colon B^{r,1} \otimes B^{\bullet} & \to B^{1,1} \otimes B^{r-1,1} \otimes B^{\bullet} \qquad (r \neq 1),
\\ \begin{array}{|c|} \hline b_1 \\\hline \vdots \\\hline b_{r-1} \\\hline b_r \\\hline \end{array} \otimes b^{\bullet} & \mapsto \begin{array}{|c|} \hline b_r \\\hline \end{array} \otimes \begin{array}{|c|} \hline b_1 \\\hline \vdots \\\hline b_{r-1} \\\hline \end{array} \otimes b^{\bullet},
\allowdisplaybreaks \\ \ls \colon B^{r,s} \otimes B^{\bullet} & \to B^{r,1} \otimes B^{r,s-1} \otimes B^{\bullet} \qquad (s \geq 2),
\\ \begin{array}{|c|c|c|c|} \hline b_{11} & b_{12} & \cdots & b_{1s} \\\hline \vdots & \vdots & \ddots & \vdots \\\hline b_{r1} & b_{r2} & \cdots & b_{rs} \\\hline \end{array} \otimes b^{\bullet} & \mapsto \begin{array}{|c|} \hline b_{11} \\\hline \vdots \\\hline b_{r1} \\\hline \end{array} \otimes  \begin{array}{|c|c|c|} \hline b_{12} & \cdots & b_{1s} \\\hline \vdots & \ddots & \vdots \\\hline b_{r2} & \cdots & b_{rs} \\\hline \end{array} \otimes b^{\bullet},
&& 
\end{align*}
where $B^{\bullet}$ is a tensor product of KR crystals, and their corresponding maps on rigged configurations. We do not explicitly recall the map $\KSS{\delta}$ on rigged configurations here as it strongly depends upon type and we later give a more uniform description of the map. Instead we refer the reader to~\cite{OS12,OSS03,Scrimshaw15} for the explicit (type-dependent) descriptions. Note that $\KSS{\delta}$ is currently only described/known for non-exceptional affine types, type $E_6^{(1)}$, and type $D_4^{(3)}$.\footnote{A map for when the left factor is $B^{2,1}$ of type $E_6^{(1)}$ was conjectured in~\cite{Mahathir}.} Moreover, $\KSS{\delta}$ is the key component of the bijection. The map $\lb$ is given for all non-exceptional types by adding a length 1 singular row to all $\nu^{(a)}$ for all $a < r$. The map $\ls$ is the identity map.

We recall and consolidate some of the conjectures given in~\cite{SchillingS15} and has been known to experts prior.

\begin{conjecture}
\label{conj:bijection}
Let $B$ be a tensor product of KR crystals of affine type. Then $\KSS{\Phi} \colon \RC(B) \to B$ is a (classical) crystal isomorphism such that $D \circ \KSS{\Phi} \circ \eta = \cc$ and the diagram
\[
\xymatrixrowsep{3pc}
\xymatrixcolsep{3.5pc}
\xymatrix{\RC(B) \ar[r]^{\KSS{\Phi}} \ar[d]_{\id} & B \ar[d]^{R} \\ \RC(B') \ar[r]_{\KSS{\Phi}} & B'}
\]
commutes, where $B'$ is a reordering of the factors of $B$.
\end{conjecture}

When we restrict $\KSS{\Phi}$ to classically highest weight elements, this gives a combinatorial proof of the $X = M$ conjecture of~\cite{HKOTY99,HKOTT02}.

In type $A_n^{(1)}$, it was shown in~\cite{KSS02} that Conjecture~\ref{conj:bijection} holds on classically highest weight elements, and as such, the analogous (conjectural) bijections are known as KSS-type bijections. This was an extension of the pioneering work of Kerov, Kirillov, and Reshetikhin in~\cite{KKR86,KR86}, where they showed Conjecture~\ref{conj:bijection} is true for classically highest weight elements in the special case $B = (B^{1,1})^{\otimes N}$ of type $A_n^{(1)}$. In~\cite{DS06}, it was shown that $\KSS{\Phi}$ is a classical crystal isomorphism in type $A_n^{(1)}$ and a $U_q'(\g)$-crystal isomorphism in~\cite{SW10}. Furthermore, Conjecture~\ref{conj:bijection} is known to (partially) hold in many different special cases for classically highest weight elements across the non-exceptional types~\cite{OSS13,OSS03,OSS03III,OSS03II,S05,SchillingS15,SS2006}, with recent and some progress has been made in the exceptional types~\cite{OS12,Scrimshaw15}. Furthermore, in~\cite{Sakamoto14}, it was shown that $\KSS{\Phi}$ is a classical crystal isomorphism in type $D_n^{(1)}$. Recently, the general case for type $D_n^{(1)}$ was proven in~\cite{OSSS16} and all non-exceptional types in~\cite{OSS17}.

As we will need it later on, we recall the general steps of the proof that $\KSS{\Phi}$ is a statistic preserving bijection on highest weight elements for $B = \bigotimes_{k=1}^N B^{1, 1}$ when $r_k = 1$ for all $k$. Let $(\nu, J) \in \hwRC(B; \lambda)$.
Define $(\overline{\nu}, \overline{J}) = \KSS{\delta}(\nu, J)$ and let $r$ be the return value from $\KSS{\delta}$.
Denote by $\beta_1^{(r_N)}, \overline{\beta}_1^{(r_N)}$ the length of the first column in $\nu^{(r_N)}, \overline{\nu}^{(r_N)}$, respectively.

There are five things the need to be verified to show that $\KSS{\Phi}$ is a statistic preserving bijection on classically highest weight elements for $B = B^{r_N, 1} \otimes B^{\bullet}$, where $B^{\bullet}$ is a tensor product of KR crystals:
\begin{enumerate}[(I)]
\item\label{cI} $\lambda - \overline{\wt}(r)$ is dominant.
\item\label{cII} $\KSS{\delta}(\nu, J) \in \hwRC\bigl(B^{\bullet}, \lambda - \wt(r)\bigr)$.
\item\label{cIII} $r$ can be appended to $(\overline{\nu}, \overline{J})$ to give $(\nu, J)$.
\item\label{cIV} For $N \geq 2$, we have
\begin{equation}
\label{eq:change_cocharge}
\cc(\nu, J) - \cc(\overline{\nu}, \overline{J}) = \frac{t_{r_N}^{\vee}}{c_0^{\vee}} \beta_1^{(r_N)} - \chi(b_N = \emptyset).
\end{equation}
\item\label{cV} For $N \geq 2$, we have
\begin{equation}
\label{eq:change_energy}
H(b_N \otimes b_{N-1}) = \frac{t_{r_N}^{\vee}}{c_0^{\vee}} \left( \beta_1^{(r_N)} - \overline{\beta}_1^{(r_N)} \right) - \chi(b_N = \emptyset) + \chi(b_{N-1} = \emptyset),
\end{equation}
\end{enumerate}
where $\chi(S)$ is $1$ if the statement $S$ is true and $0$ otherwise.

We remark that Equation~\eqref{eq:change_cocharge} and Equation~\eqref{eq:change_energy} are those given in~\cite[Lemma~5.1]{OSS03}.

Next, we will need dual notions of the maps $\KSS{\delta}$, $\lb$, and $\ls$ acting on the right, which we denote by $\KSS{\delta}^{\lusztig}$, $\rb$, and $\rs$. First, we recall the definition of \defn{Lusztig's involution} $\lusztig \colon B(\lambda) \to B(\lambda)$, the unique $U_q(\g_0)$-crystal involution satisfying
\begin{equation}
\label{eq:lusztig_involution}
(e_i b)^{\lusztig} = f_{\xi(i)} b^{\lusztig},
\qquad\quad
(f_i b)^{\lusztig} = e_{\xi(i)} b^{\lusztig},
\qquad\quad
\wt(B^{\lusztig}) = w_0 \wt(b),
\end{equation}
where $w_0$ is the long element of the Weyl group of $\g_0$ and $\xi \colon I_0 \to I_0$ is defined by $w_0\clfw_i \mapsto \clfw_{\xi(i)}$ and $w_0 \clsr_i = -\clsr_{\xi(i)}$.
In particular, Lusztig's involution sends the highest weight element to the lowest weight element.
We extend Lusztig's involution to an involution $\lusztig \colon B^{r,s} \to B^{r,s}$ by defining $\xi(0) = 0$ and satisfying Equation~\eqref{eq:lusztig_involution} and sends the maximal element to the \defn{minimal element}, the unique element of weight $-\wt\bigl(u(B^{r,s})\bigr)$.
We also can extend Lusztig's involution to tensor products by a natural isomorphism
\[
(B_2 \otimes B_1)^{\lusztig} \iso B_1^{\lusztig} \otimes B_2^{\lusztig}
\]
given by $(b_2 \otimes b_1)^{\lusztig} = b_1^{\lusztig} \otimes b_2^{\lusztig}$.
Then we define on classically highest weight elements
\[
\KSS{\delta}^{\lusztig}  := \mathbin{\uparrow} \circ \lusztig \circ \KSS{\delta} \circ \lusztig,
\qquad\qquad
\rb := \lusztig \circ \lb \circ \lusztig,
\qquad\qquad
\rs := \lusztig \circ \ls \circ \lusztig,
\]
where $\mathbin{\uparrow}(b)$ is the classically highest weight corresponding to $b$.
By considering $\diamond := \mathbin{\uparrow} \circ \lusztig$, we also have
\[
\KSS{\delta}^{\lusztig}  = \diamond \circ \KSS{\delta} \circ \diamond,
\qquad\qquad
\rb = \diamond \circ \lb \circ \diamond,
\qquad\qquad
\rs = \diamond \circ \ls \circ \diamond.
\]
We then extend these maps as classical crystal isomorphisms.

%=====================================================================
\section{Minuscule \texorpdfstring{$\delta$}{delta} for dual untwisted types}
\label{sec:minuscule}

In this section, we describe the map $\delta$ used to construct $\Phi$ for minuscule fundamental weights when $\g$ is of dual untwisted affine type. More explicitly, we restrict ourselves to simply-laced affine types and types $A_{2n-1}^{(2)}$ and $D_{n+1}^{(2)}$ as types $D_4^{(3)}$ and $E_6^{(2)}$ do not contain any minuscule fundamental weights. Note that for these types, we have $t_a = 1$ for all $a \in I$.

We construct the map $\delta_r \colon B^{r,1} \otimes B^{\bullet} \to B^{\bullet}$, where $B^{\bullet}$ is a tensor product of KR crystals and $\clfw_r$ is a minuscule weight of type $\g_0$ (\textit{i.e.}, $r$ is a special node) as follows. Start at $b_1 = u_{\clfw_r}$, and set $\ell_0 = 1$. Consider step $j$. From $b_j$, let $\ell_j$ denote a minimal $i_a \geq \ell_{j-1}$ ($a \in I_0$ also varies) such that $f_a b_j \neq 0$ and $\nu^{(a)}$ has a singular row of length $i_a$ that has not been previously selected. If no such row exists, terminate, set all $\ell_{j'} = \infty$ for $j' \geq j$, and return $b_j$. Otherwise select such a row in $\nu^{(a)}$ and repeat the above with $b_{j+1} := f_a b_j$.

We form the new rigged configuration by removing a box from each row selected by $\delta_r$, making the resulting rows singular, and keeping all other rows the same.

\begin{example}
\label{ex:minuscule_bijection}
Consider type $D_5^{(1)}$ and $B = B^{5,1} \otimes B^{4,1} \otimes B^{1,1} \otimes B^{5,1}$. See Figure~\ref{fig:B45_type_D5} for the crystal graphs of $B(\clfw_4)$ and $B(\clfw_5)$ and Figure~\ref{fig:B1_type_D5} for the crystal graph of $B(\clfw_1)$. We compute the bijection
\[
\begin{tikzpicture}[scale=.35,baseline=-18]
\begin{scope}[yshift=0cm]
\fill[lightgray] (20,-4) rectangle (21,-3);
\node[scale=.7] at (20.6, -3.5) {$\ell_1$}; % 5
\fill[lightgray] (10,-5) rectangle (11,-4);
\node[scale=.7] at (10.6, -4.5) {$\ell_2$}; % -5,3
\fill[lightgray] (5,-4) rectangle (6,-3);
\node[scale=.7] at (5.6, -3.5) {$\ell_3$}; % -3,2,4
\fill[lightgray] (0,-3) rectangle (1,-2);
\node[scale=.7] at (0.6, -2.5) {$\ell_4$}; % -2,1,4
\rpp{1,1}{0,0}{0,0}
 \begin{scope}[xshift=5cm]
 \rpp{1,1,1}{0,0,0}{0,0,0}
 \end{scope}
 \begin{scope}[xshift=10cm]
 \rpp{1,1,1,1}{0,0,0,0}{0,0,0,0}
 \end{scope}
 \begin{scope}[xshift=15cm]
 \rpp{1,1}{0,0}{1,1}
 \end{scope}
 \begin{scope}[xshift=20cm]
 \rpp{1,1,1}{0,0,0}{0,0,0}
 \end{scope}
\end{scope} % y-shift
%%%
\draw[->] (10.5,-6cm) -- (10.5,-9cm) node[midway,right] {$\delta_5$};
\draw (17,-7.5cm) node {(returns $\bon4$)};
\begin{scope}[yshift=-9cm]
\fill[lightgray] (15,-3) rectangle (16,-2);
\node[scale=.7] at (15.6, -2.5) {$\ell_1$}; % 4
\fill[lightgray] (10,-4) rectangle (11,-3);
\node[scale=.7] at (10.6, -3.5) {$\ell_2$}; % -4,3
\fill[lightgray] (5,-3) rectangle (6,-2);
\node[scale=.7] at (5.6, -2.5) {$\ell_3$}; % -3,2,5
\fill[lightgray] (20,-3) rectangle (21,-2);
\node[scale=.7] at (20.6, -2.5) {$\ell_4$}; % -2,1,5
\fill[lightgray] (10,-3) rectangle (11,-2);
\node[scale=.7] at (10.6, -2.5) {$\ell_5$}; % -2,-5,1,3
\fill[lightgray] (15,-2) rectangle (16,-1);
\node[scale=.7] at (15.6, -1.5) {$\ell_6$}; % -3,1,4
\rpp{1}{0}{1}
 \begin{scope}[xshift=5cm]
 \rpp{1,1}{0,0}{0,0}
 \end{scope}
 \begin{scope}[xshift=10cm]
 \rpp{1,1,1}{0,0,0}{0,0,0}
 \end{scope}
 \begin{scope}[xshift=15cm]
 \rpp{1,1}{0,0}{0,0}
 \end{scope}
 \begin{scope}[xshift=20cm]
 \rpp{1,1}{0,0}{0,0}
 \end{scope}
\end{scope} % y-shift
%%%
\draw[->] (10.5,-14cm) -- (10.5,-17cm) node[midway,right] {$\delta_4$};
\draw (17,-15.5cm) node {(returns $1\bfo$)};
\begin{scope}[yshift=-17cm]
\fill[lightgray] (0,-1) rectangle (1,-2);
\node[scale=.7] at (0.6, -1.5) {$\ell_1$}; % 1
\fill[lightgray] (5,-1) rectangle (6,-2);
\node[scale=.7] at (5.6, -1.5) {$\ell_2$}; % 2
\fill[lightgray] (10,-1) rectangle (11,-2);
\node[scale=.7] at (10.6, -1.5) {$\ell_3$}; % 3
\fill[lightgray] (20,-1) rectangle (21,-2);
\node[scale=.7] at (20.6, -1.5) {$\ell_1$}; % 4
\rpp{1}{0}{0}
 \begin{scope}[xshift=5cm]
 \rpp{1}{0}{0}
 \end{scope}
 \begin{scope}[xshift=10cm]
 \rpp{1}{0}{0}
 \end{scope}
 \begin{scope}[xshift=15cm]
 \node at (0.5,-1.5) {$\emptyset$};
 \end{scope}
 \begin{scope}[xshift=20cm]
 \rpp{1}{0}{0}
 \end{scope}
\end{scope} % y-shift
%%%
\draw[->] (10.5,-20cm) -- (10.5,-23cm) node[midway,right] {$\delta_1$};
\draw (17,-21.5cm) node {(returns $4\bfive$)};
\begin{scope}[yshift=-23cm]
 \node at (0.5,-1.5) {$\emptyset$};
 \begin{scope}[xshift=5cm]
 \node at (0.5,-1.5) {$\emptyset$};
 \end{scope}
 \begin{scope}[xshift=10cm]
 \node at (0.5,-1.5) {$\emptyset$};
 \end{scope}
 \begin{scope}[xshift=15cm]
 \node at (0.5,-1.5) {$\emptyset$};
 \end{scope}
 \begin{scope}[xshift=20cm]
 \node at (0.5,-1.5) {$\emptyset$};
 \end{scope}
\end{scope} % y-shift
%%%
\draw[->] (10.5,-25.5cm) -- (10.5,-28.5cm) node[midway,right] {$\delta_5$};
\draw (17,-27cm) node {(returns $5$)};
\begin{scope}[yshift=-28cm]
 \node at (0.5,-1.5) {$\emptyset$};
 \begin{scope}[xshift=5cm]
 \node at (0.5,-1.5) {$\emptyset$};
 \end{scope}
 \begin{scope}[xshift=10cm]
 \node at (0.5,-1.5) {$\emptyset$};
 \end{scope}
 \begin{scope}[xshift=15cm]
 \node at (0.5,-1.5) {$\emptyset$};
 \end{scope}
 \begin{scope}[xshift=20cm]
 \node at (0.5,-1.5) {$\emptyset$};
 \end{scope}
\end{scope} % y-shift
%%%
\end{tikzpicture}
\]
where at each step, we have labeled the sequence of boxes that are removed under $\delta_r$.
Recall that we are using the notation for minuscule nodes, so for an element $b$, any $a \in b$ (resp.~$\overline{a} \in b$) corresponds to $\varepsilon_a(b) = 1$ (resp.~$\varphi_a(b) = 1$) and is $0$ otherwise.
By using the sequence of returned elements above, we obtain
\[
\begin{tikzpicture}[scale=.35,baseline=-28]
\rpp{1,1}{0,0}{0,0}
 \begin{scope}[xshift=5cm]
 \rpp{1,1,1}{0,0,0}{0,0,0}
 \end{scope}
 \begin{scope}[xshift=10cm]
 \rpp{1,1,1,1}{0,0,0,0}{0,0,0,0}
 \end{scope}
 \begin{scope}[xshift=15cm]
 \rpp{1,1}{0,0}{1,1}
 \end{scope}
 \begin{scope}[xshift=20cm]
 \rpp{1,1,1}{0,0,0}{0,0,0}
 \end{scope}
 \end{tikzpicture}
 \overset{\Phi}{\longmapsto} \bon4 \otimes 1\bfo \otimes 4\bfive \otimes 5.
\]
\end{example}

\begin{figure}
\[
\begin{tikzpicture}[>=latex,xscale=1.4,yscale=1.1]
\node (4) at (2,2) {$4$};
\node (3b4) at (2,1) {$3\bfo$};
\node (2b35) at (2,0) {$2\bth5$};
\node (2b5) at (3,0) {$2\bfive$};
\node (1b25) at (2,-1) {$1\btw5$};
\node (1b23b5) at (3,-1) {$1\btw3\bfive$};
\node (1b34) at (4,-1) {$1\bth4$};
\node (1b4) at (5,-1) {$1\bfo$};
\node (b15) at (2,-2) {$\bon5$};
\node (b13b5) at (3,-2) {$\bon3\bfive$};
\node (b12b34) at (4,-2) {$\bon2\bth4$};
\node (b12b4) at (5,-2) {$\bon2\bfo$};
\node (b24) at (4,-3) {$\btw4$};
\node (b23b4) at (5,-3) {$\btw3\bfo$};
\node (b35) at (5,-4) {$\bth5$};
\node (b5) at (5,-5) {$\bfive$};
\draw[->,black] (4) -- (3b4) node[midway,right] {\scriptsize $4$};
\draw[->,dgreencolor] (3b4) -- (2b35) node[midway,right] {\scriptsize $3$};
\draw[->,UQpurple] (2b35) -- (2b5) node[midway,above] {\scriptsize $5$};
\draw[->,blue] (2b35) -- (1b25) node[midway,right] {\scriptsize $2$};
\draw[->,blue] (2b5) -- (1b23b5) node[midway,right] {\scriptsize $2$};
\draw[->,UQpurple] (1b25) -- (1b23b5) node[midway,above] {\scriptsize $5$};
\draw[->,dgreencolor] (1b23b5) -- (1b34) node[midway,above] {\scriptsize $3$};
\draw[->,black] (1b34) -- (1b4) node[midway,above] {\scriptsize $4$};
\draw[->,red] (1b25) -- (b15) node[midway,right] {\scriptsize $1$};
\draw[->,red] (1b23b5) -- (b13b5) node[midway,right] {\scriptsize $1$};
\draw[->,red] (1b34) -- (b12b34) node[midway,right] {\scriptsize $1$};
\draw[->,red] (1b4) -- (b12b4) node[midway,right] {\scriptsize $1$};
\draw[->,UQpurple] (b15) -- (b13b5) node[midway,above] {\scriptsize $5$};
\draw[->,dgreencolor] (b13b5) -- (b12b34) node[midway,above] {\scriptsize $3$};
\draw[->,black] (b12b34) -- (b12b4) node[midway,above] {\scriptsize $4$};
\draw[->,blue] (b12b34) -- (b24) node[midway,right] {\scriptsize $2$};
\draw[->,blue] (b12b4) -- (b23b4) node[midway,right] {\scriptsize $2$};
\draw[->,black] (b24) -- (b23b4) node[midway,above] {\scriptsize $4$};
\draw[->,dgreencolor] (b23b4) -- (b35) node[midway,right] {\scriptsize $3$};
\draw[->,UQpurple] (b35) -- (b5) node[midway,right] {\scriptsize $5$};
\end{tikzpicture}
\qquad\qquad
\begin{tikzpicture}[>=latex,xscale=1.4,yscale=1.1]
\node (5) at (2,2) {$5$};
\node (3b5) at (2,1) {$3\bfive$};
\node (2b34) at (2,0) {$2\bth4$};
\node (2b4) at (3,0) {$2\bfo$};
\node (1b24) at (2,-1) {$1\btw4$};
\node (1b23b4) at (3,-1) {$1\btw3\bfo$};
\node (1b35) at (4,-1) {$1\bth5$};
\node (1b5) at (5,-1) {$1\bfive$};
\node (b14) at (2,-2) {$\bon4$};
\node (b13b4) at (3,-2) {$\bon3\bfo$};
\node (b12b35) at (4,-2) {$\bon2\bth5$};
\node (b12b5) at (5,-2) {$\bon2\bfive$};
\node (b25) at (4,-3) {$\btw5$};
\node (b23b5) at (5,-3) {$\btw3\bfive$};
\node (b34) at (5,-4) {$\bth4$};
\node (b4) at (5,-5) {$\bfo$};
\draw[->,UQpurple] (5) -- (3b5) node[midway,right] {\scriptsize $5$};
\draw[->,dgreencolor] (3b5) -- (2b34) node[midway,right] {\scriptsize $3$};
\draw[->,black] (2b34) -- (2b4) node[midway,above] {\scriptsize $4$};
\draw[->,blue] (2b34) -- (1b24) node[midway,right] {\scriptsize $2$};
\draw[->,blue] (2b4) -- (1b23b4) node[midway,right] {\scriptsize $2$};
\draw[->,black] (1b24) -- (1b23b4) node[midway,above] {\scriptsize $4$};
\draw[->,dgreencolor] (1b23b4) -- (1b35) node[midway,above] {\scriptsize $3$};
\draw[->,UQpurple] (1b35) -- (1b5) node[midway,above] {\scriptsize $5$};
\draw[->,red] (1b24) -- (b14) node[midway,right] {\scriptsize $1$};
\draw[->,red] (1b23b4) -- (b13b4) node[midway,right] {\scriptsize $1$};
\draw[->,red] (1b35) -- (b12b35) node[midway,right] {\scriptsize $1$};
\draw[->,red] (1b5) -- (b12b5) node[midway,right] {\scriptsize $1$};
\draw[->,black] (b14) -- (b13b4) node[midway,above] {\scriptsize $4$};
\draw[->,dgreencolor] (b13b4) -- (b12b35) node[midway,above] {\scriptsize $3$};
\draw[->,UQpurple] (b12b35) -- (b12b5) node[midway,above] {\scriptsize $5$};
\draw[->,blue] (b12b35) -- (b25) node[midway,right] {\scriptsize $2$};
\draw[->,blue] (b12b5) -- (b23b5) node[midway,right] {\scriptsize $2$};
\draw[->,UQpurple] (b25) -- (b23b5) node[midway,above] {\scriptsize $5$};
\draw[->,dgreencolor] (b23b5) -- (b34) node[midway,right] {\scriptsize $3$};
\draw[->,black] (b34) -- (b4) node[midway,right] {\scriptsize $4$};
\end{tikzpicture}
\]
\caption{The crystals $B(\clfw_4)$ (left) and $B(\clfw_5)$ (right) in type $D_5$.}
\label{fig:B45_type_D5}
\end{figure}

%=====================================================================
\section{Adjoint \texorpdfstring{$\delta$}{delta} for dual untwisted types}
\label{sec:adjoint}

In this section, we describe the map $\delta_{\theta} := \delta_{N_{\g}}$ for the adjoint node $N_{\g}$ for dual untwisted types (\textit{i.e.}, $t_a = 1$ for all $a \in I$ or equivalently, $\g$ is of simply-laced affine type, $A_{2n-1}^{(2)}$, $D_{n+1}^{(2)}$, $D_4^{(3)}$, $E_6^{(2)}$). Furthermore, we give a uniform proof that $\Phi$ is a statistic preserving bijection.

We define the map $\delta_{\theta} \colon B^{\theta, 1} \otimes B^{\bullet} \to B^{\bullet}$, where $B^{\bullet}$ is a tensor product of KR crystals, by the following algorithm. Begin with $r_1 = u_{\theta}$ being the highest weight element in $B(\theta) \subseteq B^{\theta,1}$, and set $\ell_0 = 1$. Consider step $j$ such that $r_j = x_{\beta}$, where $\beta > 0$ and $\beta \neq \alpha_a$ for all $a \in I_0$. From $r_j$, consider any outgoing arrow labeled by $a$ and find the minimal $i_a \geq \ell_{j-1}$ such that $\nu^{(a)}$ has a singular row of length $i_a$ which has not been previously selected. If no such row exists, terminate, set all $\ell_{j'} = \infty$ for $j' \geq j$ and $\ellbar_{j'} = \infty$ for all $j'$, and return $r_j$. Otherwise select such a row, set $\ell_j = \min_a i_a$, and repeat the above with $r_{j+1} := f_{a'} r_j$, where $a'$ is such that $i_{a'} = \min_a i_a$. If $r_j = x_{\alpha_a}$ for some $a \in I_0$, we do one of the following disjoint cases. We discard all previously selected (singular) rows.
\begin{itemize}
\item[(S)] If there exists a singular row of length $i_a \geq \max\{\ell_{j-1}, 2\}$,\footnote{Note that if $\ell_{j-1} = 1$ and there exists a singular row of length $1$, then we would not be in this case as $i_a = 1 < \max\{\ell_{j-1},2\} = 2$.} select such a row and set $\ell_j = i_a$.
\item[(E)] If there exists a singular row of length $1$ and $\ell_{j-1} = 1$, we terminate, set $\ell_j = 1$ and $\ellbar_{j'} = \infty$ for all $j'$, and return $\emptyset$.
\item[(Q)] If there exists a quasisingular row of length $i_a \geq \ell_{j-1}$, we select the quasisingular string and set $\ell_j = i_a$.
\item[(T)] Otherwise we terminate, set $\ell_j = \ellbar_{j'} = \infty$ for all $j'$, and return $x_{\alpha_a}$.
\end{itemize}
If the process has not terminated, set $r_{j+1} := y_a$ and perform the following. Let $\ellbar_0 = \ell_h$, where $h = \sum_{a \in I_0} c_a$, \textit{i.e.}, the height of $\theta$ or the number of steps we currently have done. Consider step $j$, and consider any outgoing arrow labeled by $a$ from $r_j$. Find the minimal $i_a \geq \ellbar_{j-1}$ such that $\nu^{(a)}$ has a singular row of length $i_a$ such that
\begin{itemize}
\item[(D)] it had been selected at step $j'$ with $\ell_{j'} = i_a$ or
\item[(N)] it had not been previously selected and Case~(D) does not occur.
\end{itemize}
If no such row exists, terminate, set all $\ellbar_{j'} = \infty$ for $j' \geq j$, and return $r_j$. Otherwise select such a row, set $\ellbar_j = \min_a i_a$, redefine $\ell_{j'} := \ell_{j'} - 1$ if Case~(D) had occurred, and repeat the above with $r_{j+1} := f_{a'} r_j$, where $a'$ is such that $i_{a'} = \min_a i_a$.

We form the new rigged configuration by
\begin{itemize}
\item[(1)] removing a box from each row each time it was selected by $\delta$ (\textit{i.e.}, if Case~(D) occurred, then we remove 2 boxes);
\item[(2)] making the resulting rows singular unless Case~(Q) occurred, then we make the row selected by $\ellbar_1$ (if $\ellbar_1 \neq \infty$) quasisingular; and
\item[(3)] keeping all other rows the same.
\end{itemize}

Note that the same row cannot be selected twice by Case~(D) due to the redefinition of $\ell_{j'}$. We clearly cannot have more than $x_{\alpha_a}$ in this process since $\alpha_a$ is a simple root and hence has no directed path between them by the crystal axioms.

\begin{remark}
This is the (conjectural) map $\delta$ of bin Mohammad~\cite{Mahathir} for $\g$ of type $E_6^{(1)}$.
Moreover, this was the map $\KSS{\delta}$ for $\g$ of type $D_{n+1}^{(2)}$ in~\cite{OSS03} and of type $D_4^{(3)}$ in~\cite{Scrimshaw15}.
\end{remark}

\begin{remark}
\label{rem:adjoint_extensions}
We can extend this description for types $C_n^{(1)}$, $A_{2n}^{(2)}$, and $A_{2n}^{(2)\dagger}$.
Indeed, since $B^{\theta,1}$ for type $A_{2n}^{(2)}$ (resp., type $A_{2n}^{(2)\dagger}$) does not contain any elements $y_a$, for $a \in I_0$ (resp., $a = 0$), as noted in Remark~\ref{remark:notes_adjoint_crystal} (resp.\ Remark~\ref{remark:notes_A2dual}), we modify the definition of $\delta_{\theta}$ by removing Case~(Q) (resp., Case~(E)) as a possibility. Likewise for type $C_n^{(1)}$, we do not have $y_a$ for all $a \in I$, so we modify the definition of $\delta_{\theta}$ by removing both Case~(Q) and Case~(E) and combine Case~(S) with the first Case~(D) (think of performing these steps simultaneously to do $x_{\alpha_1} \xrightarrow[\hspace{20pt}]{1} x_{-\alpha_1}$), but we also need to consider the parts of $\nu^{(n)}$ doubled as per Remark~\ref{remark:conventions}.
\end{remark}

We have the following classification of elements in $B(\theta)$ and will be used to describe the KR tableaux of type $E_8^{(1)}$ and $E_6^{(2)}$.

\begin{proposition}
\label{prop:adjoint_elements}
Let $\g$ be of simply-laced or twisted type, and fix some $b \in B(\theta)$. Then $b$ has the following properties.
\begin{itemize}
\item $b$ is uniquely determined by $\varepsilon$ and $\varphi$.
\item $\wt(b) = 0$ if and only if there exists a unique $i \in I_0$ such that $\varepsilon_i(b) = \varphi_i(b) = 1$ and $\varepsilon_j(b) = \varphi_j(b) = 0$ for all $j \neq i$.
\item $\varepsilon_i(b) = 2$ implies $\varepsilon_j(b) = 0$ for all $j \neq i$.
\item $\varphi_i(b) = 2$ implies $\varphi_j(b) = 0$ for all $j \neq i$.
\end{itemize}
\end{proposition}

\begin{proof}
This follows from the description of $B(\theta)$.
\end{proof}

Thus, similar to types $E_{6,7}$, we can equate our elements in $B(\theta)$ by multisets of $\{1, \bon, 2, \btw, \dotsc, n, \bn\}$, which as above, we write as words.

\begin{example}
\label{ex:E6t_B11_power}
Consider type $E_6^{(2)}$ and $B = (B^{1,1})^{\otimes 4}$. See Figure~\ref{fig:B4_type_F4} for the corresponding classical crystal $B(\clfw_1)$. We compute the bijection
\[
\begin{tikzpicture}[scale=.35,baseline=-18]
\begin{scope}[yshift=0cm]
\fill[lightgray] (0,-4) rectangle (1,-3);
\node[scale=.7] at (0.6, -3.5) {$\ell_1$}; % -1,2
\fill[lightgray] (5,-7) rectangle (6,-6);
\node[scale=.7] at (5.6, -6.5) {$\ell_2$}; % -2,3
\fill[lightgray] (10,-5) rectangle (11,-4);
\node[scale=.7] at (10.6, -4.5) {$\ell_3$}; % 2,-3,4
\fill[lightgray] (5,-6) rectangle (6,-5);
\node[scale=.7] at (5.6, -5.5) {$\ell_4$}; % 1,-2,4
\fill[lightgray] (16,-3) rectangle (17,-2);
\node[scale=.7] at (16.6, -2.5) {$\ell_5$}; % 1,-2,3,-4
\fill[lightgray] (10,-4) rectangle (11,-3);
\node[scale=.7] at (10.6, -3.5) {$\ell_6$}; % 1,2,-3
\fill[lightgray] (5,-5) rectangle (6,-4);
\node[scale=.7] at (5.6, -4.5) {$\ell_7$}; % 1,1,-2
\fill[lightgray] (1,-2) rectangle (2,-1);
\node[scale=.7] at (1.6, -1.5) {$\ell_9$}; % 1,-1
\fill[lightgray] (0,-4) rectangle (1,-5);
\node[scale=.7] at (0.6, -4.5) {$\ell_8$}; % -1,-1,2
\fill[lightgray] (6,-3) rectangle (7,-4);
\node[scale=.7] at (6.6, -3.5) {$\ell_{10}$}; % -1,-2,3
\fill[lightgray] (11,-2) rectangle (12,-3);
\node[scale=.7] at (11.6, -2.5) {$\ell_{11}$}; % -1,2,-3,4
\fill[lightgray] (17,-2) rectangle (18,-1);
\node[scale=.7] at (17.6, -1.5) {$\ell_{12}$}; % -1,2,-4
\fill[lightgray] (6,-2) rectangle (7,-3);
\node[scale=.7] at (6.6, -2.5) {$\ell_{13}$}; % -2,3,-4
\fill[lightgray] (11,-1) rectangle (12,-2);
\node[scale=.7] at (11.6, -1.5) {$\ell_{14}$}; % 2,-3
\fill[lightgray] (6,-1) rectangle (7,-2);
\node[scale=.7] at (6.6, -1.5) {$\ell_{15}$}; % 1,-2
 \rpp{2,2,1,1}{1,0,2,1}{1,1,2,2}
 \begin{scope}[xshift=5cm]
 \rpp{2,2,2,1,1,1}{0,0,0,0,0,0}{0,0,0,0,0,0}
 \end{scope}
 \begin{scope}[xshift=10cm]
 \rpp{2,2,1,1}{0,0,0,0}{0,0,0,0}
 \end{scope}
 \begin{scope}[xshift=16cm]
 \rpp{2,1}{0,0}{0,0}
 \end{scope}
\end{scope} % y-shift
%%%
\draw[->] (7.5,-8cm) -- (7.5,-11cm) node[midway,right] {$\delta_{\theta}$};
\draw (13,-9.5cm) node {(returns $1\btw$)};
\begin{scope}[yshift=-11cm]
\fill[lightgray] (1,-1) rectangle (2,-2);
\node[scale=.7] at (1.6, -1.5) {$\ell_1$}; % -1,2
 \rpp{2,1}{0,1}{0,2}
 \begin{scope}[xshift=5cm]
 \rpp{1,1,1}{0,0,0}{0,0,0}
 \end{scope}
 \begin{scope}[xshift=10cm]
 \rpp{1,1}{0,0}{0,0}
 \end{scope}
 \begin{scope}[xshift=16cm]
 \rpp{1}{0}{0}
 \end{scope}
\end{scope} % y-shift
%%%
\draw[->] (7.5,-16cm) -- (7.5,-19cm) node[midway,right] {$\delta_{\theta}$};
\draw (13,-17.5cm) node {(returns $\bon2$)};
\begin{scope}[yshift=-19cm]
\fill[lightgray] (0,-3) rectangle (1,-2);
\node[scale=.7] at (0.6, -2.5) {$\ell_1$}; % -1,2
\fill[lightgray] (5,-4) rectangle (6,-3);
\node[scale=.7] at (5.6, -3.5) {$\ell_2$}; % -2,3
\fill[lightgray] (10,-2) rectangle (11,-3);
\node[scale=.7] at (10.6, -2.5) {$\ell_3$}; % 2,-3,4
\fill[lightgray] (16,-2) rectangle (17,-1);
\node[scale=.7] at (16.6, -1.5) {$\ell_4$}; % 2,-4
\fill[lightgray] (5,-3) rectangle (6,-2);
\node[scale=.7] at (5.6, -2.5) {$\ell_5$}; % 1,-2,3,-4
\fill[lightgray] (10,-1) rectangle (11,-2);
\node[scale=.7] at (10.6, -1.5) {$\ell_6$}; % 1,2,-3
\fill[lightgray] (5,-1) rectangle (6,-2);
\node[scale=.7] at (5.6, -1.5) {$\ell_7$}; % 1,1,-2
\fill[lightgray] (0,-1) rectangle (1,-2);
\node[scale=.7] at (0.6, -1.5) {$\ell_8$}; % E
 \rpp{1,1}{1,1}{1,1}
 \begin{scope}[xshift=5cm]
 \rpp{1,1,1}{0,0,0}{0,0,0}
 \end{scope}
 \begin{scope}[xshift=10cm]
 \rpp{1,1}{0,0}{0,0}
 \end{scope}
 \begin{scope}[xshift=16cm]
 \rpp{1}{0}{0}
 \end{scope}
\end{scope} % y-shift
%%%
\draw[->] (7.5,-24cm) -- (7.5,-27cm) node[midway,right] {$\delta_{\theta}$};
\draw (13,-25.5cm) node {(returns $\emptyset$)};
\begin{scope}[yshift=-27cm]
 \node at (0,-1.5) {$\emptyset$};
 \begin{scope}[xshift=5cm]
 \node at (0,-1.5) {$\emptyset$};
 \end{scope}
 \begin{scope}[xshift=10cm]
 \node at (0,-1.5) {$\emptyset$};
 \end{scope}
 \begin{scope}[xshift=16cm]
 \node at (0,-1.5) {$\emptyset$};
 \end{scope}
\end{scope} % y-shift
%%%
\end{tikzpicture}
\]
and the final application of $\delta_{\theta}$ returns $1$. As with Example~\ref{ex:minuscule_bijection}, we label the sequence of boxes removed under $\delta_{\theta}$, but in our labeling here, we have $\ell_{k-h} = \ellbar_k$ for all $k \geq h$. Note that for the first (resp.\ third) application of $\delta_{\theta}$, we used Case~(Q) (resp.\ Case~(E)) when at $x_{\alpha_1} = \overline{2}11$. Hence, the result of applying $\Phi$ is the element
\[
1\btw \otimes \bon2 \otimes \emptyset \otimes 1.
\]
\end{example}

\begin{figure}
\[
\begin{tikzpicture}[>=latex,xscale=1.4,yscale=1.1]
\node (4) at (2,2) {$1$};
\node (3b4) at (2,1) {$\bon2$};
\node (2b3) at (2,0) {$3\btw$};
\node (1b23) at (3,0) {$4\bth2$};
\node (1b34) at (4,0) {$4\btw1$};
\node (1b4) at (5,0) {$4\bon$};
\node (b13) at (3,-1) {$\bfo2$};
\node (b12b34) at (4,-1) {$\bfo3\btw\bon$};
\node (b12b4) at (5,-1) {$\bfo3\bon$};
\node (b234) at (4,-2) {$\bth21$};
\node (b233b4) at (5,-2) {$\bth22\bon$};
\node (b344) at (4,-3) {$\btw11$};
\node (b44) at (5,-3) {$\bon1$};
\node (b33) at (6,-2) {$\btw2$};
\node (3b4b4) at (6,-3) {$2\bon\bon$};
\node (2b3b34) at (7,-2) {$3\btw\btw1$};
\node (2b3b4) at (7,-3) {$3\btw\bon$};
\node (1b24) at (8,-2) {$4\bth1$};
\node (1b23b4) at (8,-3) {$4\bth2\bon$};
\node (b14) at (9,-2) {$\bfo1$};
\node (b13b4) at (9,-3) {$\bfo2\bon$};
\node (1b3) at (8,-4) {$4\btw$};
\node (b12b3) at (9,-4) {$\bfo3\btw$};
\node (b23) at (10,-4) {$\bth2$};
\node (b34) at (10,-5) {$\btw1$};
\node (b4) at (10,-6)  {$\bon$};
\draw[->,red] (4) -- (3b4) node[midway,right] {\scriptsize $1$};
\draw[->,blue] (3b4) -- (2b3) node[midway,right] {\scriptsize $2$};
\draw[->,dgreencolor] (2b3) -- (1b23) node[midway,above] {\scriptsize $3$};
\draw[->,blue] (1b23) -- (1b34) node[midway,above] {\scriptsize $2$};
\draw[->,red] (1b34) -- (1b4) node[midway,above] {\scriptsize $1$};
\draw[->,black] (1b23) -- (b13) node[midway,right] {\scriptsize $4$};
\draw[->,black] (1b34) -- (b12b34) node[midway,right] {\scriptsize $4$};
\draw[->,black] (1b4) -- (b12b4) node[midway,right] {\scriptsize $4$};
\draw[->,blue] (b13) -- (b12b34) node[midway,above] {\scriptsize $2$};
\draw[->,red] (b12b34) -- (b12b4) node[midway,above] {\scriptsize $1$};
\draw[->,dgreencolor] (b12b34) -- (b234) node[midway,right] {\scriptsize $3$};
\draw[->,dgreencolor] (b12b4) -- (b233b4) node[midway,right] {\scriptsize $3$};
\draw[->,red] (b234) -- (b233b4) node[midway,above] {\scriptsize $1$};
\draw[->,blue] (b234) -- (b344) node[midway,right] {\scriptsize $2$};
\draw[->,red] (b344) -- (b44) node[midway,above] {\scriptsize $1$};
\draw[->,red] (b44) -- (3b4b4) node[midway,above] {\scriptsize $1$};
\draw[->,blue] (b233b4) -- (b33) node[midway,above] {\scriptsize $2$};
\draw[->,blue] (b33) -- (2b3b34) node[midway,above] {\scriptsize $2$};
\draw[->,red] (2b3b34) -- (2b3b4) node[midway,right] {\scriptsize $1$};
\draw[->,blue] (3b4b4) -- (2b3b4) node[midway,above] {\scriptsize $2$};
\draw[->,dgreencolor] (2b3b34) -- (1b24) node[midway,above] {\scriptsize $3$};
\draw[->,dgreencolor] (2b3b4) -- (1b23b4) node[midway,above] {\scriptsize $3$};
\draw[->,red] (1b24) -- (1b23b4) node[midway,right] {\scriptsize $1$};
\draw[->,red] (b14) -- (b13b4) node[midway,right] {\scriptsize $1$};
\draw[->,black] (1b24) -- (b14) node[midway,above] {\scriptsize $4$};
\draw[->,black] (1b23b4) -- (b13b4) node[midway,above] {\scriptsize $4$};
\draw[->,blue] (1b23b4) -- (1b3) node[midway,right] {\scriptsize $2$};
\draw[->,blue] (b13b4) -- (b12b3) node[midway,right] {\scriptsize $2$};
\draw[->,black] (1b3) -- (b12b3) node[midway,above] {\scriptsize $4$};
\draw[->,dgreencolor] (b12b3) -- (b23) node[midway,above] {\scriptsize $3$};
\draw[->,blue] (b23) -- (b34) node[midway,right] {\scriptsize $2$};
\draw[->,red] (b34) -- (b4) node[midway,right] {\scriptsize $1$};
\end{tikzpicture}
\]
\caption{The crystal graph of $B(\clfw_1)$ in the dual of type $F_4$ (\textit{i.e.}, the usual labeling of $F_4$ has become $i \leftrightarrow 5-i$) used in constructing $B^{1,1}$ in type $E_6^{(2)}$.}
\label{fig:B4_type_F4}
\end{figure}

\begin{example}
Consider type $E_6^{(2)}$ and $B = (B^{1,1})^{\otimes 3}$. We compute the bijection
\[
\begin{tikzpicture}[scale=.35,baseline=-18]
\begin{scope}[yshift=0cm]
 \fill[lightgray] (1,-1) rectangle (2,-2);
 \node[scale=.7] at (1.6, -1.5) {$\ell_1$};
 \fill[lightgray] (1,-2) rectangle (2,-3);
 \node[scale=.7] at (1.6, -2.5) {$\ell_8$};
 \rpp{2,2}{1,0}{1,1}
 \begin{scope}[xshift=5cm]
 \fill[lightgray] (1,-3) rectangle (2,-4);
 \node[scale=.7] at (1.6, -3.5) {$\ell_2$};
 \fill[lightgray] (1,-2) rectangle (2,-3);
 \node[scale=.7] at (1.6, -2.5) {$\ell_5$};
 \fill[lightgray] (1,-1) rectangle (2,-2);
 \node[scale=.7] at (1.6, -1.5) {$\ell_7$};
 \rpp{2,2,2}{0,0,0}{0,0,0}
 \end{scope}
 \begin{scope}[xshift=10cm]
 \fill[lightgray] (1,-2) rectangle (2,-3);
 \node[scale=.7] at (1.6, -2.5) {$\ell_3$};
 \fill[lightgray] (1,-1) rectangle (2,-2);
 \node[scale=.7] at (1.6, -1.5) {$\ell_6$};
 \rpp{2,2}{0,0}{0,0}
 \end{scope}
 \begin{scope}[xshift=16cm]
 \fill[lightgray] (1,-1) rectangle (2,-2);
 \node[scale=.7] at (1.6, -1.5) {$\ell_4$};
 \rpp{2}{0}{0}
 \end{scope}
\end{scope} % y-shift
%%%
\draw[->] (7.5,-5cm) -- (7.5,-8cm) node[midway,right] {$\delta_{\theta}$};
\draw (13,-6.5cm) node {(returns $\bon1$)};
\begin{scope}[yshift=-8cm]
 \rpp{1,1}{1,0}{1,1}
 \begin{scope}[xshift=5cm]
 \rpp{1,1,1}{0,0,0}{0,0,0}
 \end{scope}
 \begin{scope}[xshift=10cm]
 \rpp{1,1}{0,0}{0,0}
 \end{scope}
 \begin{scope}[xshift=16cm]
 \rpp{1}{0}{0}
 \end{scope}
\end{scope} % y-shift
%%%
\end{tikzpicture}
\]
the second application of $\delta_{\theta}$ is similar and also returns $\bon1$ and the final returns $1$. Note that in the examples above we are in Case~(Q) when performing $\delta_{\theta}$ as we disregarded the previously selected singular row in $\nu^{(1)}$ (as in Example~\ref{ex:E6t_B11_power}). Hence, the result of applying $\Phi$ is the element
\[
\bon1 \otimes \bon1 \otimes 1.
\]
\end{example}

\begin{example}
Consider type $E_6^{(2)}$ and $B = (B^{1,1})^{\otimes 3}$. We compute the bijection
\[
\begin{tikzpicture}[scale=.35,baseline=-18]
\begin{scope}[yshift=0cm]
 \fill[lightgray] (1,-1) rectangle (2,-2);
 \node[scale=.7] at (1.6, -1.5) {$\ell_1$};
 \fill[lightgray] (1,-2) rectangle (2,-3);
 \node[scale=.7] at (1.6, -2.5) {$\ell_6$};
 \fill[lightgray] (0,-2) rectangle (1,-3);
 \node[scale=.7] at (0.6, -2.5) {$\ell_{12}$};
 \fill[lightgray] (0,-1) rectangle (1,-2);
 \node[scale=.7] at (0.6, -1.5) {$\ell_{16}$};
 \rpp{2,2}{1,1}{1,1}
 \begin{scope}[xshift=5cm]
 \fill[lightgray] (1,-3) rectangle (2,-4);
 \node[scale=.7] at (1.6, -3.5) {$\ell_2$};
 \fill[lightgray] (1,-2) rectangle (2,-3);
 \node[scale=.7] at (1.6, -2.5) {$\ell_5$};
 \fill[lightgray] (1,-1) rectangle (2,-2);
 \node[scale=.7] at (1.6, -1.5) {$\ell_8$};
 \fill[lightgray] (0,-1) rectangle (1,-2);
 \node[scale=.7] at (0.6, -1.5) {$\ell_9$};
 \fill[lightgray] (0,-2) rectangle (1,-3);
 \node[scale=.7] at (0.6, -2.5) {$\ell_{13}$};
 \fill[lightgray] (0,-3) rectangle (1,-4);
 \node[scale=.7] at (0.6, -3.5) {$\ell_{15}$};
 \rpp{2,2,2}{0,0,0}{0,0,0}
 \end{scope}
 \begin{scope}[xshift=10cm]
 \fill[lightgray] (1,-2) rectangle (2,-3);
 \node[scale=.7] at (1.6, -2.5) {$\ell_3$};
 \fill[lightgray] (1,-1) rectangle (2,-2);
 \node[scale=.7] at (1.6, -1.5) {$\ell_7$};
 \fill[lightgray] (0,-1) rectangle (1,-2);
 \node[scale=.7] at (0.6, -1.5) {$\ell_{10}$};
 \fill[lightgray] (0,-2) rectangle (1,-3);
 \node[scale=.7] at (0.6, -2.5) {$\ell_{14}$};
 \rpp{2,2}{0,0}{0,0}
 \end{scope}
 \begin{scope}[xshift=16cm]
 \fill[lightgray] (1,-1) rectangle (2,-2);
 \node[scale=.7] at (1.6, -1.5) {$\ell_4$};
 \fill[lightgray] (0,-1) rectangle (1,-2);
 \node[scale=.7] at (0.6, -1.5) {$\ell_{11}$};
 \rpp{2}{0}{0}
 \end{scope}
\end{scope} % y-shift
%%%
\draw[->] (7.5,-5cm) -- (7.5,-8cm) node[midway,right] {$\delta_{\theta}$};
\draw (13,-6.5cm) node {(returns $\bon$)};
\begin{scope}[yshift=-8cm]
 \node at (0,-1.5) {$\emptyset$};
 \begin{scope}[xshift=5cm]
 \node at (0,-1.5) {$\emptyset$};
 \end{scope}
 \begin{scope}[xshift=10cm]
 \node at (0,-1.5) {$\emptyset$};
 \end{scope}
 \begin{scope}[xshift=16cm]
 \node at (0,-1.5) {$\emptyset$};
 \end{scope}
\end{scope} % y-shift
%%%
\end{tikzpicture}
\]
and the last two applications of $\delta_{\theta}$ return $1$. In this example, we are in Case~(S) when at $x_{\alpha_2} = \bth22\bon$ and then the remaining strings are selected according to Case~(D). Hence, the result of applying $\Phi$ is the element
\[
\bon \otimes 1 \otimes 1.
\]
\end{example}

%=====================================================================
\section{Extending the left-box map}
\label{sec:box_map}

In this section, we describe a generalization of the left-box map in order to give a tableau description of the crystals $B^{r,1}$ for dual untwisted types. To do so, we first construct \defn{$\lb$-diagrams}, which are digraphs on $I_0$ such that
\begin{itemize}
\item every node has at most one outgoing edge,
\item there is a unique sink $\sigma$, and
\item each arrow $r \xrightarrow[\hspace{20pt}]{b} r'$ is labeled by $b \in B(\clfw_{\sigma})$ such that $\varepsilon_a(b) = \delta_{ar'}$ and $\varphi_a(b) = \delta_{ar}$.
%\item for every path in the diagram, the corresponding elements form a path in $B(\clfw_k)$.
\end{itemize}

For a fixed $\lb$-diagram $D$, we define the \defn{left-box} map on rigged configurations $\lb \colon \RC(B^{r,1}) \to \RC(B^{\sigma,1} \otimes B^{r',1})$, where we have the arrow $r \xrightarrow[\hspace{20pt}]{b} r'$ in $D$ as follows. Let $e_{a_1} e_{a_2} \cdots e_{a_m} b = u_{\clfw_{\sigma}}$, and define $\lb(\nu, J)$ as the rigged configuration obtained by adding a singular row of length $1$ to $\nu^{(a_i)}$, for all $1 \leq i \leq m$. By weight considerations, the map is well-defined since the result is independent of the order of the path from $b$ to the highest weight. Note that we can consider $\delta \circ \lb$ to be the same procedure as $\delta$ except starting at $b$.

Next, we define $\lb$ on $B^{r,1}$ by requiring that the diagram
\[
\xymatrixrowsep{3pc}
\xymatrixcolsep{3.5pc}
\xymatrix{\RC(B^{r,1}) \ar[r]^-{\lb} \ar[d]_{\Phi} & \RC(B^{\sigma,1} \otimes B^{r',1}) \ar[d]^{\Phi}
\\ B^{r,1} \ar[r]_-{\lb} & B^{\sigma,1} \otimes B^{r',1}}
\]
commutes, where again $\sigma$ is the unique sink in the $\lb$-diagram. In particular, we note that $\lb$ is a strict $U_q(\g_0)$-crystal embedding. Therefore, we define \defn{Kirillov--Reshetikhin (KR) tableaux} as the tableaux given by iterating the $\lb$ map, where the entries elements in $B^{\sigma,1}$ and the classical crystal structure is induced by the reverse column reading word. See Appendix~\ref{sec:KR_tableaux} for the description of $B^{r,1}$ in types $E_{6,7,8}^{(1)}$ and $E_6^{(2)}$.

For example, consider for $\lb \colon B^{r,1} \to B^{\sigma,1} \otimes B^{\sigma,1}$ corresponding to the arrow $r \xrightarrow[\hspace{10pt}]{b} \sigma$, we can use this to construct the tableau $\young(x,y)$, where $x, y \in B^{\sigma,1}$, given by its image under $\lb$, which is $\young(y) \otimes \young(x)$. We also note that the construction of the KR tableaux is dependent upon the choice of $\lb$-diagram.

We then extend the left-box map to $\lb \colon B^{r,1} \otimes B^{\bullet} \to B^{\sigma,1} \otimes B^{r',1} \otimes B^{\bullet}$, with respect to the $\lb$-diagram $D$, as the strict $U_q(\g_0)$-crystal embedding given by $b \otimes b^{\bullet} \mapsto \lb(b) \otimes b^{\bullet}$.

We note that this is a generalization of the $\lb$ map for the KSS bijection. Specifically, for the non-exceptional types, the defining $\lb$-diagram is
\begin{equation}
\label{eq:lb_general}
\begin{tikzpicture}[baseline=-4, xscale=1.3]
\node (1) at (0,0) {$1$};
\node (2) at (2,0) {$2$};
\node (3) at (4,0) {$\cdots$};
\node (4) at (6,0) {$n-1$};
\node (5) at (8,0) {$n$};
\draw[->] (2) -- (1) node[midway, above]{\small $\boxed{1}$};
\draw[->] (3) -- (2) node[midway, above]{\small $\boxed{2}$};
\draw[->] (4) -- (3) node[midway, above]{\small $\boxed{n-2}$};
\draw[->] (5) -- (4) node[midway, above]{\small $\boxed{n-1}$};
\end{tikzpicture}.
\end{equation}

For type $E_6^{(1)}$, we use the $\lb$-diagram
\begin{equation}
\label{eq:lb_E6}
\begin{tikzpicture}[baseline=-4]
\node (1) at (0,0) {$1$};
\node (2) at (4,1) {$2$};
\node (3) at (2,0) {$3$};
\node (4) at (4,0) {$4$};
\node (5) at (6,1) {$5$};
\node (6) at (2,1) {$6$};
\draw[->] (3) -- (1) node[midway, below]{\small $\bon 3$};
\draw[->] (4) -- (3) node[midway, below]{\small $\bth 4$};
\draw[->] (2) -- (6) node[midway, above]{\small $2 \overline{6}$};
%\draw[->] (2) -- (3) node[midway, above left]{\small $2 \overline{3}$};
\draw[->] (5) -- (2) node[midway, above]{\small $\btw 5$};
\draw[->] (6) -- (1) node[midway, above left]{\small $\bon 6$};
\end{tikzpicture}.
\end{equation}
(Note that the edges are labeled by the elements given in Figure~\ref{fig:B1_type_E6}.) We have chosen the $\lb$-diagram to minimize the distance from node $r$ to $\sigma$ and each edge label $b$ has minimal depth from $u_{\clfw_{\sigma}}$.

\begin{example}
In type $E_6^{(1)}$ for $B^{6,1} \otimes B^{6,1}$, we have for
\begin{align*}
(\nu, J) & = \;\;
\begin{tikzpicture}[scale=.35,baseline=-18]
 \node at (0,-1.5) {$\emptyset$};
 \begin{scope}[xshift=4cm]
 \rpp{1}{0}{0}
 \end{scope}
 \begin{scope}[xshift=8cm]
 \rpp{1}{0}{0}
 \end{scope}
 \begin{scope}[xshift=12cm]
 \rpp{1,1}{0,0}{0,0}
 \end{scope}
 \begin{scope}[xshift=16cm]
  \rpp{1,1}{0,0}{0,0}
 \end{scope}
 \begin{scope}[xshift=20cm]
 \rpp{1,1}{0,0}{0,0}
 \end{scope}
\end{tikzpicture},
\\[5pt]
\lb(\nu, J) & =
\begin{tikzpicture}[scale=.35,baseline=-18]
 \rpp{1,1}{0,0}{0,0}
 \begin{scope}[xshift=4cm]
 \rpp{1,1}{0,0}{0,0}
 \end{scope}
 \begin{scope}[xshift=8cm]
  \rpp{1,1,1}{0,0,0}{0,0,0}
 \end{scope}
 \begin{scope}[xshift=12cm]
 \rpp{1,1,1,1}{0,0,0,0}{0,0,0,0}
 \end{scope}
 \begin{scope}[xshift=16cm]
  \rpp{1,1,1}{0,0,0}{0,0,0}
 \end{scope}
 \begin{scope}[xshift=20cm]
 \rpp{1,1}{0,0}{0,0}
 \end{scope}
\end{tikzpicture},
\end{align*}
which is in $\RC(B^{1,1} \otimes B^{1,1} \otimes B^{6,1})$. In particular, we added $2$ singular rows of length $1$ to $\nu^{(1)}, \nu^{(3)}, \nu^{(4)}$ and $1$ such row to $\nu^{(2)}, \nu^{(3)}, \nu^{(5)}$ since $1 = e_1 e_3 e_4 e_2 e_5 e_4 e_3 e_1( \bon 6 )$. Note that $\bon6$ comes from the edge $1 \xleftarrow[\hspace{15pt}]{\bon6} 6$ in the $\lb$-diagram. Thus, we obtain
\[
\Phi\bigl(\lb(\nu, J) \bigr) =
\lb\bigl(\Phi(\nu, J) \bigr) =
\begin{tikzpicture}[baseline=-2pt]
\matrix [matrix of math nodes,column sep=-.4, row sep=-.5,text height=10,text width=10,align=center,inner sep=3] 
 {
	\node[draw]{\bsix}; \\
 };
\end{tikzpicture}
\otimes
\begin{tikzpicture}[baseline=-2pt]
\matrix [matrix of math nodes,column sep=-.4, row sep=-.5,text height=10,text width=10,align=center,inner sep=3] 
 {
	\node[draw]{1}; \\
 };
\end{tikzpicture}
\otimes
\begin{tikzpicture}[baseline]
\matrix [matrix of math nodes,column sep=-.4, row sep=-.5,text height=10,text width=10,align=center,inner sep=3] 
 {
	\node[draw]{1}; \\
	\node[draw]{\bon6}; \\
 };
\end{tikzpicture},
\qquad
\Phi(\nu, J) =
\begin{tikzpicture}[baseline]
\matrix [matrix of math nodes,column sep=-.4, row sep=-.5,text height=10,text width=10,align=center,inner sep=3] 
 {
	\node[draw]{1}; \\
	\node[draw]{\bsix}; \\
 };
\end{tikzpicture}
\otimes
\begin{tikzpicture}[baseline]
\matrix [matrix of math nodes,column sep=-.4, row sep=-.5,text height=10,text width=10,align=center,inner sep=3] 
 {
	\node[draw]{1}; \\
	\node[draw]{\bon6}; \\
 };
\end{tikzpicture}.
\]
\end{example}

\begin{remark}
We could alternatively use the $\lb$-digram for type $E_6^{(1)}$ by having the edge $3 \xleftarrow[\hspace{20pt}]{2\bth} 2$ instead of $6 \xleftarrow[\hspace{20pt}]{2\bsix} 2$. However, this results in different KR tableaux.
\end{remark}

%\TravisR{This is an alternative $\lb$-diagram where the elements in paths are paths in $B(\clfw_1)$.}
%\[
%\label{eq:lb_E6}
%\begin{tikzpicture}[baseline=-4]
%\node (1) at (0,0) {$1$};
%\node (2) at (4,1) {$2$};
%\node (3) at (2,0) {$3$};
%\node (4) at (4,0) {$4$};
%\node (5) at (6,0) {$5$};
%\node (6) at (2,1) {$6$};
%\draw[->] (3) -- (1) node[midway, below]{\small $\bon 3$};
%\draw[->] (4) -- (3) node[midway, below]{\small $\bth 4$};
%\draw[->] (2) -- (3) node[midway, above left]{\small $2 \overline{3}$};
%\draw[->] (5) -- (4) node[midway, below]{\small $\bfo 5$};
%\draw[->] (6) -- (1) node[midway, above left]{\small $\bon 6$};
%\end{tikzpicture}.
%\]
%\TravisR{End of alternative diagram.}

Using $\delta_6$, we define the $\lb^{\vee}$-diagram for type $E_6^{(1)}$ by
\begin{equation}
\label{eq:lb_E6_dual}
\begin{tikzpicture}[baseline=-4]
\node (1) at (2,1) {$1$};
\node (2) at (4,1) {$2$};
\node (3) at (6,1) {$3$};
\node (4) at (4,0) {$4$};
\node (5) at (2,0) {$5$};
\node (6) at (0,0) {$6$};
\draw[->] (1) -- (6) node[midway, above left]{\small $1 \bsix$};
\draw[->] (2) -- (1) node[midway, above]{\small $\bon 2$};
\draw[->] (3) -- (2) node[midway, above]{\small $\btw 3$};
\draw[->] (4) -- (5) node[midway, below]{\small $4 \bfive$};
\draw[->] (5) -- (6) node[midway, below]{\small $5 \bsix$};
\end{tikzpicture}.
\end{equation}
Note that this is a usual $\lb$-diagram, but we name it in parallel to the contragredient dual (recall $B(\clfw_1)^{\vee} = B(\clfw_6)$ and we can define $\delta_1^{\vee} = \delta_6$).

For type $E_7^{(1)}$, the definition of left-box we use is given by the $\lb$-diagram
\begin{equation}
\label{eq:lb_E7}
\begin{tikzpicture}[baseline=-4]
\node (1) at (2,1) {$1$};
\node (2) at (4,1) {$2$};
\node (3) at (6,1) {$3$};
\node (4) at (6,0) {$4$};
\node (5) at (4,0) {$5$};
\node (6) at (2,0) {$6$};
\node (7) at (0,0) {$7$};
\draw[->] (6) -- (7) node[midway, below]{\small $6 \bseven$};
\draw[->] (5) -- (6) node[midway, below]{\small $5 \overline{6}$};
\draw[->] (4) -- (5) node[midway, below]{\small $4 \overline{5}$};
\draw[->] (2) -- (1) node[midway, above]{\small $\bon 2$};
\draw[->] (1) -- (7) node[midway, above left]{\small $1 \bseven$};
\draw[->] (3) -- (2) node[midway, above]{\small $\btw 3$};
\end{tikzpicture}.
\end{equation}
We note that other $\lb$-diagrams are possible, but we use the one in~\eqref{eq:lb_E7} for its similarity to~\eqref{eq:lb_E6}.

\begin{example}
\label{ex:RC_E7_B41}
In type $E_7^{(1)}$ for $B^{4,1}$, consider the rigged configuration
\[
(\nu, J) =
\begin{tikzpicture}[scale=.35,anchor=top,baseline=-18]
 \rpp{1}{0}{0}
 \begin{scope}[xshift=4cm]
 \rpp{1,1}{0,0}{0,0}
 \end{scope}
 \begin{scope}[xshift=8cm]
 \rpp{1,1}{1,0}{1,1}
 \end{scope}
 \begin{scope}[xshift=12cm]
 \rpp{1,1,1,1}{0,0,0,0}{0,0,0,0}
 \end{scope}
 \begin{scope}[xshift=16cm]
 \rpp{1,1,1}{0,0,0}{0,0,0}
 \end{scope}
 \begin{scope}[xshift=20cm]
 \rpp{1,1}{0,0}{0,0}
 \end{scope}
 \begin{scope}[xshift=24cm]
 \rpp{1}{0}{0}
 \end{scope}
\end{tikzpicture}.
\]
Note that $e_7 e_6 e_5(4 \bfive) = 7$,
and so we have
\[
\lb(\nu, J) =
\begin{tikzpicture}[scale=.35,anchor=top,baseline=-18]
 \rpp{1}{0}{0}
 \begin{scope}[xshift=4cm]
 \rpp{1,1}{0,0}{0,0}
 \end{scope}
 \begin{scope}[xshift=8cm]
 \rpp{1,1}{1,0}{1,1}
 \end{scope}
 \begin{scope}[xshift=12cm]
 \rpp{1,1,1,1}{0,0,0,0}{0,0,0,0}
 \end{scope}
 \begin{scope}[xshift=16cm]
 \rpp{1,1,1,1}{0,0,0,0}{0,0,0,0}
 \end{scope}
 \begin{scope}[xshift=20cm]
 \rpp{1,1,1}{0,0,0}{0,0,0}
 \end{scope}
 \begin{scope}[xshift=24cm]
 \rpp{1,1}{0,0}{0,0}
 \end{scope}
\end{tikzpicture}
\]
in $\RC(B^{7,1} \otimes B^{5,1})$.
Next, by applying $\delta_7$, we remove the following boxes:
\[
\begin{tikzpicture}[scale=.35,baseline=-18]
\fill[lightgray] (24,-3) rectangle (25,-2);
\node[scale=.7] at (24.6, -2.5) {$\ell_1$}; % 6, -7
\fill[lightgray] (20,-4) rectangle (21,-3);
\node[scale=.7] at (20.6, -3.5) {$\ell_2$}; % 5, -6
\fill[lightgray] (16,-5) rectangle (17,-4);
\node[scale=.7] at (16.6, -4.5) {$\ell_3$}; % 4, -5
\fill[lightgray] (12,-5) rectangle (13,-4);
\node[scale=.7] at (12.6, -4.5) {$\ell_4$}; % 2, 3, -4
\fill[lightgray] (4,-3) rectangle (5,-2);
\node[scale=.7] at (4.6, -2.5) {$\ell_5$}; % -2, 3
\fill[lightgray] (8,-2) rectangle (9,-1);
\node[scale=.7] at (8.6, -1.5) {$\ell_6$}; % 1, -2, -3, 4
\fill[lightgray] (0,-2) rectangle (1,-1);
\node[scale=.7] at (0.6, -1.5) {$\ell_7$}; % -1, -2, 4
\fill[lightgray] (12,-4) rectangle (13,-3);
\node[scale=.7] at (12.6, -3.5) {$\ell_8$}; % -1, 3, -4, 5
\fill[lightgray] (16,-4) rectangle (17,-3);
\node[scale=.7] at (16.6, -3.5) {$\ell_9$}; % -1, 3, -5, 6
\fill[lightgray] (20,-3) rectangle (21,-2);
\node[scale=.7] at (20.6, -2.5) {$\ell_{10}$}; %-1, 3, -6, 7
\fill[lightgray] (24,-2) rectangle (25,-1);
\node[scale=.7] at (24.6, -1.5) {$\ell_{11}$}; %-1, 3, -7

 \rpp{1}{0}{0}
 \begin{scope}[xshift=4cm]
 \rpp{1,1}{0,0}{0,0}
 \end{scope}
 \begin{scope}[xshift=8cm]
 \rpp{1,1}{1,0}{1,1}
 \end{scope}
 \begin{scope}[xshift=12cm]
 \rpp{1,1,1,1}{0,0,0,0}{0,0,0,0}
 \end{scope}
 \begin{scope}[xshift=16cm]
 \rpp{1,1,1,1}{0,0,0,0}{0,0,0,0}
 \end{scope}
 \begin{scope}[xshift=20cm]
 \rpp{1,1,1}{0,0,0}{0,0,0}
 \end{scope}
 \begin{scope}[xshift=24cm]
 \rpp{1,1}{0,0}{0,0}
 \end{scope}
\end{tikzpicture}.
\]
Thus $\delta_7$ returns $\bon 3 \bseven$ and the resulting rigged configuration $(\delta_7\circ\lb)(\nu, J) \in \RC(B^{5,1})$ is the following:
\[
\begin{tikzpicture}[scale=.35,baseline=-18]
 \node at (0,-1.5) {$\emptyset$};
 \begin{scope}[xshift=3cm]
 \rpp{1}{0}{0}
 \end{scope}
 \begin{scope}[xshift=7cm]
 \rpp{1}{0}{0}
 \end{scope}
 \begin{scope}[xshift=11cm]
 \rpp{1,1}{0,0}{0,0}
 \end{scope}
 \begin{scope}[xshift=15cm]
 \rpp{1,1}{0,0}{0,0}
 \end{scope}
 \begin{scope}[xshift=19cm]
 \rpp{1}{0}{0}
 \end{scope}
 \begin{scope}[xshift=23cm]
 \node at (0,-1.5) {$\emptyset$};
 \end{scope}
\end{tikzpicture}.
\]
Since $e_7 e_6 (5 \bsix) = 7$, applying $\lb$ results in
\[
\begin{tikzpicture}[scale=.35,baseline=-18]
 \node at (0,-1.5) {$\emptyset$};
 \begin{scope}[xshift=3cm]
 \rpp{1}{0}{0}
 \end{scope}
 \begin{scope}[xshift=7cm]
 \rpp{1}{0}{0}
 \end{scope}
 \begin{scope}[xshift=11cm]
 \rpp{1,1}{0,0}{0,0}
 \end{scope}
 \begin{scope}[xshift=15cm]
 \rpp{1,1}{0,0}{0,0}
 \end{scope}
 \begin{scope}[xshift=19cm]
 \rpp{1,1}{0,0}{0,0}
 \end{scope}
 \begin{scope}[xshift=23cm]
 \rpp{1}{0}{0}
 \end{scope}
\end{tikzpicture},
\]
and applying $\delta_7$ selects
\[
\begin{tikzpicture}[scale=.35,baseline=-18]
\fill[lightgray] (23,-2) rectangle (24,-1);
\node[scale=.7] at (23.6, -1.5) {$\ell_1$}; % 6, -7
\fill[lightgray] (19,-3) rectangle (20,-2);
\node[scale=.7] at (19.6, -2.5) {$\ell_2$}; % 5, -6
\fill[lightgray] (15,-3) rectangle (16,-2);
\node[scale=.7] at (15.6, -2.5) {$\ell_3$}; % 4, -5
\fill[lightgray] (11,-3) rectangle (12,-2);
\node[scale=.7] at (11.6, -2.5) {$\ell_4$}; % 2, 3, -4
\fill[lightgray] (3,-2) rectangle (4,-1);
\node[scale=.7] at (3.6, -1.5) {$\ell_5$}; % -2, 3
\fill[lightgray] (7,-2) rectangle (8,-1);
\node[scale=.7] at (7.6, -1.5) {$\ell_6$}; % 1, -2, -3, 4
\fill[lightgray] (11,-2) rectangle (12,-1);
\node[scale=.7] at (11.6, -1.5) {$\ell_7$}; % 1, -4, 5
\fill[lightgray] (15,-2) rectangle (16,-1);
\node[scale=.7] at (15.6, -1.5) {$\ell_8$}; % 1, -5, 6
\fill[lightgray] (19,-2) rectangle (20,-1);
\node[scale=.7] at (19.6, -1.5) {$\ell_9$}; %1, -6, 7

 \node at (0,-1.5) {$\emptyset$};
 \begin{scope}[xshift=3cm]
 \rpp{1}{0}{0}
 \end{scope}
 \begin{scope}[xshift=7cm]
 \rpp{1}{0}{0}
 \end{scope}
 \begin{scope}[xshift=11cm]
 \rpp{1,1}{0,0}{0,0}
 \end{scope}
 \begin{scope}[xshift=15cm]
 \rpp{1,1}{0,0}{0,0}
 \end{scope}
 \begin{scope}[xshift=19cm]
 \rpp{1,1}{0,0}{0,0}
 \end{scope}
 \begin{scope}[xshift=23cm]
 \rpp{1}{0}{0}
 \end{scope}
\end{tikzpicture},
\]
which yields the empty rigged configuration and a return value of $1\bsix 7$. Thus, iterating this, we have
\[
\Phi(\nu, J) =
\begin{tikzpicture}[baseline]
\matrix [matrix of math nodes,column sep=-.4, row sep=-.5,text height=10,text width=15,align=center,inner sep=3] 
 {
	\node[draw]{7}; \\
	\node[draw]{\bseven6}; \\
	\node[draw]{1\bsix7}; \\
	\node[draw]{\bon3\bseven}; \\
 };
\end{tikzpicture}.
\]
\end{example}

For type $E_8^{(1)}$, the $\lb$-diagram is
\begin{equation}
\label{eq:lb_E8}
\begin{tikzpicture}[baseline=-4]
\node (1) at (2,1) {$1$};
\node (2) at (4,1) {$2$};
\node (3) at (6,1) {$3$};
\node (4) at (8,0) {$4$};
\node (5) at (6,0) {$5$};
\node (6) at (4,0) {$6$};
\node (7) at (2,0) {$7$};
\node (8) at (0,0) {$8$};
\draw[->] (7) -- (8) node[midway, below]{\small $7 \overline{8}$};
\draw[->] (6) -- (7) node[midway, below]{\small $6 \bseven$};
\draw[->] (5) -- (6) node[midway, below]{\small $5 \overline{6}$};
\draw[->] (4) -- (5) node[midway, below]{\small $4 \overline{5}$};
\draw[->] (2) -- (1) node[midway, above]{\small $\bon 2$};
\draw[->] (1) -- (8) node[midway, above left]{\small $1 \overline{8}$};
\draw[->] (3) -- (2) node[midway, above]{\small $\btw 3$};
\end{tikzpicture}.
\end{equation}

For type $E_6^{(2)}$, we use the $\lb$-diagram
\begin{equation}
\label{eq:lb_E6_2}
\begin{tikzpicture}[baseline=-4]
\node (1) at (0,0) {$1$};
\node (2) at (2,0) {$2$};
\node (3) at (4,0) {$3$};
\node (4) at (2,1) {$4$};
\draw[->] (2) -- (1) node[midway, below]{\small $\bon 2$};
\draw[->] (3) -- (2) node[midway, below]{\small $\btw 3$};
\draw[->] (4) -- (1) node[midway, above left]{\small $\bon 4$};
\end{tikzpicture}.
\end{equation}

%=====================================================================
\section{Untwisted types}
\label{sec:untwisted}

Let $\g$ be of type $C_n^{(1)}$, $F_4^{(1)}$, or $G_2^{(1)}$. For these types, we note that there is a virtualization map $v$ to the corresponding dual type $\g^{\vee}$ and that the scaling factors $(\gamma_a)_{a \in I}$ are exactly those considering $\g$ as a folding of the corresponding simply-laced type $\virtual{\g}$. For type $G_2^{(1)}$ to $D_4^{(3)}$, we also need to interchange $1 \leftrightarrow 2$ due to our numbering conventions. However, for type $B_n^{(1)}$, we will use the embedding into type $D_{n+1}^{(1)}$ as it affords an easier proof than $A_{2n-1}^{(2)}$. Using this, we construct the bijection $\Phi$ by showing it commutes with the virtualization map to the dual untwisted type.

%%%%%%%%%%
\subsection{The map $\delta_r$.}

\begin{table}
\[
\begin{array}{ccccc}
\toprule
\g & B_n^{(1)} & C_n^{(1)} & F_4^{(1)} & G_2^{(1)} \\ \midrule
\g^{\vee} & D_{n+1}^{(1)} & D_{n+1}^{(2)} & E_6^{(2)} & D_4^{(3)} \\ \midrule
r & n & 1 & 4 & 1 \\ \midrule
\delta^v_r & \virtual{\delta}_n \circ \virtual{\delta}_{n+1} & \virtual{\delta}_1 & \virtual{\delta}_1 \circ \virtual{\delta}_1 \circ \virtual{\lb} & \virtual{\delta}_1 \circ  \virtual{\delta}_1 \circ \virtual{\lb} \circ \virtual{\delta}_1 \circ \virtual{\lb} %\\ \midrule
%(\gamma_a)_{a \in I} & (2\dotsc, 2, 1) & (2, 1, \dotsc, 1, 2) & (2,2,2,1,1) & (3, 1, 3)
\\ \bottomrule
\end{array}
\]
\caption{The map $\delta_r^v$ for the virtualization map $v$ given by $\g \to \g^{\vee}$. Recall that $\virtual{\lb}$ is the $\lb$ map in type $\g^{\vee}$ and is needed to split the resulting column from the virtualization map.}
\label{table:virtual_dual}
\end{table}

It is known that $B^{r,1}$ can be realized as a virtual crystal inside of
\[
\virtual{B}^{r,1} = \begin{cases}
B^{n+1,1} \otimes B^{n,1} & \text{if } \g = B_n^{(1)}, \\
B^{2,1} & \text{if } \g = G_2^{(1)}, \\
B^{r,1} & \text{otherwise,}
\end{cases}
\]
of type $\g^{\vee}$. We want to define the map $\delta_r := v^{-1} \circ \delta^v_r \circ v$, where $r$ and $\delta_r^v$ are given in Table~\ref{table:virtual_dual}. Thus, we need to show that
\[
\delta_r^v \colon \virtual{\RC}(B^{r,1} \otimes B^{\bullet}) \to \virtual{\RC}(B^{\bullet})
\]
is well-defined when restricted to the image of $v$.

\begin{theorem}
\label{thm:virtual_untwisted_delta}
Suppose $(\virtual{\nu}, \virtual{J})$ satisfies Equation~\eqref{eq:virtual_RC}, then $\delta_r^v(\virtual{\nu}, \virtual{J})$ satisfies Equation~\eqref{eq:virtual_RC}.
Moreover, the map $\delta_r^v$ is well-defined when restricted to the image of $v$.
\end{theorem}

\begin{proof}
We proceed by induction by examining $(\delta^v_r)^{-1}$, where the base case is done by $\delta_r^v(\virtual{\nu}, \virtual{J}) = (\virtual{\nu}, \virtual{J})$, which returns the highest weight element $v(u_{\clfw_r})$. Next we assume the claim holds when $\delta_r^v$ returns $\virtual{b} := v(b)$. Fix $a \in I_0$. Let $(\virtual{\nu}', \virtual{J}')$ be the rigged configuration such that $(\overline{\nu}, \overline{J}) := \delta_r^v(\virtual{\nu}, \virtual{J}) = \delta_r^v(\virtual{\nu}', \virtual{J}')$ but with a return value of $f_a^v \virtual{b} = v(f_a b)$.

For type $C_n^{(1)}$, we have that $(\virtual{\nu}', \virtual{J}')$ differs from $(\virtual{\nu}, \virtual{J})$ by the addition of $\gamma_a$ boxes to a row in $\virtual{\nu}^{(a)}$. From Equation~\eqref{eq:virtual_vacancy_numbers}, we have all riggings $\virtual{J}'$ are still multiples of $\gamma_{a'}$ for all $a' \in \virtual{I}$, and the claim follows.

Next we consider type $F_4^{(1)}$.
The case when $f_a^v \lb(\virtual{b}) = (f_a^v \virtual{b}_2) \otimes \virtual{b}_1$ is similar to the type $C_n^{(1)}$ case.
Now suppose $f_a^v \lb(\virtual{b}) = (\virtual{f}_a \virtual{b}_2) \otimes (\virtual{f}_a \virtual{b}_1)$. Note that $\virtual{\delta}^{-1}$ for $\virtual{f}_a \virtual{b}_1$ starts at $\nu^{(a)}$ and the only singular rows in $\nu^{(a')}$ for $\gamma_{a'} > 1$ are the rows selected by $\virtual{\delta}_r^{-1}$ by Equation~\eqref{eq:virtual_vacancy_numbers}. Hence, applying $\virtual{\delta}^{-1}$ for $\virtual{f}_a \virtual{b}_2$ must select those same rows in $\nu^{(a')}$ for $\gamma_{a'} > 1$ as there are sufficient singular rows in $\nu^{(a')}$ for $\gamma_{a'} = 1$ of length $\ell_{i_j} \leq \ell_k \leq \ell_{i_{j+1}}$ for all $i_j \leq k \leq i_{j+1}$, where $\ell_{i_1}, \dotsc, \ell_{i_q}$ are the lengths of the rows selected of $\nu^{(a')}$ for fixed $a'$ such that $\gamma_{a'} > 1$. We note that such rows exists because $\virtual{b}_2 \geq \virtual{b}_1$. Once all such rows have been paired, we are equivalent to the case of $\virtual{b}' \otimes v(u_{\clfw_r})$ with all sufficiently long rows non-singular. Hence, the claim follows by induction.

Now suppose $f_a^v \lb(\virtual{b}) = \virtual{b}_2 \otimes (f_a^v \virtual{b}_1)$ and let $(\tnu, \tJ)$ and $(\tnu', \tJ')$ denote $\virtual{\delta}_r^{-1}(\overline{\nu}, \overline{J})$ by adding $\virtual{b}_1$ and $f_a^v \virtual{b}_1$ respectively. Note that any row selected to obtain $\tnu'$ is at most as long as that to obtain $\tnu$ and that $\virtual{\delta}_r^{-1}$ added $\gamma_a$ boxes to this row in obtaining $(\tnu')^{(a)}$. Therefore, this case follows from our induction assumption for the case where the necessary rows are made to be (non-)singular but with a return value of $\lb(\virtual{b})$.

The proof for type $G_2^{(1)}$ is similar. For type $B_n^{(1)}$, we note that if $\virtual{f}_a(\virtual{b}_2 \otimes \virtual{b}_1) = (\virtual{f}_a \virtual{b}_2) \otimes \virtual{b}_1$ for $a \neq n,n+1$, then we must have previously had $\virtual{f}_a$ act on the right. Specifically, this is equivalent to having $s_a$ have the same sign in both columns of $B(\virtual{\clfw}_n) \otimes B(\virtual{\clfw}_{n+1})$. Thus the proof is also similar for type $B_n^{(1)}$.
\end{proof}

We note that our proof is almost type independent as it is the same general technique, but we require some mild type dependencies. We also note that Theorem~\ref{thm:virtual_untwisted_delta},~\cite[Rem.~5.15]{SchillingS15}, and~\cite[Thm.~6.5]{SchillingS15} implies that we could define $\delta$ by considering the virtualization map of $B_n^{(1)}$ into $A_{2n-1}^{(1)}$.

\begin{remark}
\label{remark:scaling_untwisted}
Instead of using the scaling factors to enlarge the partitions, we could instead consider scaling each $\nu^{(a)}$ by $1 / T_a$. So that the partitions have integer lengths, we scale by (a multiple of) $\max_a T_a$, which the net effect would be to multiply by $\gamma_a$. This suggests a strong relationship between the Weyl chamber geometry and rigged configurations through the bijection $\Phi$.
\end{remark}

%%%%%%%%%%
\subsection{Defining $\lb$ and general columns}

The $\lb$-diagram for type $C_n^{(1)}$ is given by Equation~\eqref{eq:lb_general}.
For type $B_n^{(1)}$, the $\lb$-diagram we consider is
\[
\begin{tikzpicture}[baseline=-4]
\node (n) at (0,0) {$n$};
\node (1) at (2,0) {$1$};
\node (d) at (2,0.75) {$\vdots$};
\node (nm) at (2,1.5) {$n-1$};
\draw[->] (1) -- (n) node[midway, below]{\small $1\bn$};
\draw[->,dotted] (d) -- (n);
\draw[->] (nm) -- (n) node[midway, above left]{\small $(n-1)\bn$};
\end{tikzpicture}.
\]

For type $F_4^{(1)}$, we use the $\lb$-diagram
\[
\begin{tikzpicture}[baseline=-4]
\node (4) at (0,0) {$4$};
\node (3) at (2,0) {$3$};
\node (2) at (4,0) {$2$};
\node (1) at (2,1) {$1$};
\draw[->] (3) -- (4) node[midway, below]{\small $3 \bfo$};
\draw[->] (2) -- (3) node[midway, below]{\small $2 \bth$};
\draw[->] (1) -- (4) node[midway, above left]{\small $1 \bfo$};
\end{tikzpicture}.
\]

For type $G_2^{(1)}$, we want to consider $B(\clfw_1)$ as a virtual crystal inside of $B(3\clfw_1) \subseteq B(\clfw_2) \otimes B(\clfw_2)$. This corresponds to adding a singular row of length 1 to $\nu^{(2)}$, which would be of length $\gamma_2$ in $\gamma_2 \nu^{(2)}$. This allows us to construct an $\lb$-diagram as $1 \xrightarrow[\hspace{20pt}]{1^3\btw} 2$.

%=====================================================================
\section{Results}
\label{sec:results}

We gather our results and proofs here.
We first prove our results for minuscule nodes. Next will be for the adjoint node. We then extend our results to all single-columns.
In the following subsection, we collect our main results: a uniform description and proof of the rigged configuration bijection $\Phi$ and $X = M$ for all single-column KR tableaux in dual untwisted types. We then discuss how $\Phi$ can be extended to a bijection for all affine types by describing the relation with virtualization. We conclude this section extending $\Phi$ and $X = M$ uniformly to tensor products of higher level KR crystals corresponding to minuscule nodes.

%%%%%%%%%%
\subsection{Minuscule nodes}

We assume that $\g$ is a dual untwisted type and $r$ is a minuscule node.

We note that $\delta_1 = \KSS{\delta}$ was described by Okado and Sano~\cite{OS12} for type $E_6^{(1)}$. It is also straightforward to see that $\delta_1 = \KSS{\delta}$ in type $A_n^{(1)}$ (given in~\cite{KKR86,KR86}) and type $D_n^{(1)}$ and $A_{2n-1}^{(2)}$ (given in~\cite{OSS03}). We collect these results the following theorem.

\begin{theorem}
\label{thm:known_minuscule_bijections}
Let $B = \bigotimes_{i=1}^N B^{1,1}$ be of type $A_n^{(1)}$, $D_n^{(1)}$, $E_6^{(1)}$, or $A_{2n-1}^{(2)}$. The map $\KSS{\Phi} \colon \RC(B) \to B$ is a bijection such that $\KSS{\Phi} \circ \eta$ sends cocharge to energy.
\end{theorem}

We need a few facts about minuscule representations (see, \textit{e.g.},~\cite{Stembridge03}). Let $\clW = \langle s_a \mid a \in I_0 \rangle$ be the Weyl group of $\g_0$ with $s_a$ being the simple reflection corresponding to $\alpha_a$. The cosets $\clW / \clW_{\widehat{r}}$, where $\clW_{\widehat{r}}$ is the parabolic subgroup generated by $\langle s_i \mid i \in I_0 \setminus \{r\} \rangle$, parameterize the elements $B(\clfw_r)$. Specifically, we have
\begin{equation}
\label{eq:minuscule_parameterization}
B(\clfw_r) = \{ b_w := f_{a_1} \cdots f_{a_{\ell}} u_{\clfw_r} \mid w = s_{a_1} \cdots s_{a_{\ell}} \in \clW / \clW_{\widehat{r}} \},
\end{equation}
where the elements $w$ are the minimal length coset representatives.
Furthermore, there reduced expressions of $w$ give all paths from $b_w$ to $u_{\clfw_r}$.

\begin{lemma}
\label{lemma:ep_phi_minuscule_repr}
Let $\clfw_r$ be a minuscule weight. Then $0 \leq \varepsilon_a(b) + \varphi_a(b) \leq 1$ for all $b \in B(\clfw_r)$ and $a \in I_0$.
\end{lemma}

\begin{proof}
The claim follows immediately from Equation~\eqref{eq:minuscule_parameterization}.
\end{proof}

The following lemma is the key fact for minuscule nodes, which is a generalization of~\cite[Lemma 2.1]{OS12}.

\begin{lemma}
\label{lemma:key_minuscule}
Let $\lambda \in P^+$. Then $\lambda$ is a minuscule weight if and only if the crystal graph of $B(\lambda)$ has the following properties:
\begin{enumerate}[{\rm (A)}]
\item \label{key1} Consider a path $P$ in $B(\clfw_r)$ such that the initial and terminal arrows have the same color $a$. Then either
\begin{enumerate}[{\rm (a)}]
\item there are exactly two arrows colored by $a'$ and $a''$ in $P$ such that $a' \sim a$ and $a'' \sim a$, or
\item there is exactly one arrow colored by $a'$ in $P$ such that $A_{aa'} = -2$.
\end{enumerate}

\item \label{key2} Consider a length $2$ path with colors $(a, a')$ in $B(\clfw_r)$ with $a \not\sim a'$. Then there exists a length $2$ path $(a', a)$ with the same initial and terminal vertices in $B(\clfw_r)$.
\end{enumerate}
\end{lemma}

\begin{proof}
We recall that $\lambda \in P^+$ is minuscule implies $\lambda$ is a fundamental weight.
By Equation~\eqref{eq:minuscule_parameterization}, property~(\ref{key1}) is given by~\cite[Prop.~2.3]{Stembridge01II} and~(\ref{key2}) was shown in~\cite{Proctor99} for the simply-laced case and the general case by~\cite[Prop.~2.1]{Stembridge01II}.
\end{proof}

\begin{remark}
The conditions~(2) and~(4) of~\cite[Lemma~2.1]{OS12}, respectively, are consequences of~(\ref{key1}) and~(\ref{key2}) of Lemma~\ref{lemma:key_minuscule}, respectively, which correspond to~(1) and~(3) in~\cite{OS12}, respectively. Thus we have only stated the necessary properties.
\end{remark}

One important consequence of Lemma~\ref{lemma:key_minuscule} is that the result of $\delta_r$ does not depend on the choice of $a'$ such that $i_{a'} = \min_a i_a$ at each step. Another is that the proof given in~\cite{OS12} holds in this generality except for a fact about the local energy function (\textit{i.e.}, part~(\ref{cV})). Recall that we need to show $H(b_N \otimes b_{N-1})$ is equal to the number of length 1 singular rows selected by $\delta_r$ in $\nu^{(r)}$, which implies that $\Phi$ preserves statistics. Note that we have already shown that $\Phi$ is a bijection.

Next, we compute the (local) energy function on classically highest weight elements in $B^{r,1} \otimes B^{r,1}$.

%\Travis{In type $D_{n+1}^{(2)}$, highest weight elements of $B^{n,1} \otimes B^{n,1}$ are given by $+ \cdots + - \cdots - \otimes u_{\clfw_n}$ where the local energy is equal to the number of $-$'s.}

\begin{theorem}
\label{thm:minuscule_energy}
Let $\clfw_r$ be a special node. Then for classically highest weight elements $b \otimes u_{s'\clfw_r} \in B^{r,s} \otimes B^{r,s'}$, we have $H(b \otimes u_{s'\clfw_r})$ equal to the number of $r$-arrows in the path from $b$ to $u_{s\clfw_r}$ in $B(s\clfw_r)$.
\end{theorem}

\begin{proof}
First recall that for $r$ to be a special node, there exists a diagram automorphism $\phi$ such that $r = \phi(0)$. Therefore, if we consider the finite-type $\g_r$ given by $I_r := I \setminus \{r\}$, then the corresponding fundamental weight $\check{\Lambda}_0$ is minuscule. We note that $B^{r,s} \iso B(s\check{\Lambda}_0)$ as $U_q(\g_r)$-crystals, and the classically highest weight element in $B^{r,s}$ is the $I_r$-lowest weight element in $B(s\check{\Lambda}_0)$.
Hence, for every classically highest weight element $b \otimes u_{s'\clfw_r} \in B^{r,s} \otimes B^{r,s'}$, there exists a path to $u_{s\clfw_r} \otimes u_{s'\clfw_r}$ using the crystal operators $f_a$, for $a \in I \setminus \{r\}$. Moreover, the crystal operators only act on the left-most factor since $\varphi(u_{s\clfw_r}) = s\Lambda_r$. The number of $0$-arrows is equal to the number of $r$-arrows in the path from $b$ to $u_{s\clfw_r}$ in $B(s\clfw_r)$ because $f_r f_0 b = f_0 f_r b$ as $r \not\sim 0$. Hence, we have $H(b \otimes u_{s'\clfw_r})$ as claimed.
\end{proof}

It remains to show the local energy function satisfies Equation~\eqref{eq:change_energy}.

\begin{lemma}
\label{lemma:minuscule_local_energy_change}
Part~(\ref{cV}) holds for
\[
B = \bigotimes_{i=1}^N B^{r, 1}
\]
when $\clfw_r$ is a minuscule weight.
\end{lemma}

\begin{proof}
Note that in order for the second application of $\delta_r$ to return $u_{\clfw_r}$, there must not exist any singular rows in $\nu^{(r)}$ after the first application of $\delta_r$. Hence all rows selected in $\nu^{(r)}$ must have length $1$. Thus the claim holds on classically highest weight elements of $B^{r,1} \otimes B^{r,1}$ by Theorem~\ref{thm:minuscule_energy}.

Thus, to show this holds in the general case of $b' \otimes b$, we use induction on the classical components in $B^{r,1} \otimes B^{r,1}$. We note that we are not applying $f_r$ to the crystal/rigged configuration, but instead looking at how the two left-most factors differ, which results in a box being added to a row in $\nu^{(r)}$.
Indeed, we show the claim holds by showing the additional box removed to obtain $f_a(b' \otimes b)$ must not have come from a length $1$ row as $H(b' \otimes b) = H\bigl( f_a(b' \otimes b) \bigr)$ for all $a \in I_0$.

Suppose $f_r(b' \otimes b) = b' \otimes (f_r b)$ and let $(\nu', J')$ be the corresponding rigged configuration.
We note that the element $b'$ is unchanged, and so $\delta_r$ selects the same number of boxes in $\nu^{(r)}$ as in $\nu'^{(r)}$. There must be at least one more singular row in the $r$-th partition of $\delta_r(\nu',J')$ than in $\delta_r(\nu, J)$. This implies we must have removed the same number of (singular) rows of length $1$ from $\nu^{(r)}$ and $\nu'^{(r)}$. Hence, the claim follows by induction.

The case for $a \neq r$ is similar as above except the number of singular rows in the $r$-th partition of $\delta_r(\nu', J')$ is at least the same as in $\delta_r(\nu, J)$.

Instead suppose $f_r(b' \otimes b) = (f_r b') \otimes b$. Let $(\overline{\nu}, \overline{J}) = \delta_r(\nu, J) = \delta_r(\tnu, \tJ)$ with a return value of $b' \otimes b$ and $(f_r b') \otimes b$ respectively. Thus, we have $\varepsilon_r(b') \geq \varphi_r(b)$, and we must have $\varepsilon_r(b') = \varphi_r(b) = 0$ because $f_r(b') \neq 0$ and Lemma~\ref{lemma:ep_phi_minuscule_repr}. This implies that $b \neq u_{\clfw_r}$ as $\varphi_r(u_{\clfw_r}) = 1$, and so there must exist a singular string in $\overline{\nu}^{(r)}$. Suppose we select only length 1 singular strings in $\tnu^{(r)}$, then in order for these to be singular strings in $\overline{\nu}^{(r)}$, we must have followed twice as many $a$-arrows, for $a \sim r$, than $r$-arrows. However, by Lemma~\ref{lemma:key_minuscule}, we can only select the same number of $a$-arrows as $r$-arrows. This is a contradiction, and hence, the claim follows.
\end{proof}

Thus, we collect all of our results for this section in the following.

\begin{lemma}
\label{lemma:minuscule_bijections}
Let $B = \bigotimes_{i=1}^N B^{r,1}$, where $r$ is a minuscule node. The map
\[
\Phi_r \colon \RC(B) \to B,
\]
which is given by applying $\delta_r$, is a bijection on classically highest weight elements such that $\Phi_r \circ \eta$ sends cocharge to energy.
\end{lemma}

%%%%%%%%%%
\subsection{Adjoint nodes}

We assume that $\g$ is a dual untwisted type and consider the adjoint node $N_{\g}$.

As for minuscule representations, there is a bijection between paths $x_{\alpha}$ to $x_{\theta}$ and reduced expressions of minimal length coset representatives in $\clW / \clW_{\widehat{N}_{\g}}$ (or $\clW / \clW_{\widehat{1}, \widehat{n}}$ for type $A_n^{(1)}$).
We also have an analog of Lemma~\ref{lemma:key_minuscule} for $B(\theta)$.

\begin{lemma}
\label{lemma:key_adjoint}
Consider the path
\[
x_{\alpha_a} \xrightarrow[\hspace{20pt}]{a} y_a \xrightarrow[\hspace{20pt}]{a} x_{-\alpha_a}
\]
as a single edge.
Then (\ref{key1}) and (\ref{key2}) of Lemma~\ref{lemma:key_minuscule} holds for $B(\theta)$.
\end{lemma}

\begin{proof}
This follows from Proposition~\ref{prop:adjoint_elements} and that $\clW$ acts transitively on $\Delta$.
\end{proof}

\begin{lemma}
\label{lemma:adjoint_bijections}
Let $\g$ be such that $t_a = 1$ for all $a \in I$.
Let $B = \bigotimes_{i=1}^N B^{\theta,1}$ be a tensor product of KR crystals. Then the map
\[
\Phi \colon \RC(B) \to B
\]
is a bijection on classically highest weight elements such that $\Phi \circ \eta$ sends cocharge to energy.
\end{lemma}

We follow the proof of the KSS-type bijection~\cite{OSS03}. Recall that we have $p_{\infty}^{(a)} = \inner{\alpha^{\vee}_a}{\wt(\nu, J)}$. We give our proof when $\g$ not of type $A_n^{(1)}$ for simplicity, but the proof for type $A_n^{(1)}$ follows by considering $N_{\g} = \{1, n\}$.

\begin{proof}[Proof of~(\ref{cI})]
Suppose $\widetilde{\lambda} = \lambda - \overline{\wt}(r)$ is not dominant. Thus $r \neq y_a, \emptyset$ as $\overline{\wt}(y_a) = \overline{\wt}(\emptyset) = 0$ for all $a \in I$. Thus the only possibilities which make $\widetilde{\lambda}$ not dominant for some $a$ is when $0 \leq \inner{\alpha^{\vee}_a}{\lambda} < \varphi_a(r)$. Note that $\delta_{\theta}$ terminates at $\nu^{(a)}$ in each of these cases, and let $P_{\delta}$ denote the path taken by $\delta_{\theta}$. Let $\ell = \nu_1^{(a)}$, \textit{i.e.}, the largest part of $\nu^{(a)}$. Let $u_{\theta}$ denote the highest weight element in $B(\theta) \subseteq B^{N_{\g},1}$.

We start by assuming $\ell = 0$. Then we have
\begin{equation}
\label{eq:ell_equal_0}
\inner{\alpha^{\vee}_a}{\lambda} = \sum_{j \in \ZZ_{>0}} j L_j^{(a)} - \sum_{b \in I_0 \setminus \{a\}} A_{ab} \absval{\nu^{(b)}}.
\end{equation}
Consider the case when $\inner{\alpha^{\vee}_a}{\lambda} = 0$. If $a = N_{\g}$, then this is a contradiction since $L_1^{(N_{\g})} > 0$ and $A_{ab} \leq 0$ for all $b \in I_0 \setminus \{a\}$. Now if $a \neq N_{\g}$, then $r \neq u_{\theta}$. Thus we must have removed a box from $\nu^{(b)}$ for some $b \sim a$ when performing $\delta_{\theta}$. So $-A_{ab} \absval{\nu^{(b)}} > 0$, which is a contradiction. Next, consider the case $\inner{\alpha^{\vee}_a}{\lambda} > 0$.
Hence, we have $\varphi_a(r) > 1$, and so $r \neq u_{\theta}$ (specifically, $r = x_{\alpha_a}$ in the types we consider). Note that either
\begin{itemize}
\item there exists a $a' \sim a$ such that $\delta_{\theta}$ removes a box from $\nu^{(a')}$ with $-A_{aa'} \geq \varphi_a(r)$, or
\item there exists $a', a''$ such that $a', a'' \sim a$ such that $\delta_{\theta}$ removes a box from $\nu^{(a')}$ and $\nu^{(a'')}$ (if $a' = a''$, then two boxes are removed)
\end{itemize}
from the crystal structure of $B(\theta)$. (Note this is essentially Lemma~\ref{lemma:key_adjoint}.) Thus, Equation~\eqref{eq:ell_equal_0} implies$\inner{\alpha^{\vee}_a}{\lambda} \geq \varphi_a(r)$, which is a contradiction.

Now assume that $\ell > 0$. By the definition of the vacancy numbers, we have
\[
0 \leq p_i^{(a)} + \sum_{b \in I_0} A_{ab} m_i^{(b)} = \inner{\alpha^{\vee}_a}{\lambda}.
\]
In particular, we have
\begin{equation}
\label{eq:necessary_convexity}
-p_{i-1}^{(a)} + 2 p_i^{(a)} - p_{i+1}^{(a)} = L_i^{(a)} - \sum_{b \in I_0} A_{ab} m_i^{(b)}.
\end{equation}

\vspace{12pt} \noindent
\underline{$\inner{\alpha^{\vee}_a}{\lambda} = 0$}:

We note that this is the same proof of~(\ref{cI}) for $\ell > 0$ as given in~\cite{OS12}.

In this case, $\nu^{(a)}$ has a singular string of length $\ell$ since $0 \geq p_{\ell}^{(a)} \geq J_{\ell}^{(a)} \geq 0$ by convexity. Moreover, we have $m_i^{(b)} = 0$ for all $b \sim a$ and $i > \ell$. If not all rows of length $\ell$ in $\nu^{(a)}$ have been selected (or doubly selected in Case~(S) if we return $x_{\alpha}$ with $\alpha < 0$), then we have a contradiction as we can select a row from $m_{\ell}$. Next, by considering the smallest subpath $P$ of $P_{\delta}$ ending at $r$ and starting with an $a$ arrow, the fact selecting every row of length $\ell$ implies
\begin{equation}
\label{eq:ell_inequality}
2m_{\ell}^{(a)} \leq \sum_{b \sim a} A_{ab} m_i^{(a)}
\end{equation}
by Lemma~\ref{lemma:key_adjoint}.
Thus, Equation~\eqref{eq:necessary_convexity} implies that $p_{\ell-1}^{(a)} = 0$ and that Equation~\eqref{eq:ell_inequality} is an equality. Therefore all rows of length $\ell$ in $\nu^{(b)}$ for $b \sim a$ are selected by $\delta$ between. Hence, when $\delta_{\theta}$ selected the first row of length $\ell$, then all rows of length at most $\ell-1$ must have been selected by the definition of $\delta_{\theta}$, otherwise we would have selected a row of length at most $\ell-1$ in $\mu^{(b)}$ for $b \sim a$. Thus, we can proceed by induction and show that we select every row of $\nu^{(a)}$.

Therefore, we now have
\[
0 = L_1^{(a)} - \sum_{b \in I_0} A_{ab} m_1^{(b)}.
\]
If $a = N_{\g}$, then this is a contradiction from Equation~\eqref{eq:ell_inequality} and $L_1 \geq 1$. If $a \neq N_{\g}$, then we must have selected all boxes in length $1$ rows of $\nu^{(b)}$ for $b \sim a$ between the first box of a length $1$ row of $\nu^{(a)}$. However, to get to $a$, we must have selected a length $1$ row of $\nu^{(b)}$ for some $b \sim a$. This is a contradiction.

\vspace{12pt} \noindent
\underline{$\inner{\alpha^{\vee}_a}{\lambda} = 1$}:

As above, we have $\varphi_a(r) > 1$ and $r = x_{\alpha_a}$.
By convexity, we have $p_{\ell}^{(a)} \leq p_{\ell+1}^{(a)} \leq \inner{\alpha^{\vee}_a}{\lambda}$.
If $p_{\ell}^{(a)} = 1$, then we have $m_i^{(a')} = 0$ for all $a' \sim a$ and $i > \ell$ because $\inner{\alpha^{\vee}_a}{\lambda} = p_i^{(a)} = p_{\ell}^{(a)}$ and Equation~\eqref{eq:necessary_convexity}. Furthermore, every row of $\nu^{(a)}$ of length $\ell$ must be selected and singular, as otherwise $\delta_{\theta}$ would select such a (quasi)singular row. Therefore, the previous argument holds and results in a contradiction. Hence, we can assume $p_{\ell}^{(a)} = 0$.

Since $0 = p_{\ell}^{(a)} < p_{\ell+1}^{(a)} = 1$, we must have $m_{\ell+1}^{(a')} = 1$ for precisely one $a' \sim a$ with $A_{aa'} = -1$ by Equation~\eqref{eq:necessary_convexity}.
Note that there exists a singular row of length $\ell$ in $\nu^{(a)}$.
If $\delta_{\theta}$ can select a row of length $\ell$ from $\nu^{(a)}$ (whether there is a selectable row or not), then we have a contradiction as above. Therefore, we must have $\delta_{\theta}$ selecting the row of length $\ell+1$ in $\nu^{(a')}$.
If the row of length $\ell+1$ was the first row of length $\ell+1$ selected, then we would have chosen this row in $\nu^{(a)}$ before the row in $\nu^{(a')}$, which is a contradiction. Thus suppose $\delta_{\theta}$ previously selected a row of length $\ell+1$ in $\nu^{(a'')}$ corresponding to some node $x_{\alpha} \in B(\theta)$. If $a'' \sim a$, then we have $p_{\ell+1}^{(a)} > 1$, which is a contradiction. We note that $f_a x_{\alpha} \neq 0$ as the coefficient of $\alpha_{a'}$ in $\alpha$ is still 1. Hence, by Lemma~\ref{lemma:key_adjoint}, the map $\delta_{\theta}$ would remove a box from $\nu^{(a)}$ before the one from $\nu^{(a'')}$. This is a contradiction.
\end{proof}

\begin{proof}[Proof of~(\ref{cII})]
We need to show that
\begin{equation}
\label{eq:validity}
0 \leq \max J_i^{(a)} \leq p_i^{(a)}
\end{equation}
for all $(a,i) \in \HH_0$.
Considering the algorithm for $\delta$ and considering the change in vacancy numbers, the only way to violate Equation~\eqref{eq:validity} is when the following cases occur.
\begin{enumerate}
\item \label{part2_case1} There exists a singular or quasisingular string of length $i$ in $\nu^{(a)}$ such that $\widetilde{p}_i^{(a)} - p_i^{(a)} = -1, -2$ respectively and $\ell' \leq i < \ell$.
\item \label{part2_case2} We have $m_{\ell-1}^{(a)} = 0$, $p_{\ell-1}^{(a)} = 0$, and $\ell' < \ell$.
\end{enumerate}
In both cases, let $\ell$ be the length of the row selected by $\delta$ in $\nu^{(a)}$ and $\ell'$ be the length of the row selected immediately before $\ell$ in $\nu^{(c)}$ such that $c \sim a$.

We will first show that~(\ref{part2_case1}) cannot occur. Assume a singular string occurs, then $\delta$ would have selected the string of length $i$, which is a contradiction. Now suppose a quasisingular string occurs which corresponds to $\widetilde{p}_i^{(a)} - p_i^{(a)} = -2$, in this case we have $a = 1$ and again, the map $\delta$ would have selected this quasisingular string, which is a contradiction.

Now we will show that~(\ref{part2_case2}) cannot occur. Let $t < \ell$ be the largest integer such that $m_t^{(a)} > 0$, and if no such $t$ exists, set $t = 0$. By the convexity condition, we have $p_i^{(a)} = p_{\ell-1}^{(a)}$ for all $t < i < \ell$. Thus $m_i^{(c)} = 0$ for all $c \sim a$ and $t \leq i \leq \ell$. Thus $\ell' < t$. If $t = 0$, then this is a contradiction since $1 \leq \ell'$. If $t > 0$, then we have $p_t^{(a)} = 0$ and the row must be singular because $m_t^{(a)} > 0$. Thus we must have $t = \ell$, which is a contradiction.
\end{proof}

\begin{proof}[Proof of~(\ref{cIII})]
From the change in vacancy numbers, any string of length not of a length created from $\delta_{\theta}$ becomes non-singular. Therefore it is easy to see that the procedure outlined for $\delta_{\theta}^{-1}$ is the inverse of $\delta_{\theta}$.
\end{proof}

\begin{proof}[Proof of~(\ref{cIV})]
Recall that $(\overline{\nu}, \overline{J}) = \delta_{\theta}(\nu, J)$.
We can rewrite Equation~\eqref{eq:cocharge_configurations} using Equation~\eqref{eq:vacancy} as
\begin{equation}
\label{eq:cocharge_vacancy}
\cc(\nu) = \frac{1}{2} \sum_{(a,i) \in \HH_0}  t_a^{\vee} \left( \sum_{j \in \ZZ_{>0}} L_j^{(a)} \min(i, j) - p_i^{(a)} \right) m_i^{(a)}
\end{equation}
Next we express $(\ell_k)_k$ as $(\ell_m^{(a_k)})_k$ by the $m$-th selection in $\nu^{(a_k)}$ by $\delta$. Let $\overline{m}_i^{(a)}$, $\overline{p}_i^{(a)}$, $\overline{J}_i^{(a)}$, and $\overline{L}_i^{(a)}$ denote $m_i^{(a)}$, $p_i^{(a)}$, $J_i^{(a)}$, and $L_i^{(a)}$, respectively, after applying $\delta$, and we have
\begin{subequations}
\begin{align}
\overline{m}_i^{(a)} & = m_i^{(a)} + \sum_m \delta_{\ell^{(a)}_m-1,i} - \delta_{\ell^{(a)}_m, i},
\\ \overline{L}_i^{(a)} & = L_i^{(a)} - \delta_{a, N_{\g}} \delta_{i, 1},
\end{align}
\end{subequations}
Let $\Delta\bigl( \cc(\nu, J) \bigr) := \cc(\nu, J) - \cc(\overline{\nu}, \overline{J})$.
Next, from a straight-forward, but somewhat tedious, calculation using the changes in vacancy numbers, we obtain
\begin{align*}
\Delta\bigl(\cc(\nu)\bigr) & = \sum_{(a,i) \in \HH_0} t_a^{\vee} \left( p_i^{(a)} - \overline{p}_i^{(a)} \right) \left( m_i^{(a)} - \sum_m \delta_{\ell_m^{(a)},i} \right) - \chi(r = y_a)
\\ & \hspace{20pt} + \frac{t_{N_{\g}}^{\vee}}{c_0^{\vee}} \sum_{i \in \ZZ_{\geq 0}} m_i^{(N_{\g})} - \chi(r = \emptyset),
\end{align*}
where $r$ is the return value of $\delta_{\theta}$.
Moreover, from the description of $\delta_{\theta}$, we have
\[
\Delta\bigl(\lvert J \rvert\bigr)
= \sum_{(a,i) \in \HH_0} t_a^{\vee} \left( p_i^{(a)} - \overline{p}_i^{(a)} \right) \left( m_i^{(a)} - \sum_m \delta_{\ell_m^{(a)},i} \right) + \chi(r = y_a),
\]
where the last term is because a quasisingular string was changed into a singular string if that holds.
Combining this with the fact that $b_N = r$, we have
\[
\Delta\bigl(\cc(\nu, J)\bigr) = \Delta\bigl(\cc(\nu)\bigr) + \Delta\bigl(\lvert J \rvert\bigr) = \frac{t_{N_{\g}}^{\vee}}{c_0^{\vee}} \beta_1^{(N_{\g})} - \chi(b_N = \emptyset)
\]
as desired.
\end{proof}

\begin{table}
\[
\begin{array}{c@{\ \ }c@{\ \ }c@{\ \ }c@{\ \ }c@{\ \ }c@{\ \ }c@{\ \ }c@{\ \ }c}
\toprule
b' \otimes b & x_{\theta} \otimes x_{\theta} & x_{\theta^-} \otimes x_{\theta} & \emptyset \otimes x_{\theta} & x_{\theta} \otimes \emptyset & x_{\alpha} \otimes x_{\theta} & y_a \otimes x_{\theta} & x_{-\theta} \otimes x_{\theta} & \emptyset \otimes \emptyset
\\ \midrule
H(b' \otimes b) & 0 & -A_{0,N_{\g}} & 1 & 1 & 2 & 2 & 2 & 2
\\ \bottomrule
\end{array}
%\begin{array}{|c|c|}
%\hline b' \otimes b & H(b' \otimes b) \\\hline
%x_{\theta} \otimes x_{\theta} & 0 \\
%x_{\theta-\alpha_{N_{\g}}} \otimes x_{\theta} & 1 \\
%\emptyset \otimes x_{\theta} & 1 \\
%x_{\theta} \otimes \emptyset & 1 \\
%x_{\alpha} \otimes x_{\theta} & 2 \\
%y_a \otimes x_{\theta} & 2 \\
%x_{-\theta} \otimes x_{\theta} & 2 \\
%\emptyset \otimes \emptyset & 2 \\
%\hline
%\end{array}
\]
\caption{Local energy on $B^{\theta,1} \otimes B^{\theta,1}$, where $\theta^- = \theta-\alpha_{N_{\g}}$ and $\alpha \in \Delta^+ \setminus \{ \theta^- \}$ is such that $\alpha + \alpha_{N_{\g}} \in \Delta^+$.}
\label{table:adjoint_local_energy}
\end{table}

\begin{proof}[Proof of~(\ref{cV})]
We recall the local energy function on $B^{\theta,1} \otimes B^{\theta,1}$ from~\cite{BFKL06} in Table~\ref{table:adjoint_local_energy}, renormalized to our convention. We note that there are two minor typos in~\cite{BFKL06}, where it is stated $H(x_{\alpha} \otimes x_{\theta}) = 1$ and $H(x_{\theta-\alpha_{N_{\g}}} \otimes x_{\theta}) = 1$ (the only difference is for types $D_{n+1}^{(2)}$ and $A_{2n}^{(2)}$).

Note that in order for the second application of $\delta_{\theta}$ to return $x_{\theta}$ or $\emptyset$, we must not have any singular rows in $\nu^{(N_{\g})}$. Hence, $\beta_1^{(N_{\g})} - \overline{\beta}_1^{(N_{\g})}$ is the number of $N_{\g}$-arrows in the path from $x_{\theta}$ to $b'$ from the definition of $\delta_{\theta}$. We note that for every type, we have $c_{N_{\g}} = 2$, and hence, it is straightforward to see the claim holds in this case.

When $b' \otimes b$ is not classically highest weight, this is similar to the non-highest weight case in the proof of Lemma~\ref{lemma:minuscule_local_energy_change}. % as Lemma~\ref{lemma:difference} holds with minor modifications for when the return value is $y_a$.
\end{proof}

We can also extend our proof to types $C_n^{(1)}$, $A_{2n}^{(2)}$, and $A_{2n}^{(2)\dagger}$ (recall Remark~\ref{rem:adjoint_extensions}).
We note that for the proofs above to hold for type $A_{2n}^{(2)}$ and $A_{2n}^{(2)\dagger}$. We have to use the (modified) scaling factors of $\widetilde{\gamma}_a = 1$ for all $a \in I$ and the matrix
\[
(\widetilde{A}_{ab})_{a,b\in I_0} = \begin{bmatrix}
2 & -1 & & & \\
-1 & 2 & -1 & & \\
& \ddots & \ddots & \ddots & \\
& & -1 & 2 & -1 \\
& & & -1 & 1
\end{bmatrix}
\]
instead of the Cartan matrix in, \textit{e.g.}, Equation~\eqref{eq:necessary_convexity} and Equation~\eqref{eq:ell_inequality}.

%%%%%%%%%%
\subsection{Onwards and upwards}

We have shown that $\delta$ induces a bijection $\Phi$. Now we need to show $\lb$ induces a bijection. Note that $\lb$ is clearly a (classically strict) crystal embedding. Let $\sigma$ denote the unique sink of the $\lb$-diagram.

\begin{lemma}
\label{lemma:lb_fixes_vacancy_numbers}
The vacancy numbers are invariant under $\lb$.
\end{lemma}

\begin{proof}
Given an arrow $r \xrightarrow[\hspace{20pt}]{b} r'$, note that $\wt(b) = \clfw_r - \clfw_{r'}$. Since $\lb$ adds rows of length $1$ and the number of rows in $\nu^{(a)}$ corresponds to the coefficient of $\alpha_a$ in $\clfw_{\sigma} - \wt(b) = \clfw_{\sigma} - \clfw_r + \clfw_{r'}$, the claim follows from the definition of vacancy numbers.
\end{proof}

\begin{proposition}
\label{prop:lb_bijection_reduction}
Suppose $r \neq \sigma$, and let $r \xrightarrow[\hspace{20pt}]{b_r} r'$ be the outgoing arrow of $r$ in the $\lb$-digram.
Let $B = B^{r,1} \otimes B^{\bullet}$.
Then we have
\[
\lb \circ \Phi = \Phi \circ \lb.
\]
%Then the diagram
%\[
%\xymatrixrowsep{3pc}
%\xymatrixcolsep{3.5pc}
%\xymatrix{\RC(B^{r,1} \otimes B^{\bullet}) \ar[r]^-{\lb} \ar[d]_{\Phi} & \RC(B^{\sigma,1} \otimes B^{r',1} \otimes B^{\bullet}) \ar[d]^{\Phi}
%\\ B^{r,1} \otimes B^{\bullet} \ar[r]_-{\lb} & B^{\sigma,1} \otimes B^{r',1} \otimes B^{\bullet}}
%\]
%commutes.
\end{proposition}

\begin{proof}
Suppose the claim holds for $r'$ (so $\Phi$ is a bijection for $B^{\sigma,1} \otimes B^{r',1} \otimes B^{\bullet}$ by induction).
It is sufficient to show the claim on classically highest weight elements.
Consider some classically highest weight element $b \otimes b^{\bullet} \in B^{r,1} \otimes B^{\bullet}$, and let $\lb(b \otimes b^{\bullet}) = b'' \otimes b' \otimes b^{\bullet}$.
If $b'' = b_r$, then $b = u(B^{r,1})$ and $b' = u(B^{r',1})$ from the definition of $\lb$ and there is a unique element of maximal weight. Therefore, $\Phi^{-1}(b^{\bullet})$ has the same partitions and riggings as $\Phi^{-1}(b' \otimes b^{\bullet})$ (we will see below that applying $\lb^{-1}$ removes the boxes added by $\delta_{\sigma}^{-1}$ with $b_r$). Now since $b' = u(B^{r',1})$, the vacancy numbers $p_i^{(r')}$ of $(\nu, J) := \Phi^{-1}(b' \otimes b^{\bullet})$ are all one larger than those of $\Phi^{-1}(b^{\bullet})$. Thus, there are no singular rows in $\nu^{(r')}$. Hence, when we add $b_r$, we only add singular rows of length $1$ as the first box added under $\delta_{\sigma}^{-1}(b_r)$ has to be to a row of length $0$ in $\nu^{(r)}$ and $\delta_{\sigma}^{-1}$ adds boxes to weakly shorter rows at each subsequent step.
Hence, the result is in the image of $\lb$, and hence $\lb^{-1}$ can be applied.

Next, we proceed by induction on the depth of $b'' \otimes b'$ to $b_r \otimes u(B^{r',1})$ (\textit{i.e.}, the number of $e_a$ operators, for $a \in I_0$, one needs to apply). Suppose $f_a(b'' \otimes b') = f_a b'' \otimes b'$, then the claim follows by induction as one additional box is added by $\delta_{\sigma}$. Indeed, $\delta_{\sigma}^{-1}(f_a b'')$ could only have increased the number of length-one singular rows added compared to $\delta_{\sigma}^{-1}(b'')$ as $\delta_{\sigma}^{-1}(f_a b'')$ adds an extra box to some singular row in $\nu^{(a)}$ during its initial step and the subsequent selected rows must be weakly smaller. Thus every row selected for $\delta_{\sigma}^{-1}(f_a b'')$ must be weakly smaller than the corresponding one from $\delta_{\sigma}^{-1}(b'')$.
Now instead assume $f_a(b'' \otimes b') = b'' \otimes f_a b'$.
Let $(\nu, J) = \Phi^{-1}(b' \otimes b^{\bullet})$ and $(\tnu, \tJ) = \Phi^{-1}(f_a b' \otimes b^{\bullet})$.
For $b^{\dagger} \in B^{r',1}$, to ease notation define $\delta_{\sigma}^{-1}(b^{\dagger}) := \delta_{\sigma}^{-1}(b_{\sigma}^{\dagger})$, where $\lb(b^{\dagger}) = b_{\sigma}^{\dagger} \otimes \overline{b}^{\dagger}$.
As in the previous case, the rows selected by $\delta_{\sigma}^{-1}(f_a b')$ are at most as long as those added by  $\delta_{\sigma}^{-1}(b')$.
From the tensor product rule, we must have $0 = \varepsilon_a(b'') < \varphi_a(b') = 1$, and so the first row that we add a box for $\delta_{\sigma}^{-1}(b'')$ cannot be in $\nu^{(a)}$. By considering how the vacancy numbers change after applying $\delta_{\sigma}^{-1}(b')$, the first box added by $\delta_{\sigma}^{-1}(b'')$ to $(\tnu, \tJ)$ can be at a row at most as long as the one selected by $\delta_{\sigma}^{-1}(b'')$ for $(\nu, J)$. Hence, the rows selected by $\delta_{\sigma}^{-1}(b'')$ in $(\tnu, \tJ)$ must be at most as long as those selected for $(\nu, J)$. Therefore, we can apply $\lb^{-1}$.
\end{proof}

\begin{lemma}
\label{lemma:cocharge_left_top}
Let $B = B^{r,1} \otimes \bigotimes_{i=1}^N B^{r_i, s_i}$. For $(\nu, J) \in \RC(B)$, we have
\[
\cc\bigl( \lb(\nu, J) \bigr) - \cc(\nu, J) = \sum_{k=1}^{\ell} t_{a_k}^{\vee} \sum_{j \in \ZZ_{>0}} L_j^{(a_k)}
+ \frac{1}{2} \sum_{k=1}^{\ell} t_{a_k}^{\vee} (\delta_{a_k r'} + \delta_{a_k \sigma} - \delta_{a_k r}),
\]
where we have the path $b = f_{a_1} f_{a_2} \dotsm f_{a_{\ell}} u_{\Lambda_{\sigma}}$ for the arrow $r \xrightarrow[\hspace{15pt}]{b} r'$ in the $\lb$-diagram.
\end{lemma}

\begin{proof}
Recall from Equation~\eqref{eq:cocharge_vacancy} that the cocharge of a configuration can be expressed as
\[
\cc(\nu) = \frac{1}{2} \sum_{(a,i) \in \HH_0} t_a^{\vee} \left( \sum_{j \in \ZZ_{>0}} L_j^{(a)} \min(i, j) - p_i^{(a)} \right) m_i^{(a)}. %+ \sum_{x \in J_i^{(a)}} x
\]
Since $\lb$ adds a singular string of length 1 to $\nu^{(a_k)}$ for all $1 \leq k \leq \ell$, where we have $b = f_{a_1} f_{a_2} \dotsm f_{a_{\ell}} u_{\theta}$, for the arrow $r \xrightarrow[\hspace{15pt}]{b} r'$ in the corresponding $\lb$-diagram.
A straightforward tedious calculation gives
\begin{align*}
K & := \sum_{(a,i) \in \HH_0} t_a^{\vee} \sum_{j \in \ZZ_{>0}} \left( \widetilde{L}_j^{(a)} \widetilde{m}_i^{(a)} - L_j^{(a)} m_i^{(a)} \right) \min(i, j)
%\\ & = \sum_{(a,i) \in \HH_0} t_a^{\vee} \sum_{j \in \ZZ_{>0}} \left( \left[L_j^{(a)} + \delta_{j,1} (\delta_{ar'} + \delta_{a \sigma} - \delta_{ar})\right] \left[m_i^{(a)} + \delta_{i,1} \sum_{k=1}^{\ell} \delta_{a a_k}\right] - L_j^{(a)} m_i^{(a)} \right) \min(i, j)
%\\ & = \sum_{(a,i) \in \HH_0} t_a^{\vee} \sum_{j \in \ZZ_{>0}} \Biggl( L_j^{(a)} m_i^{(a)} + \delta_{j,1} (\delta_{ar'} + \delta_{a \sigma} - \delta_{ar}) m_i^{(a)}
%\\ & \hspace{40pt} + L_j^{(a)} \delta_{i,1} \sum_{k=1}^{\ell} \delta_{a a_k} + \delta_{j,1} (\delta_{ar'} + \delta_{a \sigma} - \delta_{ar}) \delta_{i,1} \sum_{k=1}^{\ell} \delta_{a a_k} - L_j^{(a)} m_i^{(a)} \Biggr) \min(i, j)
\\ & = \sum_{(a,i) \in \HH_0} t_a^{\vee} (\delta_{ar'} + \delta_{a \sigma} - \delta_{ar}) m_i^{(a)}
 + \sum_{k=1}^{\ell} t_{a_k}^{\vee} \sum_{j \in \ZZ_{>0}} L_j^{(a_k)}
\\ & \hspace{40pt} + \sum_{k=1}^{\ell} t_{a_k}^{\vee} (\delta_{a_k r'} + \delta_{a_k \sigma} - \delta_{a_k r})
%\\ & = \sum_{(a,i) \in \HH_0} \left[ \sum_{k=1}^{\ell} (\clsr_{a_k} | \clsr_a) \right] m_i^{(a)}
% + \sum_{k=1}^{\ell} t_{a_k}^{\vee} \sum_{j \in \ZZ_{>0}} L_j^{(a_k)}
% + \sum_{k=1}^{\ell} t_{a_k}^{\vee} (\delta_{a_k r'} + \delta_{a_k \sigma} - \delta_{a_k r})
\\ & = -\sum_{k=1}^{\ell} t_{a_k}^{\vee} p_1^{(a_k)} + 2\sum_{k=1}^{\ell} t_{a_k}^{\vee} \sum_{j \in \ZZ_{>0}} L_j^{(a_k)}
 + \sum_{k=1}^{\ell} t_{a_k}^{\vee} (\delta_{a_k r'} + \delta_{a_k \sigma} - \delta_{a_k r}).
\end{align*}
Note that during the derivation, we used Lemma~\ref{lemma:lb_fixes_vacancy_numbers} and note that $\min(i, j) = 1$ if either $i = 1$ or $j = 1$. Therefore, we have
\begin{align*}
\cc\bigl( \lb(\nu, J) \bigr) - \cc(\nu, J) & = \frac{1}{2} \left( K - \sum_{k=1}^{\ell} t_{a_k}^{\vee} p_1^{(a_k)} \right) + \sum_{k=1}^{\ell} t_{a_k}^{\vee} p_1^{(a_k)}
\\ & = \sum_{k=1}^{\ell} t_{a_k}^{\vee} \sum_{j \in \ZZ_{>0}} L_j^{(a_k)}
+ \frac{1}{2} \sum_{k=1}^{\ell} t_{a_k}^{\vee} (\delta_{a_k r'} + \delta_{a_k \sigma} - \delta_{a_k r})
\end{align*}
as desired.
\end{proof}

We first need a type-independent proof of~\cite[Thm.~6.1]{ST12}, which we show for $s = 1$. This is essentially the same proof as the proof of~\cite[Thm.~9.4]{Naoi12}.

\begin{proposition}
\label{prop:demazure_description}
Let $B = \bigotimes_{i=1}^N B^{r_i,1}$ for dual untwisted types, and let $\lambda = - \sum_{i=1}^N \Lambda_{\xi(r_i)}$. Let $t_{\lambda}$ denote the translation by $\lambda$ in the extended affine Weyl group, and let $t_{\lambda} = v \tau$, where $v$ is an element in the affine Weyl group. Then there exists a $U_q'(\g)$-crystal isomorphism
\[
j \colon B(\Lambda_{\tau(0)}) \to B \otimes B(\Lambda_0)
\]
such that the image of the Demazure subcrystal $B_v(\Lambda_{\tau(0)})$ is $B \otimes u_{\Lambda_0}$.
\end{proposition}

\begin{proof}
Recall from~\cite{L95,L95-2} that any Littelmann path that stays in the dominant chamber with an endpoint of $\lambda$ generates $B(\lambda)$. The claim follows from using (projected) level-zero LS paths~\cite{NS06II,NS08II,NS08}, that tensor products in the Littelmann path model corresponding to concatenation of the LS paths~\cite{L95,L95-2}.
\end{proof}

\begin{lemma}
\label{lemma:right_bottom_energy}
Let $B = \bigotimes_{i=1}^N B^{r_i, 1} \otimes B^{r,1}$.
For $b \in B$, we have
\[
D\bigl( \rb(b) \bigr) - D(b) = \sum_{k=1}^{\ell} t_{a_k}^{\vee} \left( \sum_{j \in \ZZ_{>0}} L_j^{(a_k)}
+ \frac{1}{2} (\delta_{a_k r'} + \delta_{a_k \sigma} - \delta_{a_k r}) \right),
\]
where $b = f_{a_1} f_{a_2} \dotsm f_{a_{\ell}} u_{\clfw_{\sigma}}$, for the arrow $r_N \xrightarrow[\hspace{15pt}]{b} r'_N$.
\end{lemma}

\begin{proof}
Let $u$ denote the maximal vector in $B$. Note that $u^{\lusztig} \in B$ is the unique element of classical weight $w_0 \wt(u)$ and is in the same classical component as $u$.

Let $B^{\bullet} = \bigotimes_{i=1}^N B^{r_i, 1}$, and denote the maximal vector in $B$ by $u_{\bullet} \otimes u_r$, where $u_{\bullet}$ and $u_r$ are the maximal elements in $B^{\bullet}$ and $B^{r,1}$ respectively.
We note that the unique element of classical weight $w_0 \wt(u)$ is given by $v = u_{\bullet}^{\lusztig} \otimes u_r^{\lusztig}$ (technically we apply $\lusztig$ to each factor individually, but we do not care about the ordering of $B^{\bullet}$ and $u_{\bullet}$).
We note that $v$ is in the same classical component as $u_{\bullet} \otimes u_r$.
Therefore, $v^{\lusztig}$ is the maximal element in $B^{r,1} \otimes B^{\bullet}$.
%$B^{\xi(r),1} \otimes (B^{\bullet})^{\lusztig}$ (recall that $(B^{r,1})^{\lusztig} \iso B^{\xi(r),1}$).

Therefore, we have
\begin{align*}
\lb(v^{\lusztig}) & = b \otimes u_{r'} \otimes u_{\bullet},
\\ \lb(v^{\lusztig})^{\lusztig} & = u_{\bullet}^{\lusztig} \otimes u_{\xi(r')} \otimes b^{\lusztig} =: b_0,
\end{align*}
where
\[
b^{\lusztig} = f_{\xi(a_1)} f_{\xi(a_2)} \dotsm f_{\xi(a_{\ell})} u_{\clfw_{\sigma}}.
\]
Next, we want to apply a sequence $e_{\vec{\jmath}_k}$ to $b_{k-1}$ such that we obtain
\[
b_k := u_{\bullet}^{\lusztig} \otimes u_{\xi(r')} \otimes \left( f_{\xi(a_{k+1})} \dotsm f_{\xi(a_{\ell})} u_{\clfw_{\sigma}} \right),
\]
where $b_{\ell} = u_{\bullet}^{\lusztig} \otimes u_{\xi(r')} \otimes u_{\clfw_{\sigma}}$.

Fix some $j \in I_0$ and $i$. Let $I_{r_i} = I \setminus \{r_i\}$.
If we have $v_{r_i} \otimes x$, where $v_{r_i}$ is the lowest weight element in $B^{r_i,1}$, then if $r_i \neq j$, we have $e_j(v_{r_i} \otimes x) = v_{r_i} \otimes e_j x$.
If $r_i = j$, then $e_j^2(v_{r_i} \otimes x) = (e_j v_{r_i}) \otimes (e_j x)$, but we note that there exists a sequence of crystal operators $e_{\vec{j}}$ that acts only on the left that returns to $v_i \otimes (e_j x)$ that uses precisely $t_{a_k}$ $0$-arrows. This follows from the fact that $e_j v_i$ is the unique $I_{r_i}$-lowest weight element in the unique $I_{r_i}$-component $B(\check{\Lambda}_0) \subseteq B^{r,1}$. Note that these $0$-arrows are \defn{Demazure}, not at the beginning of a $0$-string, as the extra $B^{\sigma,1}$ results in $\varphi_0(\lb(v^{\lusztig})^{\lusztig}) - \varphi_0(v) = 1$.

From~\cite[Lemma~7.3]{ST12}, the number of Demazure $0$-arrows, a $0$-arrow that is not at the end of a $0$-string, between two elements determines the difference in energy. Therefore, we have
\[
D\bigl( \rb(v) \bigr) - D(v) = \sum_{k=1}^{\ell} t_{a_k}^{\vee} \sum_{j \in \ZZ_{>0}} \widetilde{L}_j^{(a_k)} + S,
\]
where $\widetilde{L}$ corresponds to $B^{\bullet}$ and $S$ is the claim for a single tensor factor $B^{r,1}$. Thus, it remains to show the claim for a single tensor factor, which is a finite computation using the results from~\cite{LNSSS14,LNSSS14II}.
Thus the claim holds for any element in the same classical component of $v$.

Let $d(b)$ denote the affine grading of $b$~\cite[Def.~7.1]{ST12}, the number of Demazure $0$-arrows in the path from $b$ to $u_{\Lambda_{\tau(0)}}$ in $B(\Lambda_{\tau(0)})$.
Next, for any other classically lowest weight element $v'$, we have
$d\bigl( \rb(v') \bigr) = d(b')$ since the number of Demazure $0$-arrows is determined by the coefficient of $\alpha_0$ in $\Lambda_{\tau(0)} - \wt(v')$ and $\wt\bigl( \rb(v') \bigr) = \wt(v')$.
Therefore, we have
\begin{align*}
D\bigl(\rb(v')\bigr) - D(v') & = \bigl[ d\bigl(\rb(v')\bigr) - d(v_{\rb}) \bigr] - \bigl[ d(v') - d(v) \bigr]
\\ & = \bigl[ d\bigl(\rb(v')\bigr) - d(v') \bigr] + \bigl[ d(v) - d(v_{\rb}) \bigr]
\\ & = D\bigl(\rb(v)\bigr) - D(v),
\end{align*}
where $v_{\rb}$ is the minimal element in $\rb(B)$. Hence, the claim follows.
\end{proof}

\begin{remark}
The proof of Lemma~\ref{lemma:cocharge_left_top} and Lemma~\ref{lemma:right_bottom_energy} is true for general type $\g$ and \emph{any} $\lb$-diagram.
\end{remark}

\begin{remark}
The proof of Lemma~\ref{lemma:right_bottom_energy} holds for general tensor factors provided there is an analog of~\cite[Lemma~7.2]{ST12} and~\cite[Thm.~6.1]{ST12} for all $B^{r,s}$ in any type. In particular, we are using~\cite[Lemma~7.3]{ST12}, which is essentially a type-independent proof (other than the single factor case, \textit{i.e.},~\cite[Lemma~7.2]{ST12}).
\end{remark}

%%%%%%%%%%
\subsection{Statistic preserving bijection}

\begin{proposition}
\label{prop:R_matrix_invariance}
Let $\g$ be of type $A_n^{(1)}$, $D_n^{(1)}$, $E_{6,7,8}^{(1)}$, or $E_6^{(2)}$. Let $B = \bigotimes_{i=1}^N B^{r_i,1}$ and $B' = \bigotimes_{i=1}^{N'} B^{r_i',1}$ (except possibly for $r_i$ and $r_i'$ being $4,5$ in type $E_8^{(1)}$). Then the diagram
\[
\xymatrixrowsep{3pc}
\xymatrixcolsep{3.5pc}
\xymatrix{\RC(B \otimes B') \ar[r]^{\Phi} \ar[d]_{\id} & B \otimes B' \ar[d]^{R} \\ \RC(B' \otimes B) \ar[r]_{\Phi} & B' \otimes B}
\]
commutes.
\end{proposition}

\begin{proof}
For type $A_n^{(1)}$ (resp.~$D_n^{(1)}$) was shown in~\cite{KSS02} (resp.~\cite{OSSS16}).
For types $E_{6,7,8}^{(1)}$ and $E_6^{(2)}$, the claim reduces to $B^{r_1,1} \otimes B^{r_2,1}$ by the definition of $\Phi$ and is a finite computation.
\end{proof}

Now we can show that the map $\Phi$ is a bijection that sends cocharge to energy.

\begin{theorem}
\label{thm:bijection}
Let $B = \bigotimes_{i=1}^N B^{r_i,1}$ be a tensor product of KR crystals. Then the map
\[
\Phi \colon \RC(B) \to B
\]
is a bijection on highest weight elements such that $\Phi \circ \eta$ sends cocharge to energy.
\end{theorem}

\begin{proof}
By Proposition~\ref{prop:lb_bijection_reduction}, Lemma~\ref{lemma:cocharge_left_top}, and Lemma~\ref{lemma:right_bottom_energy}, showing the bijection is well-defined and preserves statistics is reduced to showing when the left-most factor for $\RC(B)$ and right-most factor for $B^{\mathrm{rev}}$ is minuscule or adjoint. Thus the claim follows by Lemma~\ref{lemma:minuscule_bijections} and Lemma~\ref{lemma:adjoint_bijections}.
\end{proof}

\begin{theorem}
\label{thm:minuscule_classical_isomorphism}
Let $\Phi$ be defined using $\delta_r$ such that $r$ is a minuscule node. Then $\Phi \colon \RC(B) \to B$ is a $U_q(\g_0)$-crystal isomorphism.
\end{theorem}

\begin{proof}
A straightforward check shows that $\Phi$ preserves weights. A tedious but straightforward check shows that the arguments in~\cite{Sakamoto14} extend to all minuscule nodes by the description of $\delta_r$ and that the arguments are about the relative position where two boxes are added by applying $f_a$.
\end{proof}

\begin{remark}
\label{rem:universal_R_construction}
Let $B$ be as in Theorem~\ref{thm:minuscule_classical_isomorphism} and $B'$ be some reordering of the tensor factors.
From Proposition~\ref{prop:R_matrix_invariance} and Theorem~\ref{thm:minuscule_classical_isomorphism}, we can construct the combinatorial $R$-matrix $R \colon B \to B'$ by $R = \Phi' \circ \Phi^{-1}$, where $\Phi' \colon \RC(B') \to B'$ is the corresponding bijection. Note that this provides a uniform combinatorial construction of the combinatorial $R$-matrix.
\end{remark}

For $r = 1, 2, 6$ in type $E_6^{(1)}$, we can describe $e_0$ and $f_0$ on $\RC(B^{r,s})$ by using the description in~\cite{JS10} as it is given solely in terms of $\varepsilon_i$, $\varphi_i$, and the weight. Hence, Theorem~\ref{thm:minuscule_classical_isomorphism} immediately implies the following.

\begin{corollary}
Let $\g$ be of type $E_6^{(1)}$ and $r = 1, 2, 6$. Then
\[
\Phi \colon \RC(B^{r,s}) \to B^{r,s}
\]
is a $U_q'(\g)$-crystal isomorphism.
\end{corollary}

%%%%%%%%%%
\subsection{Virtual bijection}

In order to extend the bijection to full columns in all other types, we need to extend $\lb$ to commute with virtualization maps. In particular, we introduce the notion of a \defn{virtual $\lb$ map}, which we denote by $\lb^v$. Specifically, we generalize the notion of an $\lb$-diagram to use for the map $\lb^v$, which we call a $\lb^v$-diagram. In all cases below, the resulting $\lb^v$-diagram is derived from a virtualization map.

For type $E_6^{(2)}$ as a folding of type $E_6^{(1)}$, we require having an arrow $(r_1, r_2) \longrightarrow r'$ defining a map $\lb^v \colon B^{r_1,1} \otimes B^{r_2,1} \to B^{2,1} \otimes B^{r',1}$. Hence, the virtual $\lb^v$-diagram we use is
\begin{equation}
\label{eq:lbv_E6_1}
\begin{tikzpicture}[baseline=-4]
\node (1) at (0,0) {$2$};
\node (2) at (2,0) {$4$};
\node (3) at (4,0) {$(3,5)$};
\node (4) at (2,1) {$(1,6)$};
\draw[->] (2) -- (1) node[midway, below]{\small $\btw 4$};
\draw[->] (3) -- (2) node[midway, below]{\small $\bfo 35$};
\draw[->] (4) -- (1) node[midway, above left]{\small $\btw 16$};
\end{tikzpicture}.
\end{equation}

For type $B_n^{(1)}$ as the dual of $A_{2n-1}^{(2)}$ (this has a virtualization map with scaling factors given by Table~\ref{table:scaling_factors}), the $\lb^v$-diagram we consider is
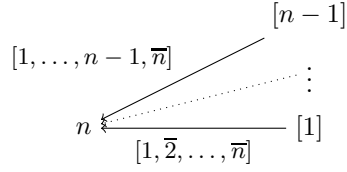
\[
\begin{tikzpicture}[baseline=-4]
\node (n) at (0,0) {$n$};
\node (1) at (3,0) {$[1]$};
\node (d) at (3,0.75) {$\vdots$};
\node (nm) at (3,1.5) {$[n-1]$};
\draw[->] (1) -- (n) node[midway, below]{\small $\column{1,\btw, \dotsc, \bn}$};
\draw[->,dotted] (d) -- (n);
\draw[->] (nm) -- (n) node[midway, above left]{\small $\column{1,\dotsc,n-1,\bn}$};
\end{tikzpicture}
\]
where $[r]$ corresponds to $B^{r,2}$. Note that the arrows are labeled by a single-column KR tableau $[t_1, \dotsc, t_k]$, where we read the column from top to bottom. See also Appendix~\ref{sec:KR_tableaux}.

For type $F_4^{(1)}$ as the dual of $E_6^{(2)}$, the $\lb^v$-diagram is a modification of Equation~\eqref{eq:lb_E6_2}:
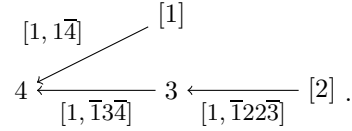
\[
\begin{tikzpicture}[baseline=-4]
\node (1) at (0,0) {$4$};
\node (2) at (2,0) {$3$};
\node (3) at (4,0) {$[2]$};
\node (4) at (2,1) {$[1]$};
\draw[->] (2) -- (1) node[midway, below]{\small $\column{1, \bon3\bfo}$};
\draw[->] (3) -- (2) node[midway, below]{\small $\column{1, \bon22\bth}$};
\draw[->] (4) -- (1) node[midway, above left]{\small $\column{1, 1\bfo}$};
\end{tikzpicture}.
\]
The derivation is similar to the type $B_n^{(1)}$ case.

For type $G_2^{(1)}$ as the dual of $D_4^{(3)}$, recall that we consider $B(\clfw_1)$ as the natural virtual crystal of $B(3\clfw_1) \subseteq B(\clfw_2) \otimes B(\clfw_2)$ in type $G_2^{(1)}$. Continuing this in type $D_4^{(3)}$, we construct an $\lb^v$-diagram as $2 \xrightarrow[\hspace{20pt}]{\column{1,3}} [[1]]$, where $[[1]]$ corresponds to $B^{1,3}$.

\begin{example}
Consider type $B_4^{(1)}$ and $B^{3,1}$. Note that the image under the virtualization map $B_4^{(1)} \lhook\joinrel\longrightarrow A_7^{(2)}$ is $B^{3,2}$. Then we have
\[
\lb^v\left(\young(1,2,3)\right) = \young(1,2,3,\bfo) \otimes \young(1,2,3,4) \in B^{4,1} \otimes B^{4,1}.
\]
\end{example}

\begin{proposition}
\label{prop:type_A_virtualization}
Let $v$ be one of the virtualization maps of
\begin{equation}
\label{eq:flipping_arrows}
C_n^{(1)} \lhook\joinrel\longrightarrow A_{2n}^{(2)} \lhook\joinrel\longrightarrow D_{n+1}^{(2)}\lhook\joinrel\longrightarrow A_{2n-1}^{(1)},
\end{equation}
with scaling factors $(2, 1, \dotsc, 1)$, $(1, \dotsc, 1, 2)$, and $(1, \dotsc, 1)$ respectively.
Then
\[
v \circ \delta_{\theta} = \virtual{\delta}_{\theta} \circ v.
\]
\end{proposition}

\begin{proof}
Let $\g$ be of type $D_{n+1}^{(2)}$ and $\virtual{\g}$ of type $A_{2n-1}^{(2)}$.
We note that if $\delta_{\theta}$ selects a singular string in $\nu^{(a)}$ for $a \neq n$, then it must select the same singular string in $\nu^{(a')} = \nu^{(a'')}$ for all $a', a'' \in \phi^{-1}(a)$.

For the remaining virtualization maps, this follows from the description of $\delta_{\theta}$ as per Section~\ref{sec:adjoint}.
\end{proof}

\begin{remark}
We can also compose the virtualization maps of Equation~\eqref{eq:flipping_arrows}, and we obtain another proof of~\cite[Thm.~7.1]{OSS03II}.
\end{remark}

We also can use
\[
C_n^{(1)} \lhook\joinrel\longrightarrow A_{2n}^{(2)\dagger} \lhook\joinrel\longrightarrow D_{n+1}^{(2)}\lhook\joinrel\longrightarrow A_{2n-1}^{(1)},
\]
with scaling factors $(1, \dotsc, 1, 2)$, $(2, 1, \dotsc, 1)$, and $(1, \dotsc, 1)$, respectively, instead of Equation~\eqref{eq:flipping_arrows}.

For type $D_{n+1}^{(2)}$ as a folding of type $A_{2n-1}^{(1)}$, we use the $\lb^v$-diagram:
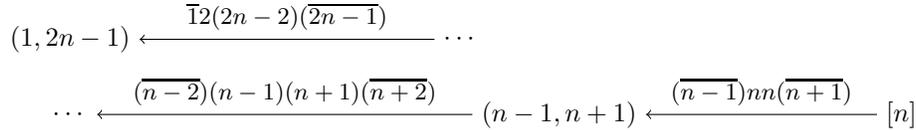
\[
\label{eq:lb_general_virtual}
\begin{tikzpicture}[baseline=-4, xscale=1.3]
\node (1) at (0,0) {$(1,2n-1)$};
\node (2) at (4,0) {$\cdots$};
\node (3) at (0,-1) {$\cdots$};
\node (4) at (5,-1) {$(n-1,n+1)$};
\node (5) at (8.5,-1) {$[n]$};
\draw[->] (2) -- (1) node[midway, above]{\small $\bon2(2n-2)(\overline{2n-1})$};
\draw[->] (4) -- (3) node[midway, above]{\small $(\overline{n-2})(n-1)(n+1)(\overline{n+2})$};
\draw[->] (5) -- (4) node[midway, above]{\small $(\overline{n-1})nn(\overline{n+1})$};
\end{tikzpicture}.
\]

\begin{lemma}
\label{lemma:virtual_lt}
Let $\g$ be of non-simply-laced affine type.
Then we have
\[
v \circ \lb = \lb^v \circ v,
\]
where $v$ is one of the virtualization maps given above.
\end{lemma}

\begin{proof}
Suppose $\g$ is of type $E_6^{(2)}$. Recall that $\lb$ adds a length 1 singular string to $\virtual{\nu}^{(a_i)}$, where $e_{a_1} \dotsm e_{a_k} b = u_{\clfw_2}$. Moreover, recall that $\lb$ does not change the vacancy numbers. It is a finite computation to show that for any arrow $r \xrightarrow[\hspace{20pt}]{b} r'$ in the $\lb$-diagram in Equation~\eqref{eq:lb_E6_2}, we have an arrow $\phi^{-1}(r) \xrightarrow[\hspace{20pt}]{v(b)} \phi^{-1}(r')$, where $\phi \colon E_6^{(1)} \searrow E_6^{(2)}$ is the diagram folding, in the $\lb^v$-diagram in Equation~\eqref{eq:lbv_E6_1}. Thus the claim follows.

For the other types, the proof is similar with also considering the doubling (or tripling) map.
\end{proof}

\begin{theorem}
\label{thm:virtual_bijection}
Let $\g$ be of non-simply-laced affine type. Then
\[
v \circ \Phi = \virtual{\Phi} \circ v,
\]
where $v$ is one of the virtualization maps given above.
Moreover, $\Phi$ is a bijection such that $\Phi \circ \eta$ sends cocharge to energy.
\end{theorem}

\begin{proof}
Lemma~\ref{lemma:virtual_lt} implies that it is sufficient to show $v \circ \delta = \virtual{\delta} \circ v$.
For untwisted types, this follows by definition and Theorem~\ref{thm:virtual_untwisted_delta}. For all dual untwisted types, type $A_{2n}^{(2)}$, and type $A_{2n}^{(2)\dagger}$, the proof is similar to the proof of Proposition~\ref{prop:type_A_virtualization}. Thus the claim follows.
\end{proof}

\begin{proposition}
\label{prop:virtual_R_matrix}
Let $\g$ be of affine type. Let $B = \bigotimes_{i=1}^N B^{r_i,1}$, and let $B' = \bigotimes_{i=1}^{N'} B^{r_i',1}$. Then the diagram
\[
\xymatrixrowsep{3pc}
\xymatrixcolsep{3.5pc}
\xymatrix{v(B \otimes B') \ar[r]^{\virtual{R}} & v(B \otimes B') \\ B' \otimes B \ar[u]^{v} \ar[r]_{R} & B' \otimes B \ar[u]_{v}}
\]
commutes. Moreover, the combinatoral $R$-matrix can be defined as the restriction of $\virtual{R}$ to the image of $v$.
\end{proposition}

\begin{proof}
This follows from Proposition~\ref{prop:demazure_description}.
\end{proof}

\begin{corollary}
Proposition~\ref{prop:R_matrix_invariance} holds for all affine types.
\end{corollary}

\begin{proof}
This follows from Theorem~\ref{thm:virtual_bijection}, Proposition~\ref{prop:virtual_R_matrix}, and Proposition~\ref{prop:R_matrix_invariance}.
\end{proof}

%%%%%%%%%%
\subsection{Higher levels for minuscule nodes}

\begin{proposition}
Let $r$ be such that $\clfw_r$ is a minuscule weight. Then $B(s\clfw_r) \subseteq B(\clfw_r)^{\otimes s}$ is characterized by
\[
\{ b_1 \otimes \cdots \otimes b_s \mid b_1 \leq \cdots \leq b_s \}.
\]
\end{proposition}

\begin{proof}
By~\cite[Lemma~2.9]{JS10}, it is sufficient to show when $s = 2$. We show this by induction, where the base case is $x_1 = x_2 = u_{\clfw_r}$. Next, suppose $x_1 \leq x_2$, and let $x_2 = f_{a_m} \cdots f_{a_1} x_1$. Fix some $i \in I_0$.
Consider when $f_i(x_1 \otimes x_2) = x_1 \otimes (f_i x_2)$, then we have $x_1 \leq f_i x_2$.
Now, consider the case $f_i(x_1 \otimes x_2) = (f_i x_1) \otimes x_2$. Since $B(\clfw_r)$ is minuscule, all elements are parameterized by trailing words of the longest coset representative $w \in \clW / \clW_{\widehat{r}}$ and that $w$ is fully commutative~\cite[Prop.~2.1]{Stembridge01II}. Since $w$ is fully commutative, we must have $f_{a_1} f_i x = f_i f_{a_1} x$, and similarly, there exists a $k$ such that $a_k = i$ and
\[
x_2 = f_{a_m} \cdots f_{a_{k+1}} f_{a_{k-1}} \cdots f_{a_1} f_{a_k} x_1.
\]
Therefore, we have $f_i x_1 \leq x_2$.
\end{proof}

Next, we note that the proof given in~\cite[Lemma~8.2]{SS2006} that the left-split map $\ls \colon B^{r,s} \otimes B^{\bullet} \to B^{r,1} \otimes B^{r,s-1} \otimes B^{\bullet}$ is preserved under $\Phi$ is straightforward to generalize to when $r$ is a minuscule node. We also can show that $\Phi$ send the combinatorial $R$-matrix to the identity map on rigged configurations

\begin{proposition}
\label{prop:higher_R_matrix}
Let $B = \bigotimes_{i=1}^N B^{r_i,s_i}$ and $B' = \bigotimes_{i=1}^{N'} B^{r_i',s_i'}$, where $r_i$ and $r_i'$ are minuscule nodes for all $i$. Then the diagram
\[
\xymatrixrowsep{3pc}
\xymatrixcolsep{3.5pc}
\xymatrix{\RC(B \otimes B') \ar[r]^{\Phi} \ar[d]_{\id} & B \otimes B' \ar[d]^{R} \\ \RC(B' \otimes B) \ar[r]_{\Phi} & B' \otimes B}
\]
commutes.
\end{proposition}

\begin{proof}
By the description of $\Phi$ it is sufficient to reduce this to the case when $B = B^{r,s}$ and $B' = B^{r',s'}$.
Recall that $B^{r,s} \iso B(s\clfw_r)$ as $U_q(\g_0)$-crystal when $r$ is a special node.
Since $\clfw_r$ and $\clfw_{r'}$ are minuscule weights, it follows from~\cite{Stembridge03} that $B^{r,s} \otimes B^{r',s'} \iso B(s\clfw_r) \otimes B(s'\clfw_{r'})$ is multiplicity free.
%By the tableaux description of $B^{r,s}$ and $B^{r',s'}$ and the fact that $\clfw_r$ and $\clfw_{r'}$ are minuscule weights, it is straightforward to see that $B^{r',s'} \otimes B^{r,s}$ is multiplicity free. Since $\Phi$ preserves weights, the claim follows.
\end{proof}

Next, it is straightforward to see that on rigged configurations, we have $\rb$ acting as the identity map on the configuration $\nu$ and preserves the coriggings. Next, we need the following description of the dual version of $\delta_r$.

\begin{lemma}
\label{lemma:dual_removal}
Let $r$ be a minuscule node, and define $\delta_r^{\lusztig} := \eta \circ \delta_r \circ \eta$ on rigged configurations.
Suppose $B = B^{\bullet} \otimes B^{r,1}$ and $(\nu, J)$ is a highest weight rigged configuration.
Define $(\nu, \overline{J})$ by
\[
\overline{J}_i^{(a)} = \begin{cases}
J_i^{(a)} & \text{if } a \neq r, \\
\{x-1 \mid x \in J_i^{(r)} \} & \text{if } a = r,
\end{cases}
\]
for all $i \in \ZZ_{>0}$. Then we have
\[
\delta^{\lusztig}(\nu, J) = e_{a_1} \dotsm e_{a_k} (\nu, \overline{J}),
\]
where $a_1, \dotsc, a_k \in I_0$ and $e_{a_1} \dotsm e_{a_k} (\nu, \overline{J})$ is a highest weight rigged configuration.
\end{lemma}

\begin{proof}
On $(\nu, J)$, it is clear that $\delta_r^{\lusztig}$ proceeds the same as $\delta_r$ except by preserving \emph{co}riggings and selecting and keeping \defn{cosingular} rows: rows with a rigging of $0$. Thus if there exists a cosingular row in $(\nu, J)^{(r)}$ (which is the first partition we must select from under $\delta_r^{\lusztig}$), there exists a row with a negative rigging in $(\nu, \overline{J})^{(r)}$ and we can apply $e_r$. Moreover, $e_r$ removes a box from the smallest row in $(\nu, \overline{J})^{(r)}$, which matches the procedure for $\delta_r^{\lusztig}$, and decreases the riggings by $1$ in all weakly longer rows in $(\nu, \overline{J})^{(r')}$ for all $r' \sim r$. Hence, if there are any cosingular rows in $(\nu, J)^{(r')}$ selected by $\delta_r^{\lusztig}$ in the next step, we can also apply $e_{r'}$ to $e_r(\nu, \overline{J})$. By repeating this argument, we can clearly apply an $e_{a_j}$ for every cosingular row in $(\nu, J)^{(a_j)}$ selected by $\delta_r^{\lusztig}$. The fact this is an if and only if is similar to the proof that $\delta_r^{\lusztig}$ gives a bijection (see the proof in~\cite{OS12}).
From the definition of $e_{a_j}$ and the resulting change in vacancy numbers (recall that the classical crystal operators on rigged configurations preserve the coriggings of all unchanged rows), it is straightforward to see that the resulting rigged configurations are equal.
\end{proof}

\begin{example}
Consider $B = B^{3,1} \otimes B^{4,1} \otimes B^{2,2} \otimes B^{6,1}$ in type $E_6^{(1)}$. Then we have
\begin{align*}
(\nu, J) & = \begin{tikzpicture}[scale=.33,baseline=-18]
 \fill[lightgray] (0,-2) rectangle (1,-3);
 \node[scale=.7] at (0.6, -2.5) {$\ell_5$};
 \rpp{2,1}{0,0}{0,0}
 \begin{scope}[xshift=5cm]
 \fill[lightgray] (2,-1) rectangle (3,-2);
 \node[scale=.7] at (2.6, -1.5) {$\ell_6$};
 \rpp{3,2,1}{0,1,1}{0,1,1}
 \end{scope}
 \begin{scope}[xshift=11cm]
 \fill[lightgray] (0,-4) rectangle (1,-5);
 \node[scale=.7] at (0.6, -4.5) {$\ell_4$};
 \rpp{2,2,1,1}{0,0,0,0}{0,0,0,0}
 \end{scope}
 \begin{scope}[xshift=16cm]
 \fill[lightgray] (0,-6) rectangle (1,-7);
 \node[scale=.7] at (0.6, -6.5) {$\ell_3$};
 \fill[lightgray] (2,-1) rectangle (3,-2);
 \node[scale=.7] at (2.6, -1.5) {$\ell_7$};
 \rpp{3,2,2,1,1,1}{0,0,0,0,0,0}{0,0,0,0,0,0}
 \end{scope}
 \begin{scope}[xshift=22cm]
 \fill[lightgray] (0,-4) rectangle (1,-5);
 \node[scale=.7] at (0.6, -4.5) {$\ell_2$};
 \fill[lightgray] (2,-1) rectangle (3,-2);
 \node[scale=.7] at (2.6, -1.5) {$\ell_8$};
 \rpp{3,2,1,1}{0,0,0,0}{0,0,0,0}
 \end{scope}
 \begin{scope}[xshift=28cm]
 \fill[lightgray] (0,-2) rectangle (1,-3);
 \node[scale=.7] at (0.6, -2.5) {$\ell_1$};
 \fill[lightgray] (2,-1) rectangle (3,-2);
 \node[scale=.7] at (2.6, -1.5) {$\ell_9$};
 \rpp{3,1}{0,0}{0,1}
 \end{scope}
\end{tikzpicture}
%\draw[->] (15.5,-8cm) -- (15.5,-11cm) node[midway,right] {$\delta_6^{\lusztig}$};
\\ \delta_6^{\lusztig}(\nu, J) & = \begin{tikzpicture}[scale=.33,baseline=-18]
 \rpp{2}{1}{1}
 \begin{scope}[xshift=5cm]
 \rpp{2,2,1}{0,1,1}{0,1,1}
 \end{scope}
 \begin{scope}[xshift=11cm]
 \rpp{2,2,1}{0,0,0}{0,0,0}
 \end{scope}
 \begin{scope}[xshift=16cm]
 \rpp{2,2,2,1,1}{0,0,0,0,0}{0,0,0,0,0}
 \end{scope}
 \begin{scope}[xshift=22cm]
 \rpp{2,2,1}{0,0,0}{0,0,0}
 \end{scope}
 \begin{scope}[xshift=28cm]
 \rpp{2}{0}{1}
 \end{scope}
\end{tikzpicture}
\end{align*}
(where the return value of $\delta_6^{\lusztig}$ is $\bon3\bsix$). Now, we compute $e_6 e_5 e_4 e_2 e_1 e_3 e_4 e_5 e_6(\nu, \overline{J})$:
\begin{align*}
(\nu, \overline{J}) & = \begin{tikzpicture}[scale=.33,baseline=-18]
 \rpp{2,1}{0,0}{0,0}
 \begin{scope}[xshift=5cm]
 \rpp{3,2,1}{0,1,1}{0,1,1}
 \end{scope}
 \begin{scope}[xshift=11cm]
 \rpp{2,2,1,1}{0,0,0,0}{0,0,0,0}
 \end{scope}
 \begin{scope}[xshift=16cm]
 \rpp{3,2,2,1,1,1}{0,0,0,0,0,0}{0,0,0,0,0,0}
 \end{scope}
 \begin{scope}[xshift=22cm]
 \rpp{3,2,1,1}{0,0,0,0}{0,0,0,0}
 \end{scope}
 \begin{scope}[xshift=28cm]
 \rpp{3,1}{-1,-1}{-1,0}
 \end{scope}
\end{tikzpicture}
\allowdisplaybreaks \\ \xleftarrow[\hspace{15pt}]{6} & \begin{tikzpicture}[scale=.33,baseline=-18]
 \rpp{2,1}{0,0}{0,0}
 \begin{scope}[xshift=5cm]
 \rpp{3,2,1}{0,1,1}{0,1,1}
 \end{scope}
 \begin{scope}[xshift=11cm]
 \rpp{2,2,1,1}{0,0,0,0}{0,0,0,0}
 \end{scope}
 \begin{scope}[xshift=16cm]
 \rpp{3,2,2,1,1,1}{0,0,0,0,0,0}{0,0,0,0,0,0}
 \end{scope}
 \begin{scope}[xshift=22cm]
 \rpp{3,2,1,1}{-1,-1,-1,-1}{-1,-1,-1,-1}
 \end{scope}
 \begin{scope}[xshift=28cm]
 \rpp{3}{1}{1}
 \end{scope}
\end{tikzpicture}
\allowdisplaybreaks \\ \xleftarrow[\hspace{15pt}]{5} & \begin{tikzpicture}[scale=.33,baseline=-18]
 \rpp{2,1}{0,0}{0,0}
 \begin{scope}[xshift=5cm]
 \rpp{3,2,1}{0,1,1}{0,1,1}
 \end{scope}
 \begin{scope}[xshift=11cm]
 \rpp{2,2,1,1}{0,0,0,0}{0,0,0,0}
 \end{scope}
 \begin{scope}[xshift=16cm]
 \rpp{3,2,2,1,1,1}{-1,-1,-1,-1,-1,-1}{-1,-1,-1,-1,-1,-1}
 \end{scope}
 \begin{scope}[xshift=22cm]
 \rpp{3,2,1}{1,1,1}{1,1,1}
 \end{scope}
 \begin{scope}[xshift=28cm]
 \rpp{3}{0}{0}
 \end{scope}
\end{tikzpicture}
\allowdisplaybreaks \\ \xleftarrow[\hspace{15pt}]{4} & \begin{tikzpicture}[scale=.33,baseline=-18]
 \rpp{2,1}{0,0}{0,0}
 \begin{scope}[xshift=5cm]
 \rpp{3,2,1}{{\!\!-1},0,0}{-1,0,0}
 \end{scope}
 \begin{scope}[xshift=11cm]
 \rpp{2,2,1,1}{-1,-1,-1,-1}{-1,-1,-1,-1}
 \end{scope}
 \begin{scope}[xshift=16cm]
 \rpp{3,2,2,1,1}{1,1,1,1,1}{1,1,1,1,1}
 \end{scope}
 \begin{scope}[xshift=22cm]
 \rpp{3,2,1}{0,0,0}{0,0,0}
 \end{scope}
 \begin{scope}[xshift=28cm]
 \rpp{3}{0}{0}
 \end{scope}
\end{tikzpicture}
\allowdisplaybreaks \\ \xleftarrow[\hspace{15pt}]{3} & \begin{tikzpicture}[scale=.33,baseline=-18]
 \rpp{2,1}{{\!\!-1},-1}{-1,-1}
 \begin{scope}[xshift=5cm]
 \rpp{3,2,1}{-1,0,0}{-1,0,0}
 \end{scope}
 \begin{scope}[xshift=11cm]
 \rpp{2,2,1}{1,1,1}{1,1,1}
 \end{scope}
 \begin{scope}[xshift=16cm]
 \rpp{3,2,2,1,1}{0,0,0,0,0}{0,0,0,0,0}
 \end{scope}
 \begin{scope}[xshift=22cm]
 \rpp{3,2,1}{0,0,0}{0,0,0}
 \end{scope}
 \begin{scope}[xshift=28cm]
 \rpp{3}{0}{0}
 \end{scope}
\end{tikzpicture}
\allowdisplaybreaks \\ \xleftarrow[\hspace{15pt}]{1} & \begin{tikzpicture}[scale=.33,baseline=-18]
 \rpp{2}{1}{1}
 \begin{scope}[xshift=5cm]
 \rpp{3,2,1}{-1,0,0}{-1,0,0}
 \end{scope}
 \begin{scope}[xshift=11cm]
 \rpp{2,2,1}{0,0,0}{0,0,0}
 \end{scope}
 \begin{scope}[xshift=16cm]
 \rpp{3,2,2,1,1}{0,0,0,0,0}{0,0,0,0,0}
 \end{scope}
 \begin{scope}[xshift=22cm]
 \rpp{3,2,1}{0,0,0}{0,0,0}
 \end{scope}
 \begin{scope}[xshift=28cm]
 \rpp{3}{0}{0}
 \end{scope}
\end{tikzpicture}
\allowdisplaybreaks \\ \xleftarrow[\hspace{15pt}]{2} & \begin{tikzpicture}[scale=.33,baseline=-18]
 \rpp{2}{1}{1}
 \begin{scope}[xshift=5cm]
 \rpp{2,2,1}{0,0,0}{0,0,0}
 \end{scope}
 \begin{scope}[xshift=11cm]
 \rpp{2,2,1}{0,0,0}{0,0,0}
 \end{scope}
 \begin{scope}[xshift=16cm]
 \rpp{3,2,2,1,1}{-1,0,0,0,0}{-1,0,0,0,0}
 \end{scope}
 \begin{scope}[xshift=22cm]
 \rpp{3,2,1}{0,0,0}{0,0,0}
 \end{scope}
 \begin{scope}[xshift=28cm]
 \rpp{3}{0}{0}
 \end{scope}
\end{tikzpicture}
\allowdisplaybreaks \\ \xleftarrow[\hspace{15pt}]{4} & \begin{tikzpicture}[scale=.33,baseline=-18]
 \rpp{2}{1}{1}
 \begin{scope}[xshift=5cm]
 \rpp{2,2,1}{0,0,0}{0,0,0}
 \end{scope}
 \begin{scope}[xshift=11cm]
 \rpp{2,2,1}{0,0,0}{0,0,0}
 \end{scope}
 \begin{scope}[xshift=16cm]
 \rpp{2,2,2,1,1}{0,0,0,0,0}{0,0,0,0,0}
 \end{scope}
 \begin{scope}[xshift=22cm]
 \rpp{3,2,1}{-1,0,0}{-1,0,0}
 \end{scope}
 \begin{scope}[xshift=28cm]
 \rpp{3}{0}{0}
 \end{scope}
\end{tikzpicture}
\allowdisplaybreaks \\ \xleftarrow[\hspace{15pt}]{5} & \begin{tikzpicture}[scale=.33,baseline=-18]
 \rpp{2}{1}{1}
 \begin{scope}[xshift=5cm]
 \rpp{2,2,1}{0,0,0}{0,0,0}
 \end{scope}
 \begin{scope}[xshift=11cm]
 \rpp{2,2,1}{0,0,0}{0,0,0}
 \end{scope}
 \begin{scope}[xshift=16cm]
 \rpp{2,2,2,1,1}{0,0,0,0,0}{0,0,0,0,0}
 \end{scope}
 \begin{scope}[xshift=22cm]
 \rpp{2,2,1}{0,0,0}{0,0,0}
 \end{scope}
 \begin{scope}[xshift=28cm]
 \rpp{3}{-1}{-1}
 \end{scope}
\end{tikzpicture}
\allowdisplaybreaks \\ \xleftarrow[\hspace{15pt}]{6} & \begin{tikzpicture}[scale=.33,baseline=-18]
 \rpp{2}{1}{1}
 \begin{scope}[xshift=5cm]
 \rpp{2,2,1}{0,0,0}{0,0,0}
 \end{scope}
 \begin{scope}[xshift=11cm]
 \rpp{2,2,1}{0,0,0}{0,0,0}
 \end{scope}
 \begin{scope}[xshift=16cm]
 \rpp{2,2,2,1,1}{0,0,0,0,0}{0,0,0,0,0}
 \end{scope}
 \begin{scope}[xshift=22cm]
 \rpp{2,2,1}{0,0,0}{0,0,0}
 \end{scope}
 \begin{scope}[xshift=28cm]
 \rpp{2}{0}{1}
 \end{scope}
\end{tikzpicture}
\end{align*}
and note that the result equals $\delta_6^{\lusztig}(\nu, J)$.
\end{example}

\begin{proposition}
\label{prop:rb_intertwines}
Let $r$ be a minuscule node.
Let $\delta_r^{\lusztig} := \eta \circ \delta_r \circ \eta$ on rigged configurations and $\delta_r^{\lusztig} := \diamond \circ \delta_r \circ \diamond$ on classically highest weight elements in tensor product of KR crystals and extended as a classical crystal isomorphism.
We have
\[
\delta^{\lusztig} \circ \Phi = \Phi \circ \delta^{\lusztig}.
\]
\end{proposition}

\begin{proof}
Recall $\delta_r$ commutes with the classical crystal operators by Theorem~\ref{thm:minuscule_classical_isomorphism}. Note that we can define $\delta_r^{\lusztig}$ on classically highest weight elements in a tensor product of KR crystals by removing the rightmost factor and then going to the highest weight element from~\cite[Prop.~5.9(5)]{SS2006}.\footnote{Recall that $\delta_r^{\lusztig}$ in~\cite{SS2006}, where it was denoted by $\operatorname{rh}$, was defined by removing the rightmost factor (and then going to the highest weight element) in contrast to our definition of conjugating the left factor removal by $\lusztig$. However, these definitions are equivalent by~\cite[Prop.~5.9(5)]{SS2006}.}
Lemma~\ref{lemma:dual_removal} shows the same description for the rigged configurations.

To prove the claim, it is sufficient to prove this on classically highest weight elements.
Since the rightmost factor is $B^{r,1}$, where $r$ is a minuscule node, the corresponding rightmost element in the tensor product must be $u_{\clfw_r}$. Recall that $\Phi^{-1}(u_{\clfw_r})$ is the empty rigged configuration (but with vacancy numbers $p_i^{(r)}$ shifted by $1$ for all $i \in \ZZ_{>0}$). Note also that $\Phi^{-1}$ is only computed using singular rows, not the actual rigging values. Therefore, if $\Phi^{-1}(b \otimes u_{\clfw_r}) = (\nu, J)$, then $\Phi^{-1}(b) = (\nu, \overline{J})$ as defined in Lemma~\ref{lemma:dual_removal}. If $e_{a_1} \dotsm e_{a_k} b$, where $a_1, \dotsc, a_k \in I_0$, is the corresponding classically highest weight element, then we have
\[
\Phi^{-1}\bigl( \delta^*(b \otimes u_{s\clfw_r}) \bigr) = \Phi^{-1}(e_{a_1} \dotsm e_{a_k} b) = e_{a_1} \dotsm e_{a_k} \Phi^{-1}(b) = \delta^*\bigl( \Phi^{-1}(b \otimes u_{s\clfw_r}) \bigr)
\]
as desired.
\end{proof}

We also have the following from Proposition~\ref{prop:rb_intertwines} and that clearly $[\delta, \delta^{\lusztig}] = 0$ on a tensor product of KR crystals.

\begin{corollary}
\label{cor:delta_star_commute}
On rigged configurations, we have $[\delta, \delta^{\lusztig}] = 0$.
\end{corollary}

We remark that this provides an alternative proof of~\cite[Lemma~3.13]{KSS02} and~\cite[Thm.~8.4(1)]{SS2006}.

Using Corollary~\ref{cor:delta_star_commute}, the remaining proof of~\cite[Thm.~8.6]{SS2006} that $\diamond \circ \Phi = \Phi \circ \theta$ holds. Hence, the proof of~\cite[Thm.~8.8]{SS2006} by using Theorem~\ref{thm:minuscule_classical_isomorphism}. Thus, we have the following.

\begin{theorem}
\label{thm:general_minuscule_bijection}
Let $B = \bigotimes_{i=1}^N B^{r_i,s_i}$ be a tensor product of KR crystals, where $r_i$ is a minuscule node for all $i$. Then the map
\[
\Phi \colon \RC(B) \to B
\]
is a $U_q(\g_0)$-crystal isomorphism such that $\Phi \circ \eta$ sends cocharge to energy.
\end{theorem}

%=====================================================================
\section{Equivalent bijections}
\label{sec:equivalence}

In this section, we show that all of our defined bijections give the same bijection and $\Phi = \KSS{\Phi}$ for all types except $B_n^{(1)}$ (where it remains a conjecture), $F_4^{(1)}$, and $G_2^{(1)}$. We show that $\Phi$ defined using a combination of various $\delta_{\sigma}$ defines the same bijection as $\KSS{\Phi}$ for types $A_n^{(1)}$, $D_n^{(1)}$, and $E_6^{(1)}$. Note that for type $A_n^{(1)}$, the map $\delta_{n-1}$ is the dual bijection of~\cite[Sec.~9]{SS2006}. For types $A_n^{(1)}$ and $D_n^{(1)}$, we will show that the bijection defined by $\delta_{\sigma}$ and $\KSS{\delta}$ are equal. For type $A_{2n-1}^{(2)}$, we only have $\delta_1 = \KSS{\delta}$, which was done in~\cite{OSS03}. For type $D_{n+1}^{(2)}$, we will use the virtualization map and type $A_{2n-1}^{(1)}$. For type $E_6^{(1)}$, the map $\delta_1 = \KSS{\delta}$ from~\cite{OS12} and $\delta_6$ follows from the diagram automorphism given by $1 \leftrightarrow 6$.

We will first show that all of the bijections defined using $\delta_r$ for any $r \in I_0$ in type $A_n^{(1)}$ are equal.

\begin{lemma}
\label{lemma:type_A_equality}
Let $\g$ be of type $A_n^{(1)}$ and $B = B^{r,1} \otimes B^{\bullet}$. We have
\[
\delta_r = \KSS{\delta} \circ (\KSS{\delta} \circ \lb)^{r-1}
\]
such that $\lb^{r-1}(b_r) = b_1^r \otimes \cdots \otimes b_1^1$ (with $\lb$ always acting on the rightmost factor), were $b_1^i$ is the return value for the $i$-th application of $\KSS{\delta}$.
\end{lemma}

\begin{proof}
We note that $\delta_1 = \KSS{\delta}$. Thus fix an $r > 1$. We proceed by induction on $r$.
Note that it is sufficient to show by our induction assumption that
\begin{equation}
\label{eq:type_A_reduction}
\delta_r = \delta_{r-1} \circ \KSS{\delta} \circ \lb
\end{equation}
such that $\lb(b_r) = \young(b) \otimes b_{r-1}$, where $b_r$ denotes the return value of $\delta_r$ and $\young(b) \otimes b_{r-1}$ the return values under $\delta_{r-1} \circ \KSS{\delta} \circ \lb$.
Recall that $\lb(u_{\clfw_r}) = \young(r) \otimes u_{\clfw_{r-1}}$, which defines the unique strict crystal embedding $B(\clfw_r) \to B(\clfw_1) \otimes B(\clfw_{r-1})$.

Suppose $\delta_r$ selects (singular) rows $\ell^{(a)}_1 \leq \cdots \leq \ell^{(a)}_{k_a}$ from $\nu^{(a)}$ for all $a \in I_0$; we consider $k_a = 0$ if no such row was selected in $\nu^{(a)}$.
Note that $\KSS{\delta} \circ \lb$ selects the same singular row of length $\ell^{(a)}_1$ in $\nu^{(a)}$ under $\delta_r$ as it is the smallest selectable singular row in $\nu^{(a)}$ for all $r \leq a < b$, where $\KSS{\delta}$ returns $\young(b)$, and both algorithms (effectively\footnote{From the definition of $\lb$ and that $\lb$ preserves vacancy numbers, we can consider $\KSS{\delta} \circ \lb$ as one operation that is the same as $\KSS{\delta}$ except that it begins at $\nu^{(r)}$ instead of $\nu^{(1)}$.}) start at $\nu^{(r)}$.
Indeed, by the definition of $\delta_r$, we must have $\ell_1^{(a)} \leq \ell_1^{(a+1)}$ and no singular rows in $\nu^{(a+1)}$ of length $\ell_1^{(a)} \leq i < \ell_1^{(a+1)}$ for all $r \leq a < b-1$ as $\delta_r$ selects the minimal singular row and would have instead selected any singular row of length $\ell_1^{(a)} \leq i < \ell_1^{(a+1)}$.

Next, let $(\overline{\nu}, \overline{J}) = (\KSS{\delta} \circ \KSS{\lb})(\nu, J)$. We claim that $\delta_{r-1}$ on $(\overline{\nu}, \overline{J})$ must select exactly all other rows selected by $\delta_r$. If $k_{r-1} = 0$, then $\delta_r$ has selected no other singular rows (note that in this case the crystal structure of $B(\clfw_r) \subseteq B(\clfw_1) \otimes B(\clfw_{r-1})$ implies $k_a = 1$ for all $r \leq a < b$ and $k_a = 0$ otherwise). Now we have
\begin{equation}
\label{eq:vac_change_A_bijections_equal}
p_i^{(r-1)}(\overline{\nu},\overline{J}) =
\begin{cases}
p_i^{(r-1)}(\nu,J) + 1 &\text{if } i < \ell_1^{(r)}, \\
p_i^{(r-1)}(\nu,J) & \text{otherwise},
\end{cases}
\end{equation}
so none of these rows of length $i < \ell_1^{(r)}$ are singular and $k_{r-1} = 0$ with Equation~\eqref{eq:vac_change_A_bijections_equal} implies the other rows are not singular, thus $\delta_{r-1}$ returns $u_{\clfw_{r-1}}$, and the claim holds.

Now suppose $k_{r-1} > 0$. The algorithm for $\delta_{r-1}$ begins by selecting the row of length $\ell^{(r-1)}_1$ because of Equation~\eqref{eq:vac_change_A_bijections_equal} and $\delta_r$ would have selected any singular rows of length $\ell_1^{(r)} \leq i < \ell_1^{(r-1)}$.
Indeed, if there was such a singular row $R$ and $\delta_r$ selects $\ell_1^{(r-1)}$ at step $t$, then $R$ would have been selected during the an earlier step of $\delta_r$ by the fact we select a singular row of minimal length and we could always select a string from $\nu^{(r-1)}$ after the first step (and up to step $t$).
Next, we note that there are no singular rows in $\overline{\nu}^{(r)}$ of length $\ell^{(r)}_1 \leq i < \ell^{(r+1)}_1$ as $p_i^{(r)}$ has been increased by $1$ in that region. Therefore, the next selected row in $\overline{\nu}^{(r)}$ is of length $\ell_2^{(r)}$ since
\begin{itemize}
\item all vacancy numbers $p_i^{(r)}$ for $i \geq \ell^{(r+1)}_1$ remain unchanged,
\item $\ell_1^{(r-1)} \geq \ell_1^{(r)}$,
\item all other rows of length at least $\ell_1^{(r+1)}$ are singular in $(\nu, J)^{(r)}$ if and only if they are singular in $(\overline{\nu}, \overline{J})^{(r)}$,
\item and we can select a second row in $\nu^{(r)}$ as soon as we select a row in $\nu^{(r-1)}$ from the definition of $\delta_r$.
\end{itemize}
Similarly for all $(\overline{\nu}, \overline{J})^{(a)}$ with $r < a < b-1$, but we note there are no singular rows in $\overline{\nu}^{(b-1)}$ of length at least $\ell_1^{(b-1)}$. Note that by the definition of $\delta_r$ and the crystal $B(\clfw_r)$, we must have $\ell_2^{(b-2)} \geq \ell_1^{(b-1)}$ (as otherwise we would have $\begin{array}{|c|}\hline b-1 \\\hline b-1 \\\hline \end{array}$ as a subtableau for that element since the definition of $\delta_r$ would imply the $(b-2)$-arrow $\begin{array}{|c|}\hline b-1 \\\hline b-2 \\\hline \end{array} \xrightarrow[\hspace{20pt}]{b-2} \begin{array}{|c|}\hline b-1 \\\hline b-1 \\\hline \end{array}$ in the crystal graph exists, which is impossible (note all other entries in the tableau do not matter by the column strictness, so we do not write them)). Thus, we have $\young(b) \otimes b_{r-1} \in B(\clfw_r) \subseteq B(\clfw_1) \otimes B(\clfw_{r-1})$ (equivalently $\lb(b_r) = \young(b) \otimes b_{r-1}$) as $\delta_{r-1}$ must select the row of length $\ell_2^{(b-2)}$ in $\overline{\nu}^{(b-2)}$ before selecting a row in $\overline{\nu}^{(b-1)}$ and hence cannot select any more rows in $\overline{\nu}^{(b-1)}$. Additionally, we have $(\nu,J)^{(a)} = (\overline{\nu},\overline{J})^{(a)}$ for all $a < r$ with the same vacancy numbers for all $a < r - 1$ and
\[
p_i^{(a)}(\overline{\nu},\overline{J}) =
\begin{cases}
p_i^{(a)}(\nu,J) + 1 &\text{if } \ell_1^{(a)} \leq i < \ell_1^{(a+1)}, \\
p_i^{(a)}(\nu,J) & \text{if } i \geq \ell_1^{(a+1)},
\end{cases}
\]
for all $r \leq a < b - 1$.
Thus, once $\delta_{r-1}$ starts selecting $\ell_2^{(a)}$, it must select the same rows of length $\ell_k^{(a)}$ for $k \geq 2$ as $\delta_r$. Hence the rest of the algorithm for $\delta_{r-1}$ selects all the same other rows selected by $\delta_r$ following the same path as $\delta_r$ except for those boxes selected by $\KSS{\delta} \circ \lb$. Therefore, the resulting rigged configurations are equal, and hence the claim follows.
\end{proof}

We note that the proof of Lemma~\ref{lemma:type_A_equality} is given by $\delta_{r-1} \circ \KSS{\delta} \circ \lb$ following the path
\[
u_{\clfw_r} = \underbrace{
\young(u,r)
\xrightarrow[\hspace{20pt}]{r}
\cdots
\xrightarrow[\hspace{20pt}]{b-1}
\young(u,b)
}_{\KSS{\delta} \circ \lb}
\xrightarrow[\hspace{30pt}]{\delta_{r-1}}
\young(x,b) = b_r
\]
in $B(\clfw_r)$, where $u = u_{\clfw_{r-1}}$ and $x = b_{r-1}$.

\begin{example}
Consider
\[
B = B^{3,1} \otimes B^{2,2} \otimes B^{4,2} \otimes B^{1,5} \otimes B^{4,3} \otimes B^{3,2} \otimes B^{4,1} \otimes B^{2,1} \otimes B^{4,2}\
\]
of type $A_5^{(1)}$. Then we have
\[
\begin{tikzpicture}[scale=.33,baseline=-18]
 \fill[lightgray] (0,-2) rectangle (1,-3);
 \node[scale=.7] at (0.6, -2.5) {$\ell_3$};
 \rpp{4,1}{0,1}{2,1}
 \begin{scope}[xshift=8cm]
 \fill[lightgray] (0,-4) rectangle (1,-5);
 \node[scale=.7] at (0.6, -4.5) {$\ell_2$};
 \fill[lightgray] (3,-1) rectangle (4,-2);
 \node[scale=.7] at (3.6, -1.5) {$\ell_6$};
 \rpp{4,2,1,1}{0,0,0,0}{0,0,0,0}
 \end{scope}
 \begin{scope}[xshift=16cm]
 \fill[lightgray] (0,-4) rectangle (1,-5);
 \node[scale=.7] at (0.6, -4.5) {$\ell_1$};
 \fill[lightgray] (3,-1) rectangle (4,-2);
 \node[scale=.7] at (3.6, -1.5) {$\ell_5$};
 \rpp{4,2,1,1}{4,1,2,2}{4,4,2,2}
 \end{scope}
 \begin{scope}[xshift=24cm]
 \fill[lightgray] (1,-4) rectangle (2,-3);
 \node[scale=.7] at (1.6, -3.5) {$\ell_4$};
 \fill[lightgray] (3,-1) rectangle (4,-2);
 \node[scale=.7] at (3.6, -1.5) {$\ell_8$};
 \rpp{4,2,2,1}{1,1,1,0}{1,1,1,1}
 \end{scope}
 \begin{scope}[xshift=32cm]
 \fill[lightgray] (2,-1) rectangle (3,-2);
 \node[scale=.7] at (2.6, -1.5) {$\ell_7$};
 \rpp{3}{2}{2}
 \end{scope}
%%%
\draw[->] (15.5,-6cm) -- (15.5,-9cm) node[midway,right] {$\delta_3$};
\draw (21,-7.5cm) node {(returns $\young(3,5,6)$\,)};
\begin{scope}[yshift=-9cm]
 \rpp{4}{0}{2}
 \begin{scope}[xshift=8cm]
 \rpp{3,2,1}{0,0,0}{0,0,0}
 \end{scope}
 \begin{scope}[xshift=16cm]
 \rpp{3,2,1}{3,1,2}{3,3,2}
 \end{scope}
 \begin{scope}[xshift=24cm]
 \rpp{3,2,1,1}{2,1,0,0}{2,2,0,0}
 \end{scope}
 \begin{scope}[xshift=32cm]
 \rpp{2}{2}{2}
 \end{scope}
\end{scope} % y-shift
\end{tikzpicture}
\]
Furthermore, we also compute 
\[
\begin{tikzpicture}[scale=.33,baseline=-18]
 \rpp{4,1}{0,1}{2,1}
 \begin{scope}[xshift=8cm]
 \rpp{4,2,1,1}{0,0,0,0}{0,0,0,0}
 \end{scope}
 \begin{scope}[xshift=16cm]
 \fill[lightgray] (0,-4) rectangle (1,-5);
 \node[scale=.7] at (0.6, -4.5) {$\ell_1$};
 \rpp{4,2,1,1}{4,1,2,2}{4,4,2,2}
 \end{scope}
 \begin{scope}[xshift=24cm]
 \fill[lightgray] (1,-3) rectangle (2,-4);
 \node[scale=.7] at (1.6, -3.5) {$\ell_2$};
 \rpp{4,2,2,1}{1,1,1,0}{1,1,1,1}
 \end{scope}
 \begin{scope}[xshift=32cm]
 \fill[lightgray] (2,-1) rectangle (3,-2);
 \node[scale=.7] at (2.6, -1.5) {$\ell_3$};
 \rpp{3}{2}{2}
 \end{scope}
%%%
\draw[->] (15.5,-6cm) -- (15.5,-9cm) node[midway,right] {$\KSS{\delta} \circ \lb$};
\draw (23,-7.5cm) node {(returns $\young(6)$\,)};
\begin{scope}[yshift=-9cm]
 \fill[lightgray] (0,-2) rectangle (1,-3);
 \node[scale=.7] at (0.6, -2.5) {$\ell_2$};
 \rpp{4,1}{0,1}{2,1}
 \begin{scope}[xshift=8cm]
 \fill[lightgray] (0,-4) rectangle (1,-5);
 \node[scale=.7] at (0.6, -4.5) {$\ell_1$};
 \fill[lightgray] (3,-1) rectangle (4,-2);
 \node[scale=.7] at (3.6, -1.5) {$\ell_4$};
 \rpp{4,2,1,1}{0,0,0,0}{0,0,0,0}
 \end{scope}
 \begin{scope}[xshift=16cm]
 \fill[lightgray] (3,-1) rectangle (4,-2);
 \node[scale=.7] at (3.6, -1.5) {$\ell_3$};
 \rpp{4,2,1}{4,1,2}{4,4,3}
 \end{scope}
 \begin{scope}[xshift=24cm]
 \fill[lightgray] (3,-1) rectangle (4,-2);
 \node[scale=.7] at (3.6, -1.5) {$\ell_5$};
 \rpp{4,2,1,1}{2,1,0,0}{2,2,0,0}
 \end{scope}
 \begin{scope}[xshift=32cm]
 \rpp{2}{2}{2}
 \end{scope}
\end{scope} % y-shift
%%%
\draw[->] (15.5,-15cm) -- (15.5,-18cm) node[midway,right] {$\delta_2$};
\draw (21,-16.5cm) node {(returns $\young(3,5)$\,)};
\begin{scope}[yshift=-18cm]
 \rpp{4}{0}{2}
 \begin{scope}[xshift=8cm]
 \rpp{3,2,1}{0,0,0}{0,0,0}
 \end{scope}
 \begin{scope}[xshift=16cm]
 \rpp{3,2,1}{3,1,2}{3,3,2}
 \end{scope}
 \begin{scope}[xshift=24cm]
 \rpp{3,2,1,1}{2,1,0,0}{2,2,0,0}
 \end{scope}
 \begin{scope}[xshift=32cm]
 \rpp{2}{2}{2}
 \end{scope}
\end{scope} % y-shift
\end{tikzpicture}
\]
which agrees with applying $\delta_3$ as per Lemma~\ref{lemma:type_A_equality}.
\end{example}

In order to show that the bijections agree for type $D_n^{(1)}$, we recall the map $\delta_{\mathrm{sp}} \colon B^{r,1} \otimes B^{\bullet} \to B^{\bullet}$ for $r = n-1, n$ from~\cite{S05}. Let $v \colon B^{r,1} \to B^{r,2}$ denote the virtualization map given by $\gamma_a = 2$ for all $a \in I$. Let $\lb^{(r)} \colon B^{r,2} \otimes B^{\bullet} \to B^{1,1} \otimes B^{n-1,1} \otimes B^{n,1} \otimes B^{\bullet}$ be given on rigged configurations by adding a length-one singular row to $\nu^{(a)}$ for all $a \leq n - 2$ and $a = n,n-1$ if $r = n-1, n$ respectively. Let $\overline{\lb} \colon B^{n-1, 1} \otimes B^{n,1} \otimes B^{\bullet} \mapsto B^{n-2,1} \otimes B^{\bullet}$ be given on rigged configurations by adding a length-one singular row to $\nu^{(a)}$ for all $a \leq n-2$. We define $\delta_{\mathrm{sp}} \colon B^{r,1} \otimes B^{\bullet} \to B^{\bullet}$ by
\[
\delta_{\mathrm{sp}} := v^{-1} \circ (\KSS{\delta} \circ \lb)^{n-2} \circ \KSS{\delta} \circ \overline{\lb} \circ \KSS{\delta} \circ \lb^{(r)} \circ v.
\]

We could also define $\delta_{\mathrm{sp}}$ by using the (generalized) $\lb$-diagram of
\[
\begin{tikzpicture}[baseline=-4, xscale=1.4]
\node (1) at (0,0) {$1$};
\node (2) at (2,0) {$\cdots$};
\node (3) at (4,0) {$n-2$};
\node (4) at (6,0) {$(n-1,n)$};
\node (5) at (8,0) {$[r]$};
\draw[->] (2) -- (1) node[midway, above]{\small $\boxed{1}$};
\draw[->] (3) -- (2) node[midway, above]{\small $\boxed{n-2}$};
\draw[->] (4) -- (3) node[midway, above]{\small $\boxed{n-1}$};
\draw[->] (5) -- (4) node[midway, above]{\small $\boxed{x}$};
\end{tikzpicture},
\]
where $x = \bn, n$ if $r = n-1, n$ respectively and recall $[r]$ corresponds to $B^{r,2}$.

\begin{lemma}
\label{lemma:type_D_equality}
Let $\g$ be of type $D_n^{(1)}$ and $B = B^{r,1} \otimes B^{\bullet}$ with $r = n-1, n$. We have
\[
\delta_r = \delta_{\mathrm{sp}}.
\]
\end{lemma}

\begin{proof}
We will use the spin representation for elements of $B(\clfw_r) \iso B^{r,1}$ (as $U_q(\g_0)$-crystals) of type $D_n^{(1)}$. That is, an element $b \in B(\clfw_r)$ with $\wt(b) = \frac{1}{2} \sum_{i=1}^n s_i \epsilon_i$ (given as an element in the ambient $\frac{1}{2}\ZZ^n$ lattice) we represent as the $\pm$-sequence $(s_1, \dotsc, s_n)$ (see also, \textit{e.g.},~\cite[Sec.~6.4]{KN94}). Under the doubling map $v$, the corresponding tableau has an $i$ (resp.~$\overline{\imath}$) if and only if $s_i = +$ (resp.~$s_i = -$), written in increasing order down the column. Note that $u_{\clfw_r} = (+, +, \dotsc, +, \pm)$, where we have $s_n = -,+$ if $r = n-1, n$ respectively.

The proof is essentially the same as the proof of Lemma~\ref{lemma:type_A_equality}. Indeed, it is sufficient to show that the resulting rigged configurations are equal after applying $\delta_r$ and $\delta_{\mathrm{sp}}$ to a fixed rigged configuration $(\nu, J)$ as $B(\clfw_r)$ is a minuscule representation.
Let $b$ be the return value of $\delta_r$, which selects rows of length $\ell^{(a)}_1 \leq \cdots \leq \ell^{(a)}_{k_a}$ in $\nu^{(a)}$ for all $a \in I_0$, where we consider $k_a = 0$ if no row was selected in $\nu^{(a)}$.

Let $j_1 < \cdots < j_m$ be all indices such that $s_{j_h} = -$.
Let $(\tnu, \tJ) = v(\nu, J)$.
Let $\check{r} = n,n-1$ if $r = n-1,n$ respectively.
For $h = 1$, we must have applied $f_r f_{n-2} \dotsm f_{j_1}$ to $u_{\clfw_r}$ in order to obtain $b$.
Thus, the application of $\KSS{\delta} \circ \lb^{(r)}$ results in selecting the row of length $2\ell^{(a)}_1$ in $\tnu^{(a)}$ for all $a = r, n-2, \dotsc, j_1$ as a row is singular in $(\nu, J)$ if and only if it is singular in $(\tnu, \tJ)$.
Let $(\overline{\nu}, \overline{J})$ be the resulting rigged configuration.
For $h = 2$, we must have also applied $f_{\check{r}} f_{n-2} \dotsm f_{j_2}$.
When applying $\KSS{\delta} \circ \overline{\lb}$, we note that we have increased $p_i^{(\check{r})}$ for all $i < 2\ell_1^{(n-2)}$, and so we must select a row of length $2\ell_1^{(\check{r})}$ in $\overline{\nu}^{(\check{r})}$.
Similar to the proof of Lemma~\ref{lemma:type_A_equality}, we select rows of length $2\ell^{(a)}_2$ in $\overline{\nu}^{(a)}$ for all $a = n-2, \dotsc, j_2$.
The case for all other $h$ and applying $\KSS{\delta} \circ \lb$ is similar except the initial steps until we reach $n-1, n$ remove a box from an odd length row.
Similarly, all of the positive values in the column $v(b)$ remove a box from an odd length row.
Therefore, the resulting rigged configurations are equal and the claim follows.
\end{proof}

Next, we consider the case of $D_{n+1}^{(2)}$ and recall that $\virtual{B}^{n,1} = B^{n,1}$ of type $A_{2n-1}^{(1)}$. Let $\virtual{\delta}_n$ denote the map $\delta_n$ of type $A_{2n-1}^{(1)}$. Thus, we have the following from the definition of $\delta_r$.% and Lemma~\ref{lemma:type_A_equality}.

\begin{proposition}
\label{prop:D_twisted_minuscule_virtualization}
Let $\g$ be of type $D_{n+1}^{(2)}$. Then
\[
v \circ \delta_n = \virtual{\delta}_n \circ v.
\]
%Moreover, we have $\delta_n = \KSS{\delta}$.
\end{proposition}

%For $B(\clfw_5)$ , this is generated by one of the following in $B(\clfw_1)^{\otimes 4}$:
%\begin{align*}
%\bon \overline{6} 5 \otimes \bon 6 \otimes 1 \otimes 1,
%\\ \overline{4} 5 \otimes \bth 4 \otimes \bon 3 \otimes 1,
%\\ \bon \overline{6} 5 \otimes \bth 1 6 \otimes \bon 3 \otimes 1,
%\\ \btw 5 \otimes \bth 2 \otimes \bon 3 \otimes 1,
%\\ \bon \overline{6} 5 \otimes 1 \otimes \bon 6 \otimes 1,
%\\ \btw 5 \otimes \overline{6} 2 \otimes \bon 6 \otimes 1.
%\end{align*}

\begin{lemma}
\label{lemma:dual_equivalence}
Let $\g$ be of type $E_6^{(1)}$ and $B = B^{6,1} \otimes B^{\bullet}$. Then
\begin{align*}
\delta_6 & = \delta_1 \circ \delta_1 \circ \lb,
\\ \delta_1 & = \delta_6 \circ \delta_6 \circ \lb^{\vee}.
\end{align*}
\end{lemma}

\begin{proof}
Similar to the proof of Lemma~\ref{lemma:type_A_equality}.
\end{proof}

%\begin{proof}
%Fix a return value for $\delta_6(\nu, J) = b^{\vee}$, and suppose $b^{\vee} \mapsto b \otimes b'$ under the canonical classical strict embedding $B^{6,1} \lhook\joinrel\rightarrow B^{1,1} \otimes B^{1,1}$. By direct inspection of $B^{1,1} \otimes B^{1,1}$, we see there exists a path $(a_1, a_2, \dotsc, a_i, a'_1, a'_2, \dotsc, a'_{i'})$ such that
%\[
%f_{a'_{i'}} \dotsm f_{a'_1} f_{a_i} \dotsm f_{a_1}(b \otimes b') = (f_{a_i} \dotsm f_{a_1} b) \otimes (f_{a'_{i'}} \dotsm f_{a'_1} b').
%\]
%Since $\lb$ adds precisely the length 1 singular strings, which correspond to the path $\bon6 \to 1$. They are selected by $\delta_1$, and so the continuation of the procedure for $\delta_1$ and the entire $\delta_6$ agree as claimed.
%
%Thus it remains to show that the second application of $\delta$ does not remove any additional boxes. Let $(\nu_{\delta}, J_{\delta}) = \delta_1(\nu, J)$. Suppose there is an unbarred $a \in b'$. If $\overline{a} \in b$, then there are no singular strings in $\nu^{(a)}_{\delta}$ by the change in vacancy numbers. If $\overline{a} \notin b$, then there exists no singular string in $\nu^{(a)}$ as it would be selected by $\delta_6$ or $a = 1$ and $\lb$ makes all strings in $\nu^{(1)}$ non-singular. Thus the claim follows.
%\end{proof}

In type $D_{n+1}^{(2)}$, we define $\lb \colon B^{1,1} \otimes B^{\bullet} \to B^{n,1} \otimes B^{n,1} \otimes B^{\bullet}$ by
\begin{align*}
\young(1) \otimes b^{\bullet} & \mapsto 1 \bn \otimes n \otimes b^{\bullet},
\\ \young(\emptyset) \otimes b^{\bullet} & \mapsto \bn \otimes n \otimes b^{\bullet},
\end{align*}
and extended as a $U_q(\g_0)$-crystal embedding.

\begin{lemma}
\label{lemma:E6_adjoint_equivalence}
Consider $B = B^{\theta,1} \otimes B^{\bullet}$ of type $\g$. Then, we have
\[
\delta_\theta = \begin{cases}
\delta_n \circ \delta_1 = \delta_1 \circ \delta_n & \text{if } \g = A_n^{(1)}, \\
\delta_1 \circ \delta_1 \circ \lb & \text{if } \g = D_n^{(1)}, \\
\delta_n \circ \delta_n \circ \lb & \text{if } \g = D_{n+1}^{(2)}, \\
\delta_1 \circ \delta_1 \circ \lb \circ \delta_1 \circ \lb & \text{if } \g = E_6^{(1)}, \\
\delta_7 \circ \delta_7 \circ \lb & \text{if } \g = E_7^{(1)}.
\end{cases}
\]
\end{lemma}

\begin{proof}
We first consider type $E_6^{(1)}$.

Note that $B(\theta) \subseteq B(\clfw_1)^{\otimes 3}$ is the unique factor with highest weight element $2 \overline{6} \otimes \bon 6 \otimes 1$.
By Lemma~\ref{lemma:dual_equivalence}, % and Equation~\eqref{eq:lb_E6_dual}
we have
\[
\delta_1 \circ \delta_1 \circ \lb \circ \delta_1 \circ \lb = \delta_6 \circ \delta_1 \circ \lb.
%\delta_6 \circ \delta_6 \circ \lb^{\vee} \circ \delta_1 \circ \lb = \delta_1 \circ \delta_6 \circ \lb^{\vee}.
\]
Therefore, it is sufficient to show $\delta_{\theta} = \delta_6 \circ \delta_1 \circ \lb$. %\delta_1 \circ \delta_6 \circ \lb^{\vee}$.
%Note that $6 = e_6 e_5 e_4 e_3 e_1 \bon 2$, so we add a singular row of length 1 to $\nu^{(a)}$ for $a \in \{1, 3, 4, 5, 6\}$.
%Note that we have the unique highest vector of
%\[
%\bon 2 \otimes 1 \in B(\clfw_2) \subseteq B(\clfw_6) \otimes B(\clfw_1) \iso B^{6,1} \otimes B^{1,1}.
%2\bsix \otimes 6 \in B(\clfw_2) \subseteq B(\clfw_1) \otimes B(\clfw_6) \iso B^{1,1} \otimes B^{6,1}.
%\]
Note that the return values can be ignored since there is a unique $B(\theta) \subseteq B(\clfw_6) \otimes B(\clfw_1)$ and unique $B(\clfw_6) \subseteq B(\clfw_1) \otimes B(\clfw_1)$, giving a canonical identification of the return values. Furthermore, the only non-zero weight multiplicity of $B(\theta)$ is for weight $0$, but these are distinguished by $\varepsilon_a(y_{a'}) = \delta_{aa'}$, which in terms of the algorithm $\delta_{\theta}$ is determined by the partition $\nu^{(a)}$ the algorithm terminates at.

We proceed by induction on depth of the return value of $\delta_{\theta}$. The base case is when $\delta_{\theta}$ returns $u_{\theta}$ as no boxes are removed, and so we wish to show $\delta_6 \circ \delta_1 \circ \lb$ returns $2 \bsix \otimes 6$. Adding $u_{\theta}$ under $\delta_{\theta}^{-1}$ does not change the rigged configuration. Similarly, adding $6$ under $\delta_6^{-1}$ only makes all rows in $\nu^{(6)}$ non-singular. Hence, all boxes added by $2 \bsix$ under $\delta_1^{-1}$ must be of length $1$, which are all subsequently removed by $\lb^{-1}$. Hence, we have $\delta_{\theta} = \delta_6 \circ \delta_1 \circ \lb$.

Consider $b' \otimes b \in B(\theta) \subseteq B(\clfw_1) \otimes B(\clfw_6)$. If $f_a(b' \otimes b) = (f_a b') \otimes b$, then the claim follows from induction as the new box is selected on the application of $\delta_1$. Therefore, suppose $f_a(b' \otimes b) = b' \otimes (f_a b)$. Therefore, we must have $\varepsilon_a(b') = 0 < 1 = \varphi_a(b)$ by the tensor product rule. If $\varphi_a(b') = 1$, then we have $\varphi_a(b' \otimes b) = 2$, and hence, $b' \otimes b$ corresponds to $x_{\alpha_a}$. In this case, $f_a(b' \otimes b)$ differs by a quasisingular string in $\nu^{(a)}$, hence after $\delta_1 \circ \lb$, this string becomes singular and selected by $\delta_6$. Thus, the result of $\delta_6 \circ \delta_1 \circ \lb$ is $b' \otimes f_a b$. Therefore, we have $\varphi_a(b') = 0$. Hence, the new singular row in $\nu^{(6)}$ is selected during $\delta_6$. Hence, we have $\delta_{\theta} = \delta_6 \circ \delta_1 \circ \lb$.

% -6
% -1,6 |--> -6 \otimes -1,6
% lb adds 2,-6
Finally, we consider the case when $\delta_{\theta}$ returns $\emptyset$. In this case, we want to show $\delta_6 \circ \delta_1 \circ \lb$ returns $6$ and $\bsix$ respectively. Adding $\emptyset$ under $\delta_{\theta}^{-1}$ adds $c_a$ singular rows of length $1$ to $\nu^{(a)}$. Next, we note that $\delta_6^{-1}$ adding $6$ makes all rows in $\nu^{(6)}$ non-singular. Therefore, adding $\bsix$ under $\delta_1^{-1}$ can only create singular rows of length $1$ and, after performing $\lb^{-1}$, there are precisely $c_a$ such rows. Hence, we have $\delta_{\theta} = \delta_6 \circ \delta_1 \circ \lb$.

The other cases are similar.
\end{proof}

\begin{lemma}
Let $\g$ be of type $C_n^{(1)}$. Then $\delta_1 = \KSS{\delta}$.
\end{lemma}

\begin{proof}
This follows from the description of $\delta_1$ and Proposition~\ref{prop:type_A_virtualization}.
\end{proof}

By combining the above results and noting that our proofs did not rely on any specific choice of $\lb$-diagram (with a given sink), we have the following.

\begin{theorem}
\label{thm:equivalence}
Let $\g$ be of dual untwisted type or type $A_{2n}^{(2)}$, $A_{2n}^{(2)\dagger}$, $C_n^{(1)}$.
Let $\sigma$ be either a minuscule or adjoint node.
Let $\Phi$ be a bijection defined by $\delta_{\sigma}$ with a corresponding $\lb$-diagram. Then we have
\[
\Phi = \KSS{\Phi}.
\]
Moreover, for any bijection $\Phi'$ defined by $\delta_{\sigma'}$, where $\sigma'$ is either a minuscule or adjoint node, with a corresponding $\lb$-diagram, then we have $\Phi = \Phi'$.
\end{theorem}

\begin{conjecture}
Theorem~\ref{thm:equivalence} holds for type $B_n^{(1)}$.
\end{conjecture}

Theorem~\ref{thm:equivalence} states that there is a unique bijection $\Phi$ that is defined by KKR-type algorithms. In other words, for a fixed rigged configuration, we can only obtain different KR tableaux representations of the \emph{same} element in a tensor product of KR crystals under such a bijection. It is likely that there is a unique bijection that sends cocharge to energy and the combinatorial $R$-matrix to the identity map.

%=====================================================================
\section{The filling map}
\label{sec:filling_map}

In this section, we characterize all highest weight rigged configurations appearing in $\hwRC(B^{r,s})$ for various $r \in I_0$ in types $E_{6,7,8}^{(1)}$, $E_6^{(2)}$, and $F_4^{(1)}$.
Note that it has not been shown that $B^{r,s}$ is the crystal basis of the KR module $W^{r,s}$ in general.
For non-exceptional types, this was shown in~\cite{OS08} and for certain other special cases in~\cite{JS10,KMOY07,Yamane98}. We note that when $r$ is a special (resp.\ adjoint) node, then $W^{r,s}$ has a crystal basis by~\cite[Rem.~3.1]{OS08} and~\cite[Prop.~3.4.4]{KKMMNN92} (resp.~\cite[Prop.~3.4.5]{KKMMNN92}).
However, this provides further evidence for this to be true as we show many of the graded decompositions agrees with the conjectures of~\cite{HKOTY99,HKOTT02}.
Using this, we describe the filling map and define \defn{Kirillov--Reshetikhin (KR) tableaux} for $B^{r,s}$ for certain $r$ in types $E_{6,7,8}^{(1)}$ and $E_6^{(2)}$ as an extension of~\cite{OSS13,SchillingS15,Scrimshaw15}.
We give a table of the KR tableaux for $B^{r,1}$ in Appendix~\ref{sec:KR_tableaux}.
Note that the height of the KR tableaux is the distance from $r$ to $\sigma$ in the $\lb$-diagram (recall $\sigma$ is the unique sink in the $\lb$-diagram); in particular, $r$ no longer necessarily corresponds to the height.

%\TravisR{For $B^{r,1}$, with $r$ a minuscule node, we note that the weight spaces are all 1 dimensional. Hence, we can define a twisted automorphism by sending $I \setminus \{0,r\}$ components to themselves as that decomposition is multiplicity-free.  Since the $I_r := I \setminus \{r\}$ and $I_0$ decompositions uniquely determine the structure,  we have a combinatorial description of $e_0$ and $f_0$ in this case.}

For the nodes we do not consider here, the roots which appear as edge labels in the (ambient) Kleber tree, and hence the weights can appear in the decompositions, are completely determined by those at level $1$, \textit{i.e.}, appear in the (ambient) Kleber tree for $B^{r,1}$. Thus, it is possible to determine an explicit parameterization of the classical decomposition of $B^{r,s}$ for all exceptional types, as well as the corresponding rigged configurations (and their cocharge). In particular, this parameterization will be given by integer points in a polytope.

However, the author believes any such parameterization is likely to not be enlightening as it will involve numerous linear inequalities (and possibly some equalities). Then for the filling map, the rules can become even more complicated. For example, consider the parameterization for $B^{2,s}$ of type $D_4^{(3)}$ given in~\cite[Def.~4.10]{Scrimshaw15} and the corresponding filling map.

%%%%%%%%%%
\subsection{Notation}

We define some notation to aid in the description of $\hwRC(B^{r,s})$.
Consider a tuple $(\alpha^{(1)}, \alpha^{(2)}, \dotsc, \alpha^{(\ell)})$, where $\alpha^{(k)} = \sum_{a \in I_0} c_a^{(k)} \clsr_a \in Q^+_0$ for all $1 \leq k \leq \ell$. We denote by $\nu(\alpha^{(1)}, \alpha^{(2)}, \dotsc, \alpha^{(\ell)})$ the configuration given by $\nu^{(a)}$ stacking a column of height $c_a^{(k)}$ for all $1 \leq k \leq \ell$ and left justifying this (so it is a partition). We also denote $k \ast [\alpha]$ as the sequence $(\alpha, \alpha, \dotsc, \alpha)$ of length $k$ and $k_1 \ast [\alpha^{(1)}] + k_2 \ast [\alpha^{(2)}]$ as the concatenation of the two sequences.

%%%%%%%%%%
\subsection{Type $E_6^{(1)}$}

% --- r = 1,6 ---

For $r = 1$ and $r = 6$ in type $E_6^{(1)}$, we have $B^{r,s} \iso B(s\clfw_r)$ as $U_q(\g_0)$-crystals. Thus, the filling map $\fillmap \colon B^{r,s} \to T^{r,s}$ is the identity map. Moreover, we have $\RC(B^{r,s}) \iso \RC(B^{r,s}; s \clfw_r)$ as $U_q(\g_0)$-crystals.

% --- r = 2 ---

\begin{proposition}
\label{prop:RC_E6_2}
Consider the KR crystal $B^{2,s}$ of type $E_6^{(1)}$. We have
\[
\RC(B^{2,s}) = \bigoplus_{k=0}^s \RC(B^{2,s}; (s-k) \clfw_2).
\]
Moreover, the highest weight rigged configurations in $\RC(B^{2,s})$ are given by $\nu(k\ast[\alpha])$ with all riggings $0$, where
\[
\alpha = \clsr_1 + 2\clsr_2 + 2\clsr_3 + 3\clsr_4 + 2\clsr_5 + \clsr_6 = \clfw_2,
\]
and
\[
\cc(\nu, J) = k.
\]
\end{proposition}

\begin{proof}
Note that by Condition~(K2) of Definition~\ref{def:kleber_algorithm}, the only root that we can subtract from $\clfw_2$ is $\alpha$. Thus the Kleber tree $T(B^{2,s})$ is a path where the node at depth $k$ has weight $\clfw_2 - k\alpha$ for $0 \leq k \leq s$. Hence, the rigged configuration is given by $\nu(k \ast [\alpha])$. It is straightforward to check that $\cc(\nu, J) = k$.
\end{proof}

\begin{definition}
\label{def:filling_E6_2}
We define the filling map $\fillmap \colon B^{2,s} \to T^{2,s}$ on the classically highest weight element $b \in B\bigl( (s-k)\Lambda_2\bigr) \subseteq B^{2,s}$ by defining $\fillmap(b)$ as the tableau with the first $s-k$ columns as $\column{1, \bon6, 2\bsix}$, the next $\lfloor k / 2 \rfloor$ columns as
\[
\begin{tikzpicture}[baseline]
\matrix [matrix of math nodes,column sep=-.4, row sep=-.5,text height=10,text width=15,align=center,inner sep=3] 
 {
 	\node[draw]{\btw5}; &
	\node[draw]{1}; \\
	\node[draw]{\bfive6}; &
	\node[draw]{\bon6}; \\
	\node[draw]{\bsix}; &
	\node[draw]{2\bsix}; \\
 };
\end{tikzpicture}
\]
and if $k$ is odd, the final column as $\column{1, \bon6, \bsix}$. We then extend $\fillmap$ as a classical crystal isomorphism.
\end{definition}

\begin{example}
Consider $B^{2, 8}$ of type $E_6^{(1)}$. Then $B(k\clfw_2) \subseteq B^{2,8}$ corresponds to the classically highest weight element
\begin{align*}
k = 4 \rightsquigarrow &
\begin{tikzpicture}[baseline]
\matrix [matrix of math nodes,column sep=-.4, row sep=-.5,text height=10,text width=15,align=center,inner sep=3, ampersand replacement=\&] % Last part is because this is in an align* environment
 {
	\node[draw]{1}; \&
	\node[draw]{1}; \&
	\node[draw]{1}; \&
	\node[draw]{1}; \&
 	\node[draw,fill=gray!20]{\btw5}; \&
	\node[draw,fill=gray!20]{1}; \&
	\node[draw,fill=gray!20]{\btw5}; \&
	\node[draw,fill=gray!20]{1}; \\
	\node[draw]{\bon6 }; \&
	\node[draw]{\bon6 }; \&
	\node[draw]{\bon6 }; \&
	\node[draw]{\bon6 }; \&
	\node[draw,fill=gray!20]{\bfive6}; \&
	\node[draw,fill=gray!20]{\bon6 }; \&
	\node[draw,fill=gray!20]{\bfive6}; \&
	\node[draw,fill=gray!20]{\bon6 }; \\
	\node[draw]{2\bsix}; \&
	\node[draw]{2\bsix}; \&
	\node[draw]{2\bsix}; \&
	\node[draw]{2\bsix}; \&
	\node[draw,fill=gray!20]{\bsix}; \&
	\node[draw,fill=gray!20]{2\bsix}; \&
	\node[draw,fill=gray!20]{\bsix}; \&
	\node[draw,fill=gray!20]{2\bsix}; \\
 };
\end{tikzpicture},
\\ k = 5 \rightsquigarrow &
\begin{tikzpicture}[baseline]
\matrix [matrix of math nodes,column sep=-.4, row sep=-.5,text height=10,text width=15,align=center,inner sep=3, ampersand replacement=\&] % Last part is because this is in an align* environment
 {
	\node[draw]{1}; \&
	\node[draw]{1}; \&
	\node[draw]{1}; \&
 	\node[draw,fill=gray!20]{\btw5}; \&
	\node[draw,fill=gray!20]{1}; \&
	\node[draw,fill=gray!20]{\btw5}; \&
	\node[draw,fill=gray!20]{1}; \&
	\node[draw,fill=gray!40]{1}; \\
	\node[draw]{\bon6}; \&
	\node[draw]{\bon6}; \&
	\node[draw]{\bon6}; \&
	\node[draw,fill=gray!20]{\bfive6}; \&
	\node[draw,fill=gray!20]{\bon6}; \&
	\node[draw,fill=gray!20]{\bfive6}; \&
	\node[draw,fill=gray!20]{\bon6}; \&
	\node[draw,fill=gray!40]{\bon6}; \\
	\node[draw]{2\bsix}; \&
	\node[draw]{2\bsix}; \&
	\node[draw]{2\bsix}; \&
	\node[draw,fill=gray!20]{\bsix}; \&
	\node[draw,fill=gray!20]{2\bsix}; \&
	\node[draw,fill=gray!20]{\bsix}; \&
	\node[draw,fill=gray!20]{2\bsix}; \&
	\node[draw,fill=gray!40]{\bsix}; \\
 };
\end{tikzpicture},
\end{align*}
where the shaded regions are the parts that are ``filled in.'' The corresponding rigged configurations, respectively, are
\begin{gather*}
\begin{tikzpicture}[scale=.25,baseline=-18]
 \rpp{4}{0}{0}
 \begin{scope}[xshift=8cm]
 \rpp{4,4}{0}{0,0}
 \end{scope}
 \begin{scope}[xshift=16cm]
 \rpp{4,4}{0,0}{0,0}
 \end{scope}
 \begin{scope}[xshift=24cm]
 \rpp{4,4,4}{0,0,0}{0,0,0}
 \end{scope}
 \begin{scope}[xshift=32cm]
 \rpp{4,4}{0,0}{0,0}
 \end{scope}
 \begin{scope}[xshift=40cm]
 \rpp{4}{0}{0}
 \end{scope}
\end{tikzpicture},
\\
\begin{tikzpicture}[scale=.25,baseline=-18]
 \rpp{5}{0}{0}
 \begin{scope}[xshift=8cm]
 \rpp{5,5}{0}{0,0}
 \end{scope}
 \begin{scope}[xshift=16cm]
 \rpp{5,5}{0,0}{0,0}
 \end{scope}
 \begin{scope}[xshift=24cm]
 \rpp{5,5,5}{0,0,0}{0,0,0}
 \end{scope}
 \begin{scope}[xshift=32cm]
 \rpp{5,5}{0,0}{0,0}
 \end{scope}
 \begin{scope}[xshift=40cm]
 \rpp{5}{0}{0}
 \end{scope}
\end{tikzpicture},
\end{gather*}
\end{example}

\begin{proposition}
\label{prop:filling_iso_E6_2}
Let $\fillmap \colon B^{2,s} \to T^{2,s}$ given by Definition~\ref{def:filling_E6_2} and $\iota$ be the natural (classical) crystal isomorphism. We have
\[
\Phi = \fillmap \circ \iota
\]
on classically highest weight elements.
\end{proposition}

\begin{proof}
If we apply column splitting $\ls$ to $B^{2,s}$ when $s > k$, then we make all rows in $\nu^{(2)}$ nonsingular. A straightforward computation shows that we obtain $\column{1, \bon3, 2\bth}$ and the rigged configuration  has not changed. If $s = k$, then $\ls$ keeps all rows of length $k$ being singular. It is a straightforward computation shows the column is $\column{1\btw, \bon6, 6}$ when $k > 1$ and the new rigged configuration is obtained by deleting two boxes from every row of $\nu^{(a)}$ for all $a \in I_0$. If $k = 1$, then a finite computation shows we obtain $\column{1, \bon6, \bsix}$.
\end{proof}

% --- r = 3,5 ---

\begin{proposition}
\label{prop:RC_E6_3}
Consider the KR crystal $B^{3,s}$ of type $E_6^{(1)}$. We have
\[
\RC(B^{3,s}) = \bigoplus_{k=0}^s \RC(B^{3,s}; (s-k) \clfw_3 + k \clfw_6).
\]
Moreover the highest weight rigged configurations in $\RC(B^{3,s})$ are given by $\nu(k\ast[\alpha])$ with all riggings $0$,
where
\[
\alpha = \clsr_1 + \clsr_2 + 2\clsr_3 + 2\clsr_4 + \clsr_5 = \clfw_3 - \clfw_6,
\]
and
\[
\cc(\nu, J) = k.
\]
\end{proposition}

\begin{proof}
Similar to the proof of Proposition~\ref{prop:RC_E6_2}.
\end{proof}

\begin{definition}
\label{def:filling_E6_3}
We define the filling map $\fillmap \colon B^{3,s} \to T^{3,s}$ on classically highest weight element $b \in B((s-k) \clfw_3 + k \clfw_6) \subseteq B^{3,s}$ by defining $\fillmap(b)$ as the tableau with the first $s-k$ columns as $\column{1, \bon3}$, the next $\lfloor k / 2 \rfloor$ columns as
\[
\begin{tikzpicture}[baseline]
\matrix [matrix of math nodes,column sep=-.4, row sep=-.5,text height=10,text width=15,align=center,inner sep=3] 
 {
 	\node[draw]{1\bth6}; &
	\node[draw]{1}; \\
	\node[draw]{\bon6}; &
	\node[draw]{\bon3}; \\
 };
\end{tikzpicture}
\]
and if $k$ is odd, the last column as $\column{1, \bon6}$. We then extend $\fillmap$ as a classical crystal isomorphism.
\end{definition}

\begin{proposition}
\label{prop:filling_iso_E6_3}
Let $\fillmap \colon B^{3,s} \to T^{3,s}$ given by Definition~\ref{def:filling_E6_3} and $\iota$ be the natural (classical) crystal isomorphism. We have
\[
\Phi = \fillmap \circ \iota
\]
on classically highest weight elements.
\end{proposition}

\begin{proof}
Similar to the proof of Proposition~\ref{prop:filling_iso_E6_2}.
\end{proof}

The case for $r = 5$ is dual to the above. In particular, the proofs of the following propositions are similar to those for the $r = 3$ case.

\begin{proposition}
\label{prop:RC_E6_5}
Consider the KR crystal $B^{5,s}$ of type $E_6^{(1)}$. We have
\[
\RC(B^{5,s}) = \bigoplus_{k=0}^s \RC(B^{5,s}; (s-k) \clfw_5 + k \clfw_1).
\]
Moreover the highest weight rigged configurations in $\RC(B^{5,s})$ are given by $\nu(k\ast[\alpha])$ with all riggings $0$,
where
\[
\alpha = \clsr_2 + \clsr_3 + 2\clsr_4 + \clsr_5 + \clsr_6 = \clfw_5 - \clfw_1,
\]
and
\[
\cc(\nu, J) = k.
\]
\end{proposition}

\begin{definition}
\label{def:filling_E6_5}
We define the filling map $\fillmap \colon B^{5,s} \to T^{5,s}$ on classically highest weight element $b \in B\bigl( (s-k) \clfw_5 + k \clfw_1 \bigr) \subseteq B^{5,s}$ by defining $\fillmap(b)$ as the tableau with the first $s-k$ columns be $\column{1, \bon3, 2\bth, \btw5}$, the next $\lfloor k / 2 \rfloor$ columns as
\[
\begin{tikzpicture}[baseline]
\matrix [matrix of math nodes,column sep=-.4, row sep=-.5,text height=10,text width=15,align=center,inner sep=3] 
 {
 	\node[draw]{1}; &
	\node[draw]{1}; \\
	\node[draw]{2\bfive6}; &
	\node[draw]{\bon3}; \\
	\node[draw]{\bsix}; &
	\node[draw]{2\bth}; \\
	\node[draw]{1\btw}; &
	\node[draw]{\btw5}; \\
 };
\end{tikzpicture}
\]
and if $k$ is odd, the last column as $\column{1, \bon3, 2\bth, \btw1}$. We then extend $\fillmap$ as a classical crystal isomorphism.
\end{definition}

\begin{proposition}
\label{prop:filling_iso_E6_5}
Let $\fillmap \colon B^{5,s} \to T^{5,s}$ given by Definition~\ref{def:filling_E6_5} and $\iota$ be the natural (classical) crystal isomorphism. We have
\[
\Phi = \fillmap \circ \iota
\]
on classically highest weight elements.
\end{proposition}

% --- r = 4 ---

\begin{proposition}
\label{prop:RC_E6_4}
Consider the KR crystal $B^{4,s}$ of type $E_6^{(1)}$. We have
\[
\RC(B^{4,s}) = \bigoplus_{\lambda} \RC(B^{4,s}; \lambda)^{\oplus (1 + k_2 - k_4 - 2k_5)},
\]
where $\lambda = s \clfw_4 - \sum_{i=1}^5 k_i \clsr^{(i)}$ with
\begin{align*}
\clsr^{(1)} & := 2\clsr_1 + 3\clsr_2 + 4\clsr_3 + 6\clsr_4 + 4\clsr_5 + 2\clsr_6 && = \clfw_4,
\\ \clsr^{(2)} & := \clsr_1 + \clsr_2 + 2 \clsr_3 + 3 \clsr_4 + 2 \clsr_5 + \clsr_6 && = \clfw_4 - \clfw_2,
\\ \clsr^{(3)} & := \clsr_2 + \clsr_3 + 2\clsr_4 + \clsr_5 && = \clfw_4 - \clfw_1 - \clfw_6,
\\ \clsr^{(4)} & := \clsr_2 + \clsr_4 && = \clfw_2 + \clfw_4 - (\clfw_3 + \clfw_5),
\\ \clsr^{(5)} & := \clsr_2 && = 2\clfw_2 - \clfw_4,
\end{align*}
such that
\begin{enumerate}[{\rm (1)}]
\item $k_1 + k_2 + k_3 + k_4 \leq s$,
\item $k_i \geq 0$ for $i = 1,2,3,4,5$, and
\item $k_2 \geq k_4 + 2k_5$.
\end{enumerate}
Moreover the highest weight rigged configurations in $\RC(B^{4,s})$ are given by
\[
\nu( k_1\ast [\clsr^{(1)}] + k_2 \ast [\clsr^{(2)}] + k_3 \ast [\clsr^{(3)}] + k_4 \ast [\clsr^{(4)}]+ k_5 \ast [\clsr^{(5)}] ),
\]
with all riggings $0$ except for $x$ in the first row of $\nu^{(2)}$, which satisfies $0 \leq x \leq k_2 - k_4 - 2k_5$, and
\[
\cc(\nu, J) = 3k_1 + k_2 + k_3 + k_4 + k _5 + x.
\]
\end{proposition}

\begin{proof}
Note that $\clsr^{(1)} > \clsr^{(2)} > \clsr^{(3)} > \clsr^{(4)} > \clsr^{(5)}$ component-wise expressed in terms of $\{ \clsr_i \}_{i \in I_0}$.
The claim for the rigged configurations follows from Definition~\ref{def:kleber_algorithm} and the description of $\clsr^{(i)}$ for $i = 1,2,3,4,5$. The claim for the cocharge is a straightforward computation.
\end{proof}

We first note that for $r =1,2,3,5,6$, it is clear that these graded decompositions agree with the conjectures in~\cite[App.~A]{HKOTY99} (with some straightforward relabeling). Thus, we show the $r = 4$ case.

\begin{proposition}
\label{prop:equivalence_E6_4}
Let $q^k B(\clfw_r)$ denote that $B(\clfw_r)$ is given a grading of $k$.
We have
\begin{align*}
B^{4,s} & \iso \bigoplus_{\substack{j_1+j_2+2j_3+j_4 \leq s \\ j_1,j_2,j_3,j_4 \in \ZZ_{\geq 0}}} \min(1 + j_2, 1 + s - j_1 - j_2 - 2j_3 - j_4) q^{3s - 2j_1 - 3j_2 - 4j_3 - 2j_4}
\\ & \hspace{75pt} \times \sum_{k=0}^{j_1} q^k B\bigl(j_1 \clfw_2 + j_2 \clfw_4 + j_3 (\clfw_3 + \clfw_5) + j_4 (\clfw_1 + \clfw_6)\bigr).
\end{align*}
\end{proposition}

\begin{proof}
First note that
\begin{align*}
\wt(\nu, J) & = s\clfw_4 - k_1 \clfw_2 - k_2(\clfw_4 - \clfw_2) - k_3(\clfw_4 - \clfw_1 - \clfw_6)
\\ & \hspace{30pt} - k_4\bigl(\clfw_2 + \clfw_4 - (\clfw_3 + \clfw_5) \bigr) - k_5 (2\clfw_2 - \clfw_4)
\\ & = (k_2 - k_4 - 2k_5)\clfw_2 + (s - k_1 - k_2 - k_3 - k_4 + k_5) \clfw_4
\\ & \hspace{30pt} + k_4 (\clfw_3 + \clfw_5) + k_3 (\clfw_1 + \clfw_6)
\end{align*}
Hence, we must have
\begin{equation}
\label{eq:k_to_j}
\begin{aligned}
j_1 & = k_2 - k_4 - 2k_5,
&
j_2 & = s - k_1 - k_2 - k_3 - k_4 + k_5,
\\
j_3 & = k_4,
&
j_4 & = k_3.
\end{aligned}
\end{equation}
Note that the conditions on $k_1, k_2, k_3, k_4, k_5$ guarantee that $j_1, j_2, j_3, j_4 \geq 0$. Next, we have
\begin{align*}
j_1 + j_2 + 2j_3 + j_4 & = s - k_1 - k_5 \leq s,
\\ 3s - 2j_1 - 3 j_2 - 4 j_3 - 2 j_4 & = 3k_1 + k_2 + k_3 + k_4 + k_5 = \cc(\nu).
\end{align*}
Additionally, we want the multiplicity to agree, so we need to show
\[
M := \min(1 + j_2, 1 + s - j_1 - j_2 - 2j_3 - j_4)
\]
equals the number of time a node of weight $j_1 \clfw_2 + j_2 \clfw_4 + j_3 (\clfw_3 + \clfw_5) + j_4 (\clfw_1 + \clfw_6)$ occurs since the multiplicity of a node equals $1 + k_2 - k_4 - 2k_5 = 1 + j_1$. We note that $\clsr^{(1)} = 2\clsr^{(2)} + \clsr^{(5)}$. Explicitly, we fix some $j_1, j_2, j_3, j_4 \in \ZZ_{\geq 0}$ such that $j_1 + j_2 + 2j_3 + j_4 \leq s$, and we want to show $M$ equals the number of $k_1, k_2, k_3, k_4, k_5 \in \ZZ_{\geq 0}$ such that
\begin{subequations}
\begin{align}
\label{eq:s_k1234} s &\geq k_1 + k_2 + k_3 + k_4,
\\ \label{eq:k2_k45} k_2 &\geq k_4 + 2k_5,
\end{align}
\end{subequations}
and, from Equation~\eqref{eq:k_to_j}, that:
\begin{equation}
\begin{gathered}
k_1 + k_5 = s - j_1 - j_2 - 2j_3 - j_4,
\qquad k_2 - 2 k_5 = j_1 + j_3,
\\ k_3 = j_4,
\qquad k_4 = j_3.
\end{gathered}
\end{equation}

Immediately, we have $k_3, k_4 \geq 0$ and are completely determined since we have fixed $j_3, j_4 \in \ZZ_{\geq 0}$. We note that
\begin{gather}
\label{eq:recover_k2_k45}
k_2 = j_1 + j_3 + 2k_5 \geq j_3 + 2k_5 = k_4 + 2k_5,
\\ \label{eq:k1234} k_1 + k_2 + k_3 + k_4 = s - j _2 + k_5,
\end{gather}
and so if we take $k_5 = 0$, then $k_1$ and $k_2$ are completely determined and give a valid set. Next, since we want to look over all possible values that give rise to the same weight, the linear dependence $\clsr^{(1)} = 2\clsr^{(2)} + \clsr^{(5)}$ implies we can replace $k_1 \mapsto 2k_2 + k_5$. Note Equation~\eqref{eq:recover_k2_k45}, and hence Equation~\eqref{eq:k2_k45}, still holds under this replacement, but from Equation~\eqref{eq:k1234} and Equation~\eqref{eq:s_k1234}, we have $k_5 \leq j_2$ (so we can do this replacement at most $j_2$ times). However, we can only do this at most $s - j_1 - j_2 - 2j_3 - j_4$ times since it is the maximum value of $k_1$. Thus, we have exactly $M$ possible values for $k_1$, $k_2$, $k_3$, $k_4$, and $k_5$.
\end{proof}

%%%%%%%%%%
\subsection{Type $E_7^{(1)}$}

For $r = 7$ in type $E_7^{(1)}$, we have $\RC(B^{7,s}) \iso \RC(B^{7,s}; s \clfw_7)$. Threrefore the filling map $\fillmap \colon B^{7,s} \to T^{7,s}$ is the identity map.

%For my reference, the labels go `mine' $\to$ `theirs': $1 \mapsto 1$, $2 \mapsto 7$, $i \mapsto i-1$ for $3 \leq i \leq 7$.

% --- r = 1 ---

\begin{proposition}
\label{prop:RC_E7_1}
Consider the KR crystal $B^{1,s}$ of type $E_7^{(1)}$. We have
\[
\RC(B^{1,s}) \iso \bigoplus_{k=0}^s \RC(B^{1,s}; k \clfw_1).
\]
Moreover the highest weight rigged configurations in $\RC(B^{3,s})$ are given by $\nu(k\ast[\alpha])$ with all riggings $0$,
where
\[
\alpha = 2\clsr_1 + 2\clsr_2 + 3\clsr_3 + 4\clsr_4 + 3\clsr_5 + 2\clsr_6 + \clsr_7 = \clfw_1,
\]
and
\[
\cc(\nu, J) = k.
\]
\end{proposition}

\begin{proof}
Similar to the proof of Proposition~\ref{prop:RC_E6_2}.
\end{proof}

\begin{definition}
\label{def:filling_E7_1}
We define the filling map $\fillmap \colon B^{1,s} \to T^{1,s}$ on classically highest weight element $b \in B\bigl((s-k) \clfw_1\bigr) \subseteq B^{1,s}$ by defining $\fillmap(b)$ as the tableau with the first $s-k$ columns as $\column{7,1\bseven}$, the next $\lfloor k / 2 \rfloor$ columns as
\[
\begin{tikzpicture}[baseline]
\matrix [matrix of math nodes,column sep=-.4, row sep=-.5,text height=10,text width=15,align=center,inner sep=3] 
 {
 	\node[draw]{\bon7}; &
	\node[draw]{7}; \\
	\node[draw]{\bseven}; &
	\node[draw]{1\bseven}; \\
 };
\end{tikzpicture}
\]
and if $k$ is odd, the last column as $\column{7, \bseven}$. We then extend $\fillmap$ as a classical crystal isomorphism.
\end{definition}

\begin{proposition}
\label{prop:filling_iso_E7_1}
Let $\fillmap \colon B^{1,s} \to T^{1,s}$ given by Definition~\ref{def:filling_E7_1} and $\iota$ be the natural (classical) crystal isomorphism. We have
\[
\Phi = \fillmap \circ \iota
\]
on classically highest weight elements.
\end{proposition}

\begin{proof}
Similar to the proof of Proposition~\ref{prop:filling_iso_E6_2}.
\end{proof}

% --- r = 2 ---

\begin{proposition}
\label{prop:RC_E7_2}
Consider the KR crystal $B^{2,s}$ of type $E_7^{(1)}$. We have
\[
\RC(B^{2,s}) \iso \bigoplus_{k=0}^s \RC(B^{2,s}; (s-k) \clfw_2 + k \clfw_7)
\]
Moreover the highest weight rigged configurations in $\RC(B^{3,s})$ are given by $\nu(k\ast[\alpha])$ with all riggings $0$,
where
\[
\alpha =\clsr_1 + 2\clsr_2 + 2\clsr_2 + 3\clsr_4 + 2\clsr_5 + \clsr_6 = \clfw_2 - \clfw_7,
\]
and
\[
\cc(\nu, J) = k.
\]
\end{proposition}

\begin{proof}
Similar to the proof of Proposition~\ref{prop:RC_E6_2}.
\end{proof}

\begin{definition}
\label{def:filling_E7_2}
We define the filling map $\fillmap \colon B^{2,s} \to T^{2,s}$ on classically highest weight element $b \in B((s-k) \clfw_2 + k \clfw_7) \subseteq B^{2,s}$ by defining $\fillmap(b)$ as the tableau with the first $s-k$ columns as $\column{7, 1\bseven, \bon2}$, the next $\lfloor k / 2 \rfloor$ columns as
\[
\begin{tikzpicture}[baseline]
\matrix [matrix of math nodes,column sep=-.4, row sep=-.5,text height=10,text width=15,align=center,inner sep=3] 
 {
 	\node[draw]{7}; &
	\node[draw]{7}; \\
	\node[draw]{1\btw}; &
	\node[draw]{1\bseven}; \\
	\node[draw]{\bon7}; &
	\node[draw]{\bon2}; \\
 };
\end{tikzpicture}
\]
and if $k$ is odd, the last column as $\column{7, 1\bseven, \bon7}$. We then extend $\fillmap$ as a classical crystal isomorphism.
\end{definition}

\begin{proposition}
\label{prop:filling_iso_E7_2}
Let $\fillmap \colon B^{2,s} \to T^{2,s}$ given by Definition~\ref{def:filling_E7_2} and $\iota$ be the natural (classical) crystal isomorphism. We have
\[
\Phi = \fillmap \circ \iota
\]
on classically highest weight elements.
\end{proposition}

\begin{proof}
Similar to the proof of Proposition~\ref{prop:filling_iso_E6_2}.
\end{proof}

% --- r = 6 ---

\begin{proposition}
\label{prop:RC_E7_6}
Consider the KR crystal $B^{6,s}$ of type $E_7^{(1)}$. We have
\[
\RC(B^{6,s}) \iso \bigoplus_{ \substack{k_1 + k_2 \leq s \\ k_1,k_2 \in \ZZ_{\geq 0}} } \RC(B^{6,s}; (s - k_1 - k_2) \clfw_6 + k_2 \clfw_1).
\]
Moreover the highest weight rigged configurations in $\RC(B^{6,s})$ are given by
\[
\nu( k_1\ast [\clsr^{(1)}] + k_2 \ast [\clsr^{(2)}] ),
\]
where
\begin{align*}
\clsr^{(1)} & = 2\clsr_1 + 3\clsr_2 + 4\clsr_3 + 6\clsr_4 + 5\clsr_5 + 4\clsr_6 + 2\clsr_7 && = \clfw_6,
\\ \clsr^{(2)} & = \clsr_2 + \clsr_3 + 2\clsr_4 + 2\clsr_5 + 2\clsr_6 + \clsr_7 && = \clfw_6 - \clfw_1,
\end{align*}
with all riggings $0$ and
\[
\cc(\nu, J) = 2k_1 + k_2.
\]
\end{proposition}

\begin{proof}
Similar to the proof of Proposition~\ref{prop:RC_E6_4} as $\clsr^{(1)} > \clsr^{(2)}$.
\end{proof}

% --- r = 3 ---

\begin{proposition}
\label{prop:RC_E7_3}
Consider the KR crystal $B^{3,s}$ of type $E_7^{(1)}$. We have
\[
\RC(B^{3,s}) = \bigoplus_{\lambda} \RC(B^{3,s}; \lambda)^{\oplus (1 + k_2 - k_4 - 2k_5)},
\]
where $\lambda = s \clfw_3 - \sum_{i=1}^5 k_i \clsr^{(i)}$ with
\begin{align*}
\clsr^{(1)} & = 3\clsr_1 + 4\clsr_2 + 6\clsr_3 + 8\clsr_4 + 6\clsr_5 + 4\clsr_6 + 2\clsr_7 && = \clfw_3
\\ \clsr^{(2)} & = \clsr_1 + 2\clsr_2 + 3\clsr_3 + 4\clsr_4 + 3\clsr_5 + 2\clsr_6 + 1\clsr_7 && = \clfw_3 - \clfw_1
\\ \clsr^{(3)} & = \clsr_1 + \clsr_2 + 2\clsr_3 + 2\clsr_4 + \clsr_5 && = \clfw_3 - \clfw_6
\\ \clsr^{(4)} & = \clsr_1 + \clsr_3 && = \clfw_1 + \clfw_3 - \clfw_4
\\ \clsr^{(5)} & = \clsr_1 && = 2 \clfw_1 - \clfw_3
\end{align*}
such that
\begin{enumerate}[(1)]
\item $k_1 + k_2 + k_3 + k_4 \leq s$,
\item $k_i \geq 0$ for $i = 1,2,3,4,5$, and
\item $k_2 \geq k_4 + 2k_5$.
\end{enumerate}
Moreover the highest weight rigged configurations in $\RC(B^{3,s})$ are given by
\[
\nu( k_1\ast [\clsr^{(1)}] + k_2 \ast [\clsr^{(2)}] + k_3 \ast [\clsr^{(3)}] + k_4 \ast [\clsr^{(4)}]+ k_5 \ast [\clsr^{(5)}] ),
\]
with all riggings $0$ except for $x$ in the first row of $\nu^{(1)}$, which satisfies $0 \leq x \leq k_2 - k_4 - 2k_5$, and
\[
\cc(\nu, J) = 3k_1 + k_2 + k_3 + k_4 + k _5 + x.
\]
\end{proposition}

\begin{proof}
Similar to the proof of Proposition~\ref{prop:RC_E6_4}.
\end{proof}

% --- r = 5 ---

%For $r = 5$, we have
%\begin{align*}
%\clsr^{(1)} & = 2\clsr_1 + 3\clsr_2 + 4\clsr_3 + 6\clsr_4 + 5\clsr_5 + 3\clsr_6 + \clsr_7 && = \clfw_5 - \clfw_7,
%\\ \clsr^{(2)} & = \clsr_1 + \clsr_2 + 2\clsr_3 + 3\clsr_4 + 3\clsr_5 + 2\clsr_2 + \clsr_7 && = \clfw_5 - \clfw_2,
%\\ \clsr^{(3)} & = \clsr_2 + \clsr_3 + 2\clsr_4 + 2\clsr_5 + \clsr_6 && = \clfw_5 - \clfw_2 - \clfw_7,
%\\ \clsr^{(4)} & = \clsr_2 + \clsr_4 + \clsr_5 && = \clfw_2 + \clfw_5 - \clfw_3 - \clfw_6,
%\\ \clsr^{(5)} & = \clsr_5 + \clsr_6 + \clsr_7 && = \clfw_5 + \clfw_7 - \clfw_4,
%\\ \clsr^{(6)} & = \clsr_2 + \clsr_4 + \clsr_5 + \clsr_6 && = \clfw_2 + \clfw_6 - \clfw_3 - \clfw_7,
%\\ \clsr^{(7)} & = \clsr_2 && = 2\clfw_2 - \clfw_4,
%\\ \clsr^{(8)} & = \clsr_7 && = 2\clfw_4 - \clfw_6,
%\\ \clsr^{(9)} & = \clsr_6 + \clsr_7 && = \clfw_6 + \clfw_7 - \clfw_5,
%\\ \clsr^{(10)} & = \clsr_2 + \clsr_4 + \clsr_5 + \clsr_6 + \clsr_7 && = \clfw_2 + \clfw_7 - \clfw_3,
%\\ \clsr^{(11)} & = \clsr_5 + \clsr_6 && = \clfw_5 + \clfw_6 - \clfw_4 - \clfw_7,
%\\ \clsr^{(12)} & = \clsr_5 && = 2\clfw_5 - \clfw_4 - \clfw_6,
%\\ \clsr^{(13)} & = \clsr_2 + \clsr_3 + 2\clsr_4 + 2\clsr_5 + \clsr_6 + \clsr_7 && = \clfw_5 + \clfw_7 - \clfw_1 - \clfw_6,
%\end{align*}
%\TravisR{This is hard... I think this is all of the roots, but I'm not sure.}
%\Travis{In principle, we should be able to do this and show it agrees with~\cite{HKOTY99}.}

% --- r = 4 ---

% For $r = 4$, we have
% Basically it seems impossible.

We note that our graded decompositions for $r = 1,2,3,6,7$ agree with the conjectures of~\cite[App.~A]{HKOTY99}, where the case of $r=3$ is similar to the proof of Proposition~\ref{prop:equivalence_E6_4}.

%%%%%%%%%%
\subsection{Type $E_8^{(1)}$}

For type $E_8^{(1)}$, we do not have any minuscule nodes as there are no diagram automorphism. Thus, we begin with the adjoint node.

% --- r = 8 ---

\begin{proposition}
Consider the KR crystal $B^{8,s}$ of type $E_8^{(1)}$. We have
\[
\RC(B^{8,s}) = \bigoplus_{k=0}^s \RC(B^{8,s}; (s-k) \clfw_8).
\]
Moreover the highest weight rigged configurations in $\RC(B^{8,s})$ are given by $\nu(k \ast [\alpha])$ with all riggings $0$,
where
\[
\alpha = 2 \clsr_1 + 3 \clsr_2 + 4 \clsr_3 + 6 \clsr_4 + 5 \clsr_5 + 4 \clsr_6 + 3 \clsr_7 + 2 \clsr_8 = \clfw_8
\]
and $\cc(\nu, J) = k$.
\end{proposition}

\begin{proof}
The proof is similar to Proposition~\ref{prop:RC_E6_2}.
\end{proof}

\begin{definition}
\label{def:filling_E8_8}
We define the filling map $\fillmap \colon B^{8,s} \to T^{8,s}$ on classically highest weight element $b \in B\bigl((s-k) \clfw_8\bigr) \subseteq B^{8,s}$ by defining $\fillmap(b)$ as the tableau with the first $s-k$ columns as $\column{8}$, the next $\lfloor k / 2 \rfloor$ columns as $\young(\beight8)$, and if $k$ is odd, the last column as $\column{\emptyset}$. We then extend $\fillmap$ as a classical crystal isomorphism.
\end{definition}

\begin{proposition}
\label{prop:filling_iso_E8_8}
Let $\fillmap \colon B^{8,s} \to T^{8,s}$ given by Definition~\ref{def:filling_E8_8} and $\iota$ be the natural (classical) crystal isomorphism. We have
\[
\Phi = \fillmap \circ \iota
\]
on classically highest weight elements.
\end{proposition}

\begin{proof}
Similar to the proof of Proposition~\ref{prop:filling_iso_E6_2}.
\end{proof}

% --- r = 7 ---

\begin{proposition}
Consider the KR crystal $B^{7,s}$ of type $E_8^{(1)}$. We have
\[
\RC(B^{7,s}) = \bigoplus_{\lambda} \RC(B^{7,s}; \lambda)^{\oplus (1 + k_2 - k_4 - 2k_5)}
\]
where
\begin{align*}
\clsr^{(1)} & = 4\clsr_1 + 6\clsr_2 + 8\clsr_3 + 12\clsr_4 + 10\clsr_5 + 8\clsr_6 + 6\clsr_7 + 3\clsr_8 && = \clfw_7,
\\ \clsr^{(2)} & = 2\clsr_1 + 3\clsr_2 + 4\clsr_3 + 6\clsr_4 + 5\clsr_5 + 4\clsr_6 + 3\clsr_7 + \clsr_8 && = \clfw_7 - \clfw_8,
\\ \clsr^{(3)} & = \clsr_2 + \clsr_3 + 2\clsr_4 + 2\clsr_5 + 2\clsr_6 + 2\clsr_7 + \clsr_8 && = \clfw_7 - \clfw_1,
\\ \clsr^{(4)} & = \clsr_7 + \clsr_8 && = \clfw_7 + \clfw_8 - \clfw_6,
\\ \clsr^{(5)} & = \clsr_8 && = 2 \clfw_8 - \clfw_7,
\end{align*}
and $\lambda = s \clfw_7 - \sum_{i=1}^5 k_i \clsr^{(i)}$, such that
\begin{enumerate}[(1)]
\item $k_1 + k_2 + k_3 + k_4 \leq s$,
\item $k_i \geq 0$ for $i = 1,2,3,4,5$, and
\item $k_2 \geq k_4 + 2k_5$.
\end{enumerate}
Moreover the highest weight rigged configurations in $\RC(B^{8,s})$ are given by
\[
\nu(k_1 \ast [\clsr^{(1)}] + k_2 \ast [\clsr^{(2)}] + k_3 \ast [\clsr^{(3)}] + k_4 \ast [\clsr^{(4)}] + k_5 \ast [\clsr^{(5)}])
\]
with all riggings $0$ except the first row of $\nu^{(8)}$, which satisfies $0 \leq x \leq k_2 - k_4 - 2k_5$, and
\[
\cc(\nu, J) = 3k_1 + k_2 + k_3 + k_4 + k _5 + x.
\]
\end{proposition}

\begin{proof}
Similar to the proof of Proposition~\ref{prop:RC_E6_4}.
\end{proof}

% --- r = 1 ---

\begin{proposition}
Consider the KR crystal $B^{1,s}$ of type $E_8^{(1)}$. We have
\[
\RC(B^{1,s}) = \bigoplus_{\substack{k_1,k_2 \geq 0 \\ k_1+k_2 \leq s}} \RC(B^{1,s}; (s - k_1 - k_2) \clfw_1 + k_2 \clfw_8).
\]
Moreover the highest weight rigged configurations in $\RC(B^{1,s})$ are given by
\[
\nu(k_1 \ast [\clsr^{(1)}] + k_2 \ast [\clsr^{(2)}]),
\]
where
\begin{align*}
\clsr^{(1)} & = 4\clsr_1 + 5\clsr_2 + 7\clsr_3 + 10\clsr_4 + 8\clsr_5 + 6\clsr_6 + 4\clsr_7 + 2\clsr_8 && = \clfw_1,
\\ \clsr^{(2)} & = 2\clsr_1 + 2\clsr_2 + 3\clsr_3 + 4\clsr_4 + 3\clsr_5 + 2\clsr_6 + \clsr_7 && = \clfw_1 - \clfw_8,
\end{align*}
with all riggings $0$ and
\[
\cc(\nu, J) = 2k_1 + k_2.
\]
\end{proposition}

\begin{proof}
Similar to the proof of Proposition~\ref{prop:RC_E6_4} as $\clsr^{(1)} > \clsr^{(2)}$.
\end{proof}

% --- r = 2 ---

%For $r = 2$, we have at least 44 different root edges (for $T(B^{2,7})$).

% --- r = 6 ---

%For $r = 6$, there are 118 different types of roots which appear on the edges of the Kleber tree in $T(B^{6,3})$.

We note that our graded decompositions for $r = 1,7,8$ agree with the conjectures of~\cite[App.~A]{HKOTY99}, where the case of $r = 7$ is similar to the proof of Proposition~\ref{prop:equivalence_E6_4}.

\begin{remark}
The cases of $r = 4$ of type $E_6^{(1)}$, $r = 3$ of type $E_7^{(1)}$, and $r = 7$ of type $E_8^{(1)}$ are all nodes which are distance $2$ from the affine node, and all have the same graded decompositions. Thus it would be interesting to see if there is a uniform describe of these KR crystals, and additionally to compare it with $B^{r,s}$ for $r = 1, 3$ in type $D_n^{(1)}$ and $r = 2, n-2$ in type $A_n^{(1)}$.
\end{remark}

%%%%%%%%%%
\subsection{Type $E_6^{(2)}$}

% --- r = 1 ---

For $r = 1$, this is immediate deduced from the devirtualization of Proposition~\ref{prop:RC_E6_2} ($r = 2$ in type $E_6^{(1)}$) as $\gamma_a = 1$ for all $a \in I$.

\begin{proposition}
\label{prop:RC_E6t_1}
Consider the KR crystal $B^{1,s}$ of type $E_6^{(2)}$. We have
\[
\RC(B^{1,s}) = \bigoplus_{k=0}^s \RC(B^{1,s}; (s-k) \clfw_1).
\]
Moreover, the highest weight rigged configurations in $\RC(B^{1,s})$ are given by $\nu(k\ast[\alpha])$ with all riggings $0$, where
\[
\alpha = 2\clsr_1 + 3\clsr_2 + 2\clsr_3 + \clsr_4 = \clfw_1,
\]
and
\[
\cc(\nu, J) = k.
\]
\end{proposition}

\begin{proof}
The proof is similar to Proposition~\ref{prop:RC_E6_2} as the ambient Kleber tree is the same as the virtual Kleber tree.
\end{proof}

\begin{definition}
\label{def:filling_E6t_1}
We define the filling map $\fillmap \colon B^{1,s} \to T^{1,s}$ on classically highest weight element $b \in B\bigl((s-k) \clfw_1\bigr) \subseteq B^{1,s}$ by defining $\fillmap(b)$ as the tableau with the first $s-k$ columns as $\column{1}$, the next $\lfloor k / 2 \rfloor$ columns as $\young(\bon1)$, and if $k$ is odd, the last column as $\column{\emptyset}$. We then extend $\fillmap$ as a classical crystal isomorphism.
\end{definition}

\begin{proposition}
\label{prop:filling_iso_E6t_1}
Let $\fillmap \colon B^{1,s} \to T^{1,s}$ given by Definition~\ref{def:filling_E6t_1} and $\iota$ be the natural (classical) crystal isomorphism. We have
\[
\Phi = \fillmap \circ \iota
\]
on classically highest weight elements.
\end{proposition}

\begin{proof}
Similar to the proof of Proposition~\ref{prop:filling_iso_E6_2}.
\end{proof}

% --- r = 2 ---

\begin{proposition}
\label{prop:RC_E6t_2}
Consider the KR crystal $B^{2,s}$ of type $E_6^{(2)}$. We have
\[
\RC(B^{2,s}) = \bigoplus_{\lambda} \RC(B^{2,s}; \lambda)^{\oplus (1 + k_2 - k_4 - 2k_5)},
\]
where
\begin{align*}
\clsr^{(1)} & := 3\clsr_1 + 6\clsr_2 + 4\clsr_3 + 2\clsr_4&& = \clfw_2,
\\ \clsr^{(2)} & := \clsr_1 + 3 \clsr_2 + 2 \clsr_3 + \clsr_4 && = \clfw_2 - \clfw_1,
\\ \clsr^{(3)} & := \clsr_1 + 2\clsr_2 + \clsr_3 && = \clfw_2 - \clfw_4,
\\ \clsr^{(4)} & := \clsr_1 + \clsr_2 && = \clfw_1 + \clfw_2 - \clfw_3,
\\ \clsr^{(5)} & := \clsr_1 && = 2\clfw_1 - \clfw_2,
\end{align*}
and $\lambda = s \clfw_2 - \sum_{i=1}^5 k_i \clsr^{(i)}$, such that
\begin{enumerate}[(1)]
\item $k_1 + k_2 + k_3 + k_4 \leq s$,
\item $k_i \geq 0$ for $i = 1,2,3,4,5$, and
\item $k_2 \geq k_4 + 2k_5$.
\end{enumerate}
Moreover the highest weight rigged configurations in $\RC(B^{2,s})$ are given by
\[
\nu(k_1 \ast [\clsr^{(1)}] + k_2 \ast [\clsr^{(2)}] + k_3 \ast [\clsr^{(3)}] + k_4 \ast [\clsr^{(4)}] + k_5 \ast [\clsr^{(5)}]),
\]
with all riggings $0$ except for $x$ in the first row of $\nu^{(1)}$, which satisfies $0 \leq x \leq k_2 - k_4 - 2k_5$, and
\[
\cc(\nu, J) = 3k_1 + k_2 + k_3 + k_4 + k _5 + x.
\]
\end{proposition}

\begin{proof}
Note that $\clsr^{(1)} > \clsr^{(2)} > \clsr^{(3)} > \clsr^{(4)} > \clsr^{(5)}$ component-wise expressed in terms of $\{ \clsr_i \}_{i \in I_0}$. Thus any sequence $(k_1, k_2, k_3, k_4, k_5)$ uniquely determines a path in the ambient Kleber tree, which is the same tree constructed in Proposition~\ref{prop:RC_E6_4}. Hence the claim follows from Definition~\ref{def:virtual_kleber} as all $\virtual{\clsr}^{(i)}$ are symmetric with respect to the diagram automorphism $\phi$ and that $\gamma_a = 1$ for all $a \in I$.
\end{proof}

% --- r = 4 ---

\begin{proposition}
\label{prop:RC_E62_4}
Consider the KR crystal $B^{4,s}$ of type $E_6^{(2)}$. We have
\[
\RC(B^{4,s}) = \bigoplus_{\substack{k_1,k_2 \geq 0 \\ k_1+k_2 \leq s}} \RC(B^{4,s}; k_2 \clfw_1 + (s - k_1 - k_2) \clfw_4).
\]
Moreover, the highest weight rigged configurations in $\RC(B^{4,s})$ are given by
\[
\nu(k_1 \ast [\clsr^{(1)}] + k_2 \ast [\clsr^{(2)}]),
\]
where
\begin{align*}
\clsr^{(1)} & = 2\clsr_1 + 4\clsr_2 + 3\clsr_3 + 2 \clsr_4 && = \clfw_4,
\\ \clsr^{(2)} & = \clsr_2 + \clsr_3 + \clsr_4 && = \clfw_4 - \clfw_1,
\end{align*}
with all riggings $0$ and
\[
\cc(\nu, J) = 2k_1 + k_2.
\]
\end{proposition}

\begin{proof}
We note that $\clsr^{(1)} > \clsr^{(2)}$ (and likewise for their ambient counterparts). The claim follows from Definition~\ref{def:virtual_kleber}.
\end{proof}

It is straightforward to see that our graded decompositions agree with those conjectured in~\cite[App.~A]{HKOTT02}, where the $r = 2$ case is similar to the proof of Proposition~\ref{prop:equivalence_E6_4}.

%%%%%%%%%%
\subsection{Type $F_4^{(1)}$}

We can also describe the set $\hwRC(B^{r,s})$ for $r = 1,2,4$ in type $F_4^{(1)}$.

% --- r = 1 ---
% --- r = 2 ---

For $r = 1,2$ in type $F_4^{(1)}$, this is the same as for type $E_6^{(2)}$ except for $\nu^{(3)}$ and $\nu^{(4)}$ are scaled by 2. This follows from Definition~\ref{def:virtual_kleber}, that $\gamma_{0,1,2} = 2$, and that $\gamma_{3,4} = 1$.

% --- r = 4 ---

\begin{proposition}
Consider the KR crystal $B^{4,s}$ of type $F_4^{(1)}$. We have
\[
\RC(B^{4,s}) = \bigoplus_{\substack{k_1,k_2 \geq 0 \\ k_1+k_2 \leq s/2}} \RC(B^{4,s}; k_2 \clfw_1 + (s - 2k_1 - 2k_2) \clfw_4).
\]
Moreover the highest weight rigged configurations in $\RC(B^{4,s})$ are given by
\begin{align*}
\nu^{(1)} & = (k_1, k_1)
\\ \nu^{(2)} & = (k_1+k_2, k_1, k_1, k_1),
\\ \nu^{(3)} & = (2k_1+2k_2, 2k_1, 2k_1),
\\ \nu^{(4)} & = (2k_1+2k_2, 2k_1),
\end{align*}
with all riggings $0$.
\end{proposition}

\begin{proof}
Similar to the proof of Proposition~\ref{prop:RC_E62_4} except for we select all nodes at even depths and the note above about devirtualization.
\end{proof}

Given the devirtualization map and that all nodes for $r = 1,2,4$ appear at even levels in the ambient Kleber tree, and Proposition~\ref{prop:equivalence_E6_4}, it is straightforward to see that our crystals agree with the conjectured decompositions of~\cite[App.~A]{HKOTY99}.

% --- r = 3 ---

%For $r = 3$, we first note that there are at least 50 edge labels can appear in the ambient tree, but they are bounded above by the following (which occur at depth 1):
%\begin{align*}
%\alpha^{(1)} & := 3\alpha_1 + 4\alpha_2 + 6\alpha_3 + 8\alpha_4 + 6\alpha_5 + 3\alpha_6 & = & \clfw_3 + \clfw_5,
%\\ \alpha^{(2)} & := 2\alpha_1 + 2\alpha_2 + 4\alpha_3 + 5\alpha_4 + 4\alpha_5 + 2\alpha_6 & = & \clfw_3 + \clfw_5 - \clfw_2,
%\\ \alpha^{(3)} & := \alpha_1 + 2\alpha_2 + 3\alpha_3 + 4\alpha_4 + 3\alpha_5 + \alpha_6 & = & \clfw_3 + \clfw_5 - \clfw_1 - \clfw_6,
%\\ \alpha^{(4)} & := \alpha_1 + \alpha_2 + 2\alpha_3 + 2\alpha_4 + 2\alpha_5 + \alpha_6 & = & \clfw_3 + \clfw_5 - \clfw_4,
%\\ \alpha^{(5)} & := \alpha_1 + 2\alpha_3 + 2\alpha_4 + 2\alpha_5 + \alpha_6 & = & \clfw_3 + \clfw_5 - 2\clfw_2,
%\\ \alpha^{(6)} & := \alpha_3 + \alpha_4 + \alpha_5 & = & \clfw_3 + \clfw_5 -  \clfw_1 - \clfw_2 - \clfw_6,
%\\ \alpha^{(2,1)} & := \alpha_2 + \alpha_3 + 2\alpha_4 + \alpha_5 & = & \clfw_4 - \clfw_1 - \clfw_6,
%\\ \alpha^{(3,1)} & := \alpha_3 + 2\alpha_4 + \alpha_5 & = & 2\clfw_4 - \clfw_1 - 2\clfw_2 - \clfw_6,
%\\ \alpha^{(3,2)} & := \alpha_3 + \alpha_5 & = & 2\clfw_3 + 2\clfw_5 - \clfw_1 - 2\clfw_4 - \clfw_6,
%\\ \alpha^{(3,3)} & := \alpha_4 & = & 2\clfw_4 - \clfw_2 - \clfw_3 - \clfw_5,
%\\ \alpha^{(4,1)} & := \alpha2 + 2\alpha_4 & = & 3 \clfw_4 - 2\clfw_3 - 2\clfw_5
%\\ \vdots
%\end{align*}

%=====================================================================
\section{Outlook}
\label{sec:outlook}

For untwisted non-simply-laced affine types, the author believes that there is a way to modify the description of the rigged configurations such that each partition is scaled by $1 / T_r$ as in Remark~\ref{remark:scaling_untwisted}. In particular, this modification will be similar to the description given here and the proof would be similar and uniform. We note that this was done for type $C_n^{(1)}$ and $B_n^{(1)}$ for $B^{1,1}$ in~\cite{OSS03}. However, we need to take $B^{1,2}$ for type $C_n^{(1)}$ and $B^{2,1}$ for type $B_n^{(1)}$ in order to get the adjoint node (and have a perfect crystal of level $1$). As such, a modification to our description would be needed. It seems plausible that the extra possibility of case~(Q), a singular row of length one less than previously selected, as given in $\delta_1$ of type $B_n^{(1)}$ from~\cite{OSS03} is the necessary modification. For type $G_2^{(1)}$, this is an open problem from~\cite{Scrimshaw15}, which we give here in more generality.

\begin{problem}
Describe explicitly the map $\delta_{\theta}$ for $\g$ of untwisted non-simply-laced affine type.
\end{problem}

%=====================================================================
\appendix

\section{KR tableaux for fundamental weights}
\label{sec:KR_tableaux}

{
\allowdisplaybreaks
In this section, we list the classically highest weight KR tableaux for $B^{r,1}$ for types $E_{6,7,8}^{(1)}$, $E_6^{(2)}$, $F_4^{(1)}$, and $G_2^{(1)}$.

Recall that highest weight rigged configurations must have $0 \leq p_i^{(a)}$ for all $(a,i) \in \HH_0$. Moreover, all rigged configurations in this section will have $p_i^{(a)} = 0$ except for possibly one $(b, j) \in \HH_0$ where $p_j^{(b)} = 1$. Therefore we describe the rigged configuration simply by its configuration $\nu$ and if $p_i^{(a)} = 1$, then we write the corresponding rigging $x$ as the subscript $(x)$.
We also write the column tableau
$
\begin{array}{|c|c|c|c|c|}
\hline
x_1 &
x_2 &
x_3 &
\cdots &
x_r
\\\hline
\end{array}^{\ t}
$
as $\column{x_1, x_2, x_3, \dotsc, x_r}$.

\begin{example}
In type $E_7^{(1)}$ for $B^{4,1}$, we denote the rigged configuration
\[
\begin{tikzpicture}[scale=.35,anchor=top,baseline=-18]
 \rpp{1}{0}{0}
 \begin{scope}[xshift=4cm]
 \rpp{1,1}{0,0}{0,0}
 \end{scope}
 \begin{scope}[xshift=8cm]
 \rpp{1,1}{1,0}{1,1}
 \end{scope}
 \begin{scope}[xshift=12cm]
 \rpp{1,1,1,1}{0,0,0,0}{0,0,0,0}
 \end{scope}
 \begin{scope}[xshift=16cm]
 \rpp{1,1,1}{0,0,0}{0,0,0}
 \end{scope}
 \begin{scope}[xshift=20cm]
 \rpp{1,1}{0,0}{0,0}
 \end{scope}
 \begin{scope}[xshift=24cm]
 \rpp{1}{0}{0}
 \end{scope}
\end{tikzpicture},
\]
%where we write the vacancy number of a row on the left and the rigging of a row on the right,
by $(1, 11, 1_{(1)} 1_{(0)}, 1111, 111, 11, 1)$ or more compactly $(1, 1^2, 1^2_{(1,0)}, 1^4, 1^3, 1^2, 1)$.
\end{example}

In the remaining part of this section, we give the KR tableaux for $B^{r,1}$ for the exceptional types (except for $r = 4, 5$ in type $E_8^{(1)}$, which can be generated using \textsc{SageMath}~\cite{sage}).

%%%%%%%%%%
\subsection{Type $E_6^{(1)}$}

\noindent {$r=1$}
\begin{align*}
(\emptyset, \emptyset, \emptyset, \emptyset, \emptyset, \emptyset) && \mapsto && \column{1}
\intertext{\noindent {$r=2$}}
%
%(\emptyset, \emptyset, \emptyset, \emptyset, \emptyset, \emptyset) && \mapsto && \column{1, \bon 3, 2 \bth}
%\\ (1, 1^2, 1^2, 1^3, 1^2, 1) && \mapsto && \column{1, \bon 6,  \bsix}
(\emptyset, \emptyset, \emptyset, \emptyset, \emptyset, \emptyset) && \mapsto && \column{1, \bon 6, 2 \bsix}
\\ (1, 1^2, 1^2, 1^3, 1^2, 1) && \mapsto && \column{1, \bon 6,  \bsix}
\intertext{\noindent {$r=3$}}
(\emptyset, \emptyset, \emptyset, \emptyset, \emptyset, \emptyset) && \mapsto && \column{1, \bon 3}
\\ (1, 1, 1^2, 1^2, 1, \emptyset) && \mapsto && \column{1, \bon 6}
\intertext{\noindent {$r=4$}}
(\emptyset, \emptyset, \emptyset, \emptyset, \emptyset, \emptyset) && \mapsto && \column{1, \bon 3, \bth 4}
\\ (\emptyset, 1, 1, 1^2, 1, \emptyset) && \mapsto && \column{1, \bon 3, 1 \bth 6}
\\ (1, 1_{(0)}, 1, 1^3, 1^2, 1) && \mapsto && \column{1, \bon 6, 2\bsix}
\\ (1, 1_{(1)}, 1^2, 1^3, 1^2, 1) && \mapsto && \column{1, \bon 3, 2 \bth}
\\ (1^2, 1^3, 1^4, 1^6, 1^4, 1^2) && \mapsto && \column{1, \bon 6, \bsix}
\intertext{\noindent {$r=5$}}
%
%(\emptyset, \emptyset, \emptyset, \emptyset, \emptyset, \emptyset) && \mapsto && \column{1, \bon 3, 2 \bth, \btw 5}
%\\ (\emptyset, 1, 1, 1^2, 1^2, 1) && \mapsto && \column{1, \bon 3, 2 \bth, 1 \btw}
(\emptyset, \emptyset, \emptyset, \emptyset, \emptyset, \emptyset) && \mapsto && \column{1, \bon 6, 2 \bsix, \btw 5}
\\ (\emptyset, 1, 1, 1^2, 1^2, 1) && \mapsto && \column{1, \bon 6, 2 \bsix, 1 \btw}
\intertext{\noindent {$r=6$}}
(\emptyset, \emptyset, \emptyset, \emptyset, \emptyset, \emptyset) && \mapsto && \column{1, \bon 6}
\end{align*}

%%%%%%%%%%
\subsection{Type $E_7^{(1)}$}

\noindent {$r=1$}
\begin{align*}
(\emptyset, \emptyset, \emptyset, \emptyset, \emptyset, \emptyset, \emptyset) && \mapsto && \column{7, 1 \bseven}
\\ (1^2, 1^2, 1^3, 1^4, 1^3, 1^2, 1) && \mapsto && \column{7, \bseven}
\intertext{\noindent {$r=2$}}
(\emptyset, \emptyset, \emptyset, \emptyset, \emptyset, \emptyset, \emptyset) && \mapsto && \column{7, 1 \bseven, \bon 2}
\\ (1, 1^2, 1^2, 1^3, 1^2, 1, \emptyset) && \mapsto && \column{7, 1 \bseven, \bon 7}
\intertext{\noindent {$r=3$}}
(\emptyset, \emptyset, \emptyset, \emptyset, \emptyset, \emptyset, \emptyset) && \mapsto && \column{7, 1 \bseven, \bon 2, \btw 3}
\\ (1, 1, 1^2, 1^2, 1, \emptyset, \emptyset) && \mapsto && \column{7, 1 \bseven, \bon 2, \btw 6}
\\ (1_{(0)}, 1^2, 1^3, 1^4, 1^3, 1^2, 1) && \mapsto && \column{7, 1 \bseven, \bon 7, 1 \bseven}
\\ (1_{(1)}, 1^2, 1^3, 1^4, 1^3, 1^2, 1) && \mapsto && \column{7, 1 \bseven,  \bon 2, 1 \btw}
\\ (1^3, 1^4, 1^6, 1^8, 1^6, 1^4, 1^2) && \mapsto && \column{7, 1 \bseven, \bon 7, \bseven}
\intertext{\noindent {$r=4$}}
(\emptyset, \emptyset, \emptyset, \emptyset, \emptyset, \emptyset, \emptyset) && \mapsto && \column{7, 6 \bseven, 5 \bsix, 4 \bfive}
\\ (\emptyset, 1, 1, 1^2, 1, \emptyset, \emptyset) && \mapsto && \column{7, 6 \bseven, 5 \bsix, 1 \bfive 6}
\\ (\emptyset, 1^2, 1^2, 1^4, 1^3, 1^2, 1) && \mapsto && \column{7, 6 \bseven, 1 \bsix 7, 1 \bseven}
\\ (1, 1_{(0)}, 1^2, 1^3, 1^2, 1, \emptyset) && \mapsto && \column{7, 6 \bseven, 1 \bsix 7, \bon 2}
\\ (1, 1_{(1)}, 1^2, 1^3, 1^2, 1, \emptyset) && \mapsto && \column{7, 6 \bseven, 5 \bsix, 2 \bfive 7}
\\ (1, 1^2, 1^2_{(0,0)}, 1^4, 1^3, 1^2, 1) && \mapsto && \column{7, 1 \bseven, \bon 2, \btw 3}
\\ (1, 1^2, 1^2_{(1,0)}, 1^4, 1^3, 1^2, 1) && \mapsto && \column{7, 6 \bseven, 1 \bsix 7, \bon 3 \bseven}
\\ (1, 1^2, 1^2_{(1,1)}, 1^4, 1^3, 1^2, 1) && \mapsto && \column{7, 6 \bseven, 5 \bsix, 3 \bfive}
\\ (1^2, 1^3, 1^4, 1^6, 1^4, 1^2, \emptyset) && \mapsto && \column{7, 6 \bseven, 1\bsix 7, \bon 7}
\\ (1^2, 1^3, 1^4, 1^6, 1^4, 1^2_{(0,0)}, 1) && \mapsto && \column{7, 1 \bseven, \bon 2, \btw 6}
\\ (1^2, 1^3, 1^4, 1^6, 1^4, 1^2_{(1,0)}, 1) && \mapsto && \column{7, 6 \bseven, 1 \bsix 7, \bon 6 \bseven}
\\ (1^2, 1^3, 1^4, 1^6, 1^4, 1^2_{(1,1)}, 1) && \mapsto && \column{7, 6 \bseven, 5 \bsix, \bfive 6}
\\ (1^2_{(0,0)}, 1^4, 1^5, 1^8, 1^6, 1^4, 1^2) && \mapsto && \column{7, 1 \bseven, \bon 7, 1 \bseven}
\\ (1^2_{(1,0)}, 1^4, 1^5, 1^8, 1^6, 1^4, 1^2) && \mapsto && \column{7, 1 \bseven, \bon 2, 1 \btw}
\\ (1^2_{(1,1)}, 1^4, 1^5, 1^8, 1^6, 1^4, 1^2) && \mapsto && \column{7, 6 \bseven, 1 \bsix 7, \bseven}
\\ (1^4, 1^6, 1^8, 1^{12}, 1^9, 1^6, 1^3) && \mapsto && \column{7, 1 \bseven, \bon 7, \bseven}
\\ (2, 21, 2^2, 2^2 1^2, 2 1^2, 1^2, 1) && \mapsto && \column{7, 6 \bseven, 2 \bsix, \btw 6}
\\ (1^2, 2 1^2, 2 1^3, 2^2 1^4, 2^2 1^2, 2^2, 2) && \mapsto && \column{7, 6 \bseven, 2 \bsix, \btw 1}
\\ (2^2, 2^2 1^2, 2^3 1^2, 2^4 1^4, 2^3 1^3, 2^2 1^2, 21) && \mapsto && \column{7, 6 \bseven, \bsix 7, \bseven}
\intertext{\noindent {$r=5$}}
(\emptyset, \emptyset, \emptyset, \emptyset, \emptyset, \emptyset, \emptyset) && \mapsto && \column{7, 6 \bseven, 5\bsix}
\\ (\emptyset, 1,  1, 1^2, 1^2, 1, \emptyset) && \mapsto && \column{7, 6 \bseven, 1 \bsix 7}
\\ (1, 1_{(0)}, 1^2, 1^3, 1^3, 1^2, 1) && \mapsto && \column{7, 1 \bseven, \bon 2}
\\ (1, 1_{(1)}, 1^2, 1^3, 1^3, 1^2, 1) && \mapsto && \column{7, 6 \bseven, 2\bsix}
\\ (1^2, 1^3, 1^4, 1^6, 1^5, 1^3, 1_{(0)}) && \mapsto && \column{7, 1 \bseven, \bon 7}
\\ (1^2, 1^3, 1^4, 1^6, 1^5, 1^3, 1_{(1)}) && \mapsto && \column{7, 6 \bseven, \bsix 7}
\intertext{\noindent {$r=6$}}
(\emptyset, \emptyset, \emptyset, \emptyset, \emptyset, \emptyset, \emptyset) && \mapsto && \column{7, 6 \bseven}
\\ (\emptyset, 1, 1, 1^2, 1^2, 1^2, 1) && \mapsto && \column{7, 1 \bseven}
\\ (1^2, 1^3, 1^4, 1^6, 1^5, 1^4, 1^2) && \mapsto && \column{7, \bseven}
\intertext{\noindent {$r=7$}}
(\emptyset, \emptyset, \emptyset, \emptyset, \emptyset, \emptyset, \emptyset) && \mapsto && \column{7}
\end{align*}

%%%%%%%%%%
\subsection{Type $E_8^{(1)}$}

\noindent {$r=1$}
\begin{align*}
(\emptyset, \emptyset, \emptyset, \emptyset, \emptyset, \emptyset, \emptyset, \emptyset) && \mapsto && \column{8, 1\beight}
\\ (1^2, 1^2, 1^3, 1^4, 1^3, 1^2, 1, \emptyset) && \mapsto && \column{8, \emptyset}
\\ (1^4, 1^5, 1^7, 1^{10}, 1^8, 1^6, 1^4, 1^2) && \mapsto && \column{\emptyset, \emptyset}
\intertext{\noindent {$r=2$}}
(\emptyset, \emptyset, \emptyset, \emptyset, \emptyset, \emptyset, \emptyset, \emptyset) && \mapsto && \column{8, 1\beight, \bon 2}
\\ (1, 1^2, 1^2, 1^3, 1^2, 1, \emptyset, \emptyset) && \mapsto && \column{8, 1\beight, \bon 7}
\\ (1_{(0)}, 1^3, 1^3, 1^5, 1^4, 1^3, 1^2, 1) && \mapsto && \column{8, 1\beight, \bon 1}
\\ (1_{(1)}, 1^3, 1^3, 1^5, 1^4, 1^3, 1^2, 1) && \mapsto && \column{8, 1\beight, \emptyset}
\\ (1^3, 1^5, 1^6, 1^9, 1^7, 1^5, 1^3, 1_{(0)}) && \mapsto && \column{8, \emptyset, 8\beight}
\\ (1^3, 1^5, 1^6, 1^9, 1^7, 1^5, 1^3, 1_{(1)}) && \mapsto && \column{8, \emptyset, \emptyset}
\\ (1^5, 1^8, 1^{10}, 1^{15}, 1^{12}, 1^9, 1^6, 1^3) && \mapsto && \column{\emptyset, \emptyset, \emptyset}
\intertext{\noindent {$r=3$}}
(\emptyset, \emptyset, \emptyset, \emptyset, \emptyset, \emptyset, \emptyset, \emptyset) && \mapsto && \column{8, 1\beight, \bon 2, \btw 3}
\\ (1, 1, 1^2, 1^2, 1, \emptyset, \emptyset, \emptyset) && \mapsto && \column{8, 1\beight, \bon 2, \btw 6}
\\ (1_{(0)}, 1^2, 1^3, 1^4, 1^3, 1^2, 1, \emptyset) && \mapsto && \column{8, 1\beight, \bon 7, 1\bseven 8}
\\ (1_{(1)}, 1^2, 1^3, 1^4, 1^3, 1^2, 1, \emptyset) && \mapsto && \column{8, 1\beight, \bon 2, 1\btw 8}
\\ (1^2, 1^2_{(0,0)}, 1^4, 1^5, 1^4, 1^3, 1^2, 1) && \mapsto && \column{8, 1\beight, \bon 7, 2 \bseven}
\\ (1^2, 1^2_{(1,0)}, 1^4, 1^5, 1^4, 1^3, 1^2, 1) && \mapsto && \column{8, 1\beight, \bon 2, \btw 2}
\\ (1^2, 1^2_{(1,1)}, 1^4, 1^5, 1^4, 1^3, 1^2, 1) && \mapsto && \column{8, 1\beight, \bon 2, \emptyset}
\\ (1^3, 1^4, 1^6, 1^8, 1^6, 1^4, 1^2, \emptyset) && \mapsto && \column{8, 1\beight, \bon 7, \bseven 8 8}
\\ (1^3, 1^4, 1^6, 1^8, 1^6, 1^4, 1^2_{(0,0)}, 1) && \mapsto && \column{8, 1\beight, \bon 1, \bon 7}
\\ (1^3, 1^4, 1^6, 1^8, 1^6, 1^4, 1^2_{(1,0)}, 1) && \mapsto && \column{8, 1\beight, \bon 7, \bseven 7}
\\ (1^3, 1^4, 1^6, 1^8, 1^6, 1^4, 1^2_{(1,1)}, 1) && \mapsto && \column{8, 1\beight, \bon 7, \emptyset}
\\ (1^3_{(0,0,0)}, 1^5, 1^7, 1^{10}, 1^8, 1^6, 1^4, 1^2) && \mapsto && \column{8, \emptyset, \beight 8, 1 \beight}
\\ (1^3_{(1,0,0)}, 1^5, 1^7, 1^{10}, 1^8, 1^6, 1^4, 1^2) && \mapsto && \column{8, 1\beight, \bon 1, \bon 1}
\\ (1^3_{(1,1,0)}, 1^5, 1^7, 1^{10}, 1^8, 1^6, 1^4, 1^2) && \mapsto && \column{8, 1\beight, \bon 1, \emptyset}
\\ (1^3_{(1,1,1)}, 1^5, 1^7, 1^{10}, 1^8, 1^6, 1^4, 1^2) && \mapsto && \column{8, 1\beight, \emptyset, \emptyset}
\\ (1^5, 1^7, 1^{10}, 1^{14}, 1^{11}, 1^8, 1^5, 1^2_{(0,0)}) && \mapsto && \column{8, \emptyset, \beight 8, \beight 8}
\\ (1^5, 1^7, 1^{10}, 1^{14}, 1^{11}, 1^8, 1^5, 1^2_{(1,0)}) && \mapsto && \column{8, \emptyset, \beight 8, \emptyset}
\\ (1^5, 1^7, 1^{10}, 1^{14}, 1^{11}, 1^8, 1^5, 1^2_{(1,1)}) && \mapsto && \column{8, \emptyset, \emptyset, \emptyset}
\\ (1^7, 1^{10}, 1^{14}, 1^{20}, 1^{16}, 1^{12}, 1^8, 1^4) && \mapsto && \column{\emptyset, \emptyset, \emptyset, \emptyset}
\\ (21, 2^2, 2^2 1^2, 2^3 1^2, 2^2 1^2, 21^2, 1^2, 1) && \mapsto && \column{8, 1\beight, \bon 2, \btw 7}
\\ (1^3, 21^3, 21^5, 2^2 1^6, 2^2 1^4, 2^2 1^2, 2^2, 2) && \mapsto && \column{8, 1\beight, \bon 7, 1 \bseven}
\\ (2^2 1_{(0)}, 2^2 1^3, 2^3 1^4, 2^4 1^6, 2^3 1^5, 2^2 1^4, 2 1^3, 1^2) && \mapsto && \column{8, 1\beight, \bon 1, \bon 8}
\\ (2^2 1_{(1)}, 2^2 1^3, 2^3 1^4, 2^4 1^6, 2^3 1^5, 2^2 1^4, 2 1^3, 1^2) && \mapsto && \column{8, 1\beight, \emptyset, \bon 8}
\\ (2^2 1^3, 2^3 1^4, 2^4 1^6, 2^6 1^8, 2^5 1^6, 2^4 1^4, 2^3 1^2, 2^2) && \mapsto && \column{8, \emptyset, \emptyset, \beight}
\intertext{\noindent {$r=6$}}
(\emptyset, \emptyset, \emptyset, \emptyset, \emptyset, \emptyset, \emptyset, \emptyset) && \mapsto && \column{8, 7\beight, 6\bseven}
\\ (\emptyset, 1, 1, 1^2, 1^2, 1^1, 1, \emptyset) && \mapsto && \column{8, 7\beight, 1\bseven 8}
\\ (1, 1_{(0)}, 1^2, 1^3, 1^3, 1^3, 1^2, 1) && \mapsto && \column{8, 1\beight, \bon 2}
\\ (1, 1_{(1)}, 1^2, 1^3, 1^3, 1^3, 1^2, 1) && \mapsto && \column{8, 7\beight, 2 \bseven}
\\ (1^2, 1^3, 1^4, 1^6, 1^5, 1^4, 1^2, \emptyset) && \mapsto && \column{8, 7\beight, \bseven 8 8}
\\ (1^2, 1^3, 1^4, 1^6, 1^5, 1^4, 1^2_{(0,0)}, 1) && \mapsto && \column{8, 1\beight, \bon 7}
\\ (1^2, 1^3, 1^4, 1^6, 1^5, 1^4, 1^2_{(1,0)}, 1) && \mapsto && \column{8, 7\beight, \bseven 7}
\\ (1^2, 1^3, 1^4, 1^6, 1^5, 1^4, 1^2_{(1,1)}, 1) && \mapsto && \column{8, 7\beight, \emptyset}
\\ (1^2_{(0,0)}, 1^4, 1^5, 1^8, 1^7, 1^6, 1^4, 1^2) && \mapsto && \column{8, \beight 8, 1 \beight}
\\ (1^2_{(1,0)}, 1^4, 1^5, 1^8, 1^7, 1^6, 1^4, 1^2) && \mapsto && \column{8, 1 \beight, \bon 1}
\\ (1^2_{(1,1)}, 1^4, 1^5, 1^8, 1^7, 1^6, 1^4, 1^2) && \mapsto && \column{8, 1\beight 8, \emptyset}
\\ (1^4, 1^6, 1^8, 1^{12}, 1^{10}, 1^8, 1^5, 1^2_{(0,0)}) && \mapsto && \column{8, \beight 8, \beight 8}
\\ (1^4, 1^6, 1^8, 1^{12}, 1^{10}, 1^8, 1^5, 1^2_{(1,0)}) && \mapsto && \column{8, \beight 8, \emptyset}
\\ (1^4, 1^6, 1^8, 1^{12}, 1^{10}, 1^8, 1^5, 1^2_{(1,1)}) && \mapsto && \column{8, \emptyset, \emptyset}
\\ (1^6, 1^9, 1^{12}, 1^{18}, 1^{15}, 1^{12}, 1^8, 1^4) && \mapsto && \column{\emptyset, \emptyset, \emptyset}
\\ (1^2, 21^2, 21^3, 2^2 1^4, 2^2 1^3, 2^2 1^2, 2^2, 2) && \mapsto && \column{8, 7\beight, 1 \bseven}
\\ (2^2, 2^2 1^2, 2^3 1^2, 2^4 1^4, 2^3 1^4, 2^2 1^4, 21^3, 1^2) && \mapsto && \column{8, 1\beight, \bon 8}
\\ (2^2 1^2, 2^3 1^3, 2^4 1^4, 2^6 1^6, 2^5 1^5, 2^4 1^4, 2^3 1^2, 2^2) && \mapsto && \column{8, \emptyset, \beight}
\intertext{\noindent {$r=7$}}
(\emptyset, \emptyset, \emptyset, \emptyset, \emptyset, \emptyset, \emptyset, \emptyset) && \mapsto && \column{8, 7\beight}
\\ (\emptyset, 1, 1, 1^2, 1^2, 1^2, 1^2, 1) && \mapsto && \column{8, 1\beight}
\\ (1^2, 1^3, 1^4, 1^6, 1^5, 1^4, 1^3, 1_{(0)}) && \mapsto && \column{8, 8\beight}
\\ (1^2, 1^3, 1^4, 1^6, 1^5, 1^4, 1^3, 1_{(1)}) && \mapsto && \column{8, \emptyset}
\\ (1^4, 1^6, 1^8, 1^{12}, 1^{10}, 1^8, 1^6, 1^3) && \mapsto && \column{\emptyset, \emptyset}
\intertext{\noindent {$r=8$}}
(\emptyset, \emptyset, \emptyset, \emptyset, \emptyset, \emptyset, \emptyset, \emptyset) && \mapsto && \column{8}
\\ (1^2, 1^3, 1^4, 1^6, 1^5, 1^4, 1^3, 1^2) && \mapsto && \column{\emptyset}
\end{align*}

%%%%%%%%%%
\subsection{Type $E_6^{(2)}$}

\noindent {$r=1$}

\begin{align*}
(\emptyset, \emptyset, \emptyset, \emptyset) && \mapsto && \column{1}
\\ (1^2, 1^3, 1^2, 1) && \mapsto && \column{\emptyset}
\intertext{\noindent {$r=2$}}
(\emptyset, \emptyset, \emptyset, \emptyset) && \mapsto && \column{1, \bon 2}
\\ (1, 1^2, 1, \emptyset) && \mapsto && \column{1, \bon 4}
\\ (1_{(0)}, 1^3, 1^2, 1) && \mapsto && \column{1, \bon 1}
\\ (1_{(1)}, 1^3, 1^2, 1) && \mapsto && \column{1, \emptyset}
\\ (1^3, 1^6, 1^4, 1^2) && \mapsto && \column{\emptyset, \emptyset}
\intertext{\noindent {$r=3$}}
(\emptyset, \emptyset, \emptyset, \emptyset) && \mapsto && \column{1, \bon 2, \btw 3}
\\ (\emptyset, 1, 1, \emptyset) && \mapsto && \column{1, \bon 2, 1 \btw 4}
\\ (\emptyset, 1^2, 1^2, 1) && \mapsto && \column{1, \bon 2, 11\btw}
\\ (1, 1^2_{(0,0)}, 1^2, 1) && \mapsto && \column{1, \bon 4, 2\bfo}
\\ (1, 1^2_{(1,0)}, 1^2, 1) && \mapsto && \column{1, \bon 2, \btw 2}
\\ (1, 1^2_{(1,1)}, 1^2, 1) && \mapsto && \column{1, \bon 2, \emptyset}
\\ (1^2, 1^4, 1^3, 1_{(0)}) && \mapsto && \column{1, \emptyset, \bon 4}
\\ (1^2, 1^4, 1^3, 1_{(1)}) && \mapsto && \column{1, \bon 4,\emptyset}
\\ (1^2_{(0,0)}, 1^5, 1^4, 1^2) && \mapsto && \column{1, \bon 1, \bon 1}
\\ (1^2_{(1,0)}, 1^5, 1^4, 1^2) && \mapsto && \column{1, \bon 1, \emptyset}
\\ (1^2_{(1,1)}, 1^5, 1^4, 1^2) && \mapsto && \column{1, \emptyset,\emptyset}
\\ (1^4, 1^8, 1^6, 1^3) && \mapsto && \column{\emptyset, \emptyset, \emptyset}
\\ (2, 2^2, 21, 1) && \mapsto && \column{1, \bon 2, \btw 4}
\\ (1^2, 2 1^3, 2 1^2, 2) && \mapsto && \column{1, \bon 4, 1\bfo}
\\ (2^2, 2^3 1^2, 2^2 1^2, 21) && \mapsto && \column{1, \emptyset, \bon}
\intertext{\noindent {$r=4$}}
(\emptyset, \emptyset, \emptyset, \emptyset) && \mapsto && \column{1, \bon 4}
\\ (\emptyset, 1, 1, 1) && \mapsto && \column{1, \emptyset}
\\ (1^2, 1^4, 1^3, 1^2) && \mapsto && \column{\emptyset, \emptyset}
\end{align*}

%%%%%%%%%%
\subsection{Type $F_4^{(1)}$}

We follow Proposition~\ref{prop:adjoint_elements} to describe the elements of $B(\clfw_4)$.

\noindent {$r=1$}

\begin{align*}
(\emptyset, \emptyset, \emptyset, \emptyset) && \mapsto && \column{4, 1\bfo}
\\ (1^2, 1^3, 2^2, 2) && \mapsto && \column{4, \bfo}
\intertext{\noindent {$r=4$}}
(\emptyset, \emptyset, \emptyset, \emptyset) && \mapsto && \column{4, 3\bfo, 2\bth}
\\ (1, 1^2, 2, \emptyset) && \mapsto && \column{4, 3\bfo, \bth44}
\\ (1_{(0)}, 1^3, 2^2, 2) && \mapsto && \column{4, 4\bfo, 1\bfo}
\\ (1_{(1)}, 1^3, 2^2, 2) && \mapsto && \column{4, 3\bfo, 1\bth}
\\ (1^3, 1^6, 2^4, 2^2) && \mapsto && \column{4, 4\bfo, \bfo}
\intertext{\noindent {$r=3$}}
(\emptyset, \emptyset, \emptyset, \emptyset) && \mapsto && \column{4, 3\bfo}
\\ (1, 1^2, 21, 1) && \mapsto && \column{4, 4\bfo}
\intertext{\noindent {$r=4$}}
(\emptyset, \emptyset, \emptyset, \emptyset) && \mapsto && \column{4}
\end{align*}

%%%%%%%%%%
\subsection{Type $G_2^{(1)}$}

\noindent {$r=1$}

\begin{align*}
(\emptyset, \emptyset) && \mapsto && \column{1}
\intertext{\noindent {$r=2$}}
(\emptyset, \emptyset) && \mapsto && \column{1,2}
\\ (3, 1^2) && \mapsto && \column{1, \bon} %[\btw, \bon] \otimes [1, 2]
\end{align*}
}

%=====================================================================
\section{Examples with SageMath}

We give some examples using \textsc{SageMath}~\cite{sage}, where rigged configurations, KR tableaux, and the bijection $\Phi$ has been implemented by the author.

We first construct Example~\ref{ex:E6t_B11_power}.
\begin{lstlisting}
sage: RC=RiggedConfigurations(['E',6,2], [[1,1]]*4)
sage: n=RC(partition_list=[[2,2,1,1],[2,2,2,1,1,1],[2,2,1,1],[2,1]],
....:        rigging_list=[[1,0,2,1],[0,0,0,0,0,0],[0,0,0,0],[0,0]])
sage: ascii_art(n)
1[ ][ ]1  0[ ][ ]0  0[ ][ ]0  0[ ][ ]0
1[ ][ ]0  0[ ][ ]0  0[ ][ ]0  0[ ]0
2[ ]2     0[ ][ ]0  0[ ]0   
2[ ]1     0[ ]0     0[ ]0   
          0[ ]0             
          0[ ]0             
sage: n.to_tensor_product_of_kirillov_reshetikhin_tableaux().pp()
 (1, -2) (X)  (-1, 2) (X)   E (X)  (1,)
\end{lstlisting}
Next, we construct Example~\ref{ex:RC_E7_B41}.
\begin{lstlisting}
sage: RC=RiggedConfigurations(['E',7,1], [[4,1]])
sage: nu=RC.module_generators[6]
sage: ascii_art(nu)

0[ ]0  0[ ]0  1[ ]1  0[ ]0  0[ ]0  0[ ]0  0[ ]0
       0[ ]0  1[ ]0  0[ ]0  0[ ]0  0[ ]0
                     0[ ]0  0[ ]0       
                     0[ ]0              
sage: nu.to_tensor_product_of_kirillov_reshetikhin_tableaux().pp()
        (7,)
     (-7, 6)
  (-6, 7, 1)
 (-1, -7, 3)
\end{lstlisting}

%=====================================================================
%\section{Proofs}
%\label{sec:proofs}

%\bibliographystyle{plain}
\bibliographystyle{alpha}
\bibliography{biject}{}
\end{document}